%% file: main.tex
\documentclass[article,hidelinks,onefignum,onetabnum]{siamart220329}

\input{ex_shared}

\usepackage{booktabs}
\usepackage{subfigure}
\usepackage{amssymb}
\usepackage{empheq}

\ifpdf
\hypersetup{
  pdftitle={3D FEM for nonlocal problem with ball approximation},
 pdfauthor={G. Chen, Y. Ma, and J. Zhang}
}
\fi

\usepackage{xr}
\begin{document}

\maketitle

\begin{abstract}
Nonlocality brings many challenges to the implementation of finite element methods (FEM) for nonlocal problems, such as a large number of neighborhood query operations being invoked on the meshes. {Besides}, the interactions are usually limited to Euclidean balls, so direct numerical integrals often introduce numerical errors. {The issues of interactions between the ball and finite elements have to be carefully dealt with}, such as using ball approximation strategies.
In this paper, an efficient representation and construction methods for approximate balls are presented based on the combinatorial map, and an efficient parallel algorithm is also designed for the assembly of nonlocal linear systems. Specifically, a new ball approximation method based on Monte Carlo integrals, i.e., the fullcaps method, is also proposed to compute numerical integrals over the intersection region of an element with the ball.
\end{abstract}

\begin{keywords}
{Nonlocal problem}, finite element method, combinatorial map, approximate ball, Monte Carlo integration, parallel computing
\end{keywords}

\begin{MSCcodes}
65Y10; 65D30; 37M99; 34K28; 34A45.
\end{MSCcodes}

\section{Introduction}
The nonlocal operators have been applied in various fields \cite{peridynamics:du,phase:bates,image:gilboa}. Because of the wide application of nonlocal operators, many numerical algorithms have been developed for solving nonlocal problems effectively, including finite difference method \cite{difference:du2019,difference:tian2013}, finite element method \cite{element:du2019conforming,d2021cookbook} and collocation method \cite{collection:tian2013efficient,collection:zhang2018accurate}. Among these methods, the advantages of precision and stability arising from the finite element method (FEM) \cite{GL10,DG13,DG21} are worth its application in solving nonlocal problems. But, nonlocality brings some new difficulties in FEM implementations, especially for three-dimensional (3D) case.




A comprehensive description of the computational challenges that arise in the implementation of FEM {for nonlocal problems} can be found in \cite{d2021cookbook}. For nonlocal problems, the integration region is complex, using classical quadrature rules directly to compute the integrals may introduce additional errors. 
For the case of {fixing} interaction horizon, the authors of \cite{d2021cookbook} have proposed a {new} method to avoid {this issue} by introducing the concept of {approximate balls}, 
but some ball approximation strategies are difficult to implement in 3D. 
In \cite{PS22}, Pasetto et al. compute the inner integration by using quadrature points distributed over the full ball. 
In \cite{AC21}, the authors propose a technique that allows direct computation of the inner integral over the element directly by smoothing the kernel function. It is pointed out that the smoothed kernel
allows the use of classical quadrature rules over each element without using ball approximation strategies, which is easier to implement. 

A large number of works have implemented finite element solutions for nonlocal problems while using uniform meshes or quasi-uniform meshes in 1D, 2D \cite{approx,DW15,WT12,WT14} and 3D cases \cite{VS19}, and developed many fast stiffness matrix assembly and solution algorithms based on uniform meshes \cite{DW15,LC18,TE19,WW17}. Unlike local problems, the solutions of nonlocal problems require a large number of element calls and queries to the mesh, which is difficult to implement when using unstructured meshes. 
In \cite{d2021cookbook} and \cite{Yin}, the ball approximation strategies are introduced and analyzed to deal with integrals over the interaction domain of interaction ball and element more precisely. But this again increases the difficulty {when approximating a ball} on unstructured meshes. 
Fortunately, a new data structure called combinatorial maps \cite{cmap:2007,cmap:2010} is particularly good at handling operations on meshes, including queries and modifying unstructured meshes dynamically, and has been applied in the field of computer graphics \cite{DL14}. By reviewing the characteristics of this data structure, we believe that this data structure is very suitable for describing the unstructured mesh {when} solving nonlocal problems in any dimension. 

{As far as we know, this is the first effort that discusses the implementation issues of FEM for solving nonlocal problems in high dimensions using combinatorial map in detail.} In this paper, we {study the numerical implementation issues} for solving nD($n \geq 3$) nonlocal {Poisson} problems, including efficient neighborhood queries, ball approximation strategies and the fast matrix assembly needed by nonlocal problems' {solution}. In Section \ref{sec:background}, the definitions and notations of nonlocal problems and the weak form of the nonlocal Poisson problems are reviewed. In Section \ref{sec:discretization}, we 
discuss in detail the definitions of the ball approximation strategies in nD. {The estimates of geometric errors of these ball approximation strategies in nD are also presented}. In Section \ref{sec:implementation}, we discuss the implementation of FEM for nonlocal {problems} based on combinatorial map data structure. We then design some algorithms for constructing the nonlocal approximate ball, such as topological iterators developed based on the combinatorial map. Subsequently, we present a detailed assembly procedure to compute numerical solutions for the nonlocal model. 
{Finally,} in Section \ref{sec:numerical}, the 3D numerical result {shows} the effectiveness and accuracy of our implementation.

\section{Background and notations}
\label{sec:background}

In this section, we introduce the mathematical definitions and results from previous studies that will be used throughout the paper, along with their corresponding notations. In particular, the weak form of the nonlocal Poisson problem with Dirichlet boundary condition is formulated in detail, and the finite element discretization for nonlocal problems {delivered in next section} is based on this weak form.

\subsection{Setting of nonlocal problem}
We consider the nonlocal effect with finite interaction horizon, i.e. define a kernel $\gamma(x,y):\mathbb{R}^n \times \mathbb{R}^n \rightarrow \mathbb{R}$ as a nonnegative and symmetric function for every fixed $x$, the support of $\gamma(x,\cdot)$ is assumed to be in a bounded Euclidean ball $B_\delta(x)$ centered at $x$ with the interaction radius $\delta>0$ \cite{peridynamics:du}.

The kernel can be written as
\begin{equation}\label{Dgamma}
\gamma(x,y) = \psi(x,y) \mathcal{X}_{B_\delta(x)}(y),
\end{equation}
where $\mathcal{X}_{B_\delta(x)}(y)$ is an indicator function such that the ball $B_{\delta}(x)$ is the support of $\gamma(x, \cdot)$, and $\psi(x,y):\mathbb{R}^n \times \mathbb{R}^n \rightarrow \mathbb{R}$ is a symmetric and positive function denoted as the kernel function.


Without loss of generality, we always assume in this paper that the kernel is square integrable, and translation-invariant, namely
\begin{equation}
\left\{
\begin{array}{ll}
&\int_{B_\delta(x)} \gamma(x,y)^2 dy <\infty,\\
&\psi(x,y)=\psi(x+a,y+a), \quad
 \quad\forall a\in\mathbb{R}^n,
\end{array}
\right.
\label{ass:kernel} \end{equation}
The results presented in this paper can be easily {generalized} to the case of non-symmetric kernels \cite{approx} and some sign-changing kernels \cite{MD13}. The nonlocal operator $\mathcal{L}$ associated with $\gamma(x,y)$ is defined as
\begin{equation}\label{def:nonlocal_operator}
\begin{aligned}
\mathcal{L}u(x) := 2\int_{B_{\delta}(x)} (u(y) - u(x)) \psi(x,y)dy,\quad\forall\ x\in\mathbb{R}^n.
\end{aligned}
\end{equation}

Let $\Omega \subset \mathbb{R}^n$ be a bounded and open domain. We define a set $\Omega_\mathcal{I}$ that contains those points in the domain $\mathbb{R}^n \setminus \Omega$ that interact with points in $\Omega$ through the kernel  $\gamma$. The set $\Omega_\mathcal{I}$ is called by the interaction domain corresponding to $\Omega$ and $\gamma$, and can be defined mathematically as
\begin{equation}\label{eq:omegai}
\Omega_\mathcal{I} = \{ y \in \mathbb{R}^n\setminus\Omega: \exists\ x\in\Omega \text{ such that } |x-y| \leq \delta  \}.
\end{equation}
We denote $\hat{\Omega}:=\Omega\cup\Omega_{\mathcal{I}}$ in the remainder of the paper.


We can now present the nonlocal problem considered in this paper. For a bounded and open domain $\Omega \subset \mathbb{R}^n$, the nonlocal Poisson problem is defined as:
\begin{equation}
\left\{
             \begin{array}{llr}
             -\mathcal{L}u(x)=f(x) & \text{for } x \in \Omega,  \\
             u(x)=g(x) & \text{for } x \in \Omega_\mathcal{I}.
             \end{array}
\right.
\label{eq:nonlocalproblem}
\end{equation}
With the given source item $f:\Omega \rightarrow \mathbb{R}$ and the $g:\Omega_\mathcal{I} \rightarrow \mathbb{R}$, the problem needs to determine $u: \Omega \rightarrow \mathbb{R}$. The second equation in \eqref{eq:nonlocalproblem} is called the nonlocal Dirichlet volume constraint. {We only consider Dirichlet boundary conditions in this paper, and the implementation in this paper can be applied to problems with Neumann boundary conditions considered in \cite{neumann:osti_1769929,neumann:diffusion} naturally.}

\subsection{Weak Formulation}

By applying the nonlocal Green's first identity \cite{D2019}, the weak form of nonlocal problem \eqref{eq:nonlocalproblem} is given as
\begin{equation}\label{eq:weak0}
    \int_{\Omega \cup \Omega_{\mathcal{I}}} \int_{\Omega \cup \Omega_{\mathcal{I}}} (u(y) - u(x))(v(y)- v(x))\gamma(x,y)dydx = \int_{\Omega}v(x)f(x)dx,
\end{equation}
where test function $v(x)$ is any smooth function satisfying $v(x)=0$ for $x \in \Omega_{\mathcal{I}}$.


{We encounter here a double integral $\int (\int dy) dx$ in the weak form of the nonlocal problem.} For ease of illustration, $\int dy$ is denoted as the inner integral and $\int dx$ is denoted as the outer integral. {According to \cite{d2021cookbook}, by using the Dirichlet volume constraint in \eqref{eq:nonlocalproblem} and the symmetry of kernel $\gamma(x,y)$, the weak form \eqref{eq:weak0} derives another weak form:}
\begin{equation}\label{eq:weak1}
    \begin{split}
    &\int_{\Omega} \int_{\Omega} \big(u(y)-u(x)\big)\big(v(y)-v(x)\big)\gamma(x,y)dydx\\
&+ 2 \int_{\Omega}u(x)v(x) \int_{\Omega_{\mathcal{I}}} \gamma(x,y)dydx \\
    & =  2 \int_{\Omega}v(x) \int_{\Omega_{\mathcal{I}}} g(y)\gamma(x,y)dydx
+ \int_{\Omega}v(x)f(x)dx.
    \end{split}
\end{equation}
Thus,  equation \eqref{eq:weak1} can be rewritten as
\begin{equation}
    A(u,v) = F(v),    \label{eq:weak}
\end{equation}
where the left hand side of \eqref{eq:weak} is a symmetric bilinear form
\begin{equation}
    \begin{split}
        A(u,v) := &\int_{\Omega} \int_{\Omega} \big(u(y)-u(x)\big)\big(v(y)-v(x)\big)\gamma(x,y)dydx \\
        &+ 2 \int_{\Omega} u(x)v(x)\big(\int_{\Omega_{\mathcal{I}}} \gamma(x,y)dy\big)dx.
    \end{split}
    \label{eq:left}
\end{equation}
The right hand side of \eqref{eq:weak} is a linear form
\begin{equation}
    F(v) := \int_{\Omega} v(x)\big(f(x)+2\int_{\Omega_{\mathcal{I}}}g(y)\gamma(x,y)dy\big)dx.
    \label{eq:right}
\end{equation}

\section{Finite Element Discretization and Error Estimate}
\label{sec:discretization}

We now consider the finite element discretization of weak formulation \eqref{eq:weak} defined on a {triangulation}. The theoretical results of estimation for geometric errors in general $n$-dimensional case are established. At the end of this section, we also discuss the choice {of} ball approximation strategies and quadrature rules in nD. {Without loss of generality, we restrict ourselves to general continuous, piecewise linear Lagrange polynomial basis.}

\subsection{Finite Element Grids}

Let $\mathcal{T}^{h}_{\Omega}$ denotes an $n$-dimensional triangulation (cell-decomposition) dividing $\Omega$ into $K_{\Omega}$ finite elements $\{\mathcal{E}_k\}_{k=1}^{K_\Omega}$ \cite{BS94}, where each finite element of $\{\mathcal{E}_k\}_{k=1}^{K_\Omega}$ is an $n$-dimensional simplex, and $h$ is the maximum distance between two adjacent vertices.
{However}, it is generally impossible to exactly triangulate $\Omega_\mathcal{I}$ defined in \eqref{eq:omegai} into simplex elements, {because nonlocality will create rounded corners to $\Omega_\mathcal{I}$}. One way to solve this problem is to triangulate another polytope domain that approximates $\Omega_\mathcal{I}$ into simplex elements. Another way is introduced in \cite{d2021cookbook} by replacing rounded corners with vertices. For either method, we still denote the new domain as $\Omega_\mathcal{I}$, and $\mathcal{T}^{h}_{\Omega_\mathcal{I}}$ as a triangulation of $\Omega_\mathcal{I}$ into $K_{\Omega_\mathcal{I}}$ finite elements. The elements on $\Omega_{\mathcal{I}}$ are denoted as $\{\mathcal{E}_k\}_{k=K_\Omega+1}^{K_\Omega+K_{\Omega_\mathcal{I}}}$.

We require that the subdivisions of $\Omega$ and $\Omega_\mathcal{I}$ must coincide. This is, the cells in $\Omega$ and $\Omega_\mathcal{I}$ do not straddle across the internal boundary $\partial \Omega$ \cite{d2021cookbook}. This means $\mathcal{T}^h_{\Omega\cup\Omega_\mathcal{I}}:=\mathcal{T}^h_{\Omega}\cup\mathcal{T}^h_{\Omega_\mathcal{I}}$ is a triangulation of $\Omega\cup\Omega_\mathcal{I}$ into $K=K_{\Omega_\mathcal{I}}+K_{\Omega}$ finite elements. In this paper, we always assume $h<\delta/2$.

\begin{remark}
The case of a programming implementation of the $2$-dimensional finite element when $h$ is set large enough compared to $\delta$ is also discussed in detail in \cite{d2021cookbook}. In practical applications, the mesh size $h$ and the horizon parameter $\delta$ satisfy a proportional natural condition $h=\mathcal{O}(\delta)$ \cite{TD14,TD20}. For example, the parameter configuration such as $3h\approx\delta$ is preferred in the $1$-dimensional and $2$-dimensional nonlocal problems \cite{BH12,PL08}. 
Since $h<\delta/2$ is used more frequently in practical programming implementations and it is easier for us to build our theory about ball approximation strategies, {this assumption is acceptable}.
\end{remark}

\subsection{Finite Element Space and the Discretization of the Weak Formulations}

Let $\{\widetilde x_j\}_{j=1}^J $ denote the set of nodes associated to triangulation $\mathcal{T}^h_{\Omega\cup\Omega_\mathcal{I}}$, where nodes $\{\widetilde{x}_j \}_{j=1}^{J_\Omega} $ are located in the open domain $\Omega$ and the nodes $\{\widetilde x_j\}_{j=J_\Omega+1}^{J}$ are located in {the} closed domain ${\Omega}_{\mathcal{I}}$. This means that the nodes located on $\partial \Omega = {\overline{\Omega} \cap {\Omega}_{\mathcal{I}}}$ are assigned to ${\Omega}_{\mathcal{I}}$. Then, for $j = 1, \cdots, J $, let $\phi_j (x) $ denote a continuous piecewise-{linear} function such that $\phi_j(\widetilde{x}_{j'}) = \delta_{jj'}$ for $j'=1, \cdots, J $, where $\delta_{jj'}$ denotes the Kronecker delta function. We then define the finite element spaces by

$$
V^h = \mbox{span} \{ \phi_j(x) \}_{j=1}^J \subset V(\Omega\cup\Omega_\mathcal{I})
\quad\mbox{and}\quad
V^h_c = \mbox{span} \{ \phi_j(x) \}_{j=1}^{J_\Omega}  \subset V_c(\Omega\cup\Omega_\mathcal{I}).
$$

According to our definition, functions belonging to $V^h$ and $V^h_c$ are continuous on $\Omega\cup\Omega_\mathcal{I}$.

The finite element approximation $u_h \in V^h$ can be written as the linear representation of basis functions. The volume constraint is applied at the nodes in $\Omega_\mathcal{I}$, including the nodes located on the boundary $\partial \Omega$, to set $U_j = g(\widetilde{x}_j)$ for $\widetilde{x}_j \in \mathcal{T}_{\Omega_\mathcal{I}}$. So we have
\begin{equation}
    u_h(x) =
  \sum_{j=1}^{J} U_j \phi_j(x) = \sum_{j=1}^{J_\Omega} U_j \phi_j(x) + \sum_{j=J_\Omega + 1}^J g(\widetilde x_j) \phi_j(x) \in V^h.
    \label{eq:u}
\end{equation}


Substituting \eqref{eq:u} into \eqref{eq:weak1} and choosing $v(x)\in V_c^h$, we have a linear system as
\begin{equation}
    \sum_{j=1}^{J_\Omega} A(\phi_{j},\phi_{i})U_j = F(\phi_{i})  -  \sum_{j=J_\Omega+1}^{J} A(\phi_{j},\phi_{i})g(\widetilde x_j)\quad  \mbox{for $i=1,\ldots,J_\Omega$},
    \label{eq:linear}
\end{equation}
where
\begin{equation}
    \begin{split}\label{eq:A}
        A(\phi_j,\phi_i) =& \sum_{\mathcal{E}_k \in \mathcal{T}^h_{\Omega}} \int_{\mathcal{E}_k} \int_{\Omega\cap B_{\delta}(x)} (\phi_j(y) - \phi_j(x)) (\phi_i(y) - \phi_i(x) )\psi(x,y)dydx\\
        &+ 2 \sum_{\mathcal{E}_k \in \mathcal{T}^h_{\Omega}} \int_{\mathcal{E}_k} \phi_j(x) \phi_i(x)(\int_{\Omega_\mathcal{I}\cap B_{\delta}(x)}\psi(x,y) dy)dx  .
    \end{split}
\end{equation}
The components of the $J_{\Omega}$-dimensional right-hand side vector are given by
\begin{equation}
    \begin{split}\label{eq:F}
    F(\phi_{i}) = \sum_{\mathcal{E}_k \in \mathcal{T}^h_{\Omega}}\int_{\mathcal{E}_k}\phi_{i}(x) \Big(f(x) +
2\int_{\Omega_{\mathcal I}\cap B_\delta(x)} g(y)\psi(x,y)dy
  \Big) \,dx.
    \end{split}
\end{equation}
Notice {that the support of $\phi_i$ is a subset of $\Omega$}, so for $\phi_j$ that corresponding to $\widetilde x_{j>J_\Omega}\notin\partial \Omega$, we have $A(\phi_j,\phi_i)=0$. So the linear system in \eqref{eq:linear} can be simplified into
\begin{equation}
    \sum_{j=1}^{J_\Omega} A(\phi_{j},\phi_{i})U_j = \widetilde{\widetilde{F}}(\phi_{i}) \quad  \mbox{for $i=1,\ldots,J_\Omega$},
    \label{eq:linear1.5}
\end{equation}
and the $J_{\Omega}$-dimensional right-hand side vector are now given by
\begin{equation}
    \begin{split}\label{eq:F1.5}
    \widetilde{\widetilde{F}}(\phi_{i}) = F(\phi_{i})-\sum_{j\in\{J_\Omega+1:J|\widetilde x_{j}\in\partial \Omega\}} A(\phi_{j},\phi_{i})g(\widetilde x_j).
    \end{split}
\end{equation}

We say $u_h$ by \eqref{eq:linear1.5} is our numerical solution without using ball approximation strategies, i.e., every inner integral over the $B_\delta(x)$ can be computed accurately. All the numerical experiments in Section \ref{sec:numerical} will use this discretization form \eqref{eq:linear1.5}.

\subsection{Error Estimate and Balls Approximation}
If the exact solution of nonlocal Poisson problem \eqref{eq:nonlocalproblem} is sufficiently smooth, {the convergence order of the numerical solution is of $\mathcal{O}(h^2)$ by} the result in \cite{d2021cookbook, D2019}. 
%
However this is hard to be guaranteed in practical implementations. In the assembly process, nonlocal problems usually encounter integrals of discontinuous functions over some elements \cite{d2021cookbook}. This may lead to the failure of the quadrature rule, and introduce additional errors to the linear system to be solved. For example, when the element satisfies $\mathcal{E}_i\cap\partial B_\delta(x)\neq\emptyset$, the inner integration $\sum_{i}\int_{\mathcal{E}_i\cap B_\delta(x)}g(y)\psi(x,y)dy$ cannot be computed with an error of $\mathcal{O}(h^2)$ by using the classical quadrature rules that have been performed well in solving local Poisson problems. Hence it may result in a loss of the convergence order to $\|u-u_h\|_{L^2(\Omega)}$. 
In \cite{d2021cookbook}, D'Elia et al. introduce a series of 2D ball approximation strategies, and show that some ball approximation strategies can maintain the optimal convergence order $\mathcal{O}(h^2)$ without seriously raising the computational cost in the assembly process. 
However, for the case of 3D or higher dimensions, the ball approximation strategies and related theoretical analysis have not been discussed. 
We will give some error estimates of the ball approximation strategies in this subsection, and {leave} the discussion about the algorithm and implementation of the ball approximation strategies in Section \ref{sec:implementation}.

Intuitively, the computation of the inner integration over $B_\delta(x)$ can be approximated by an integral over a polyhedral region $B_{\delta,h}(x)$ that approximates $B_\delta(x)$. The polyhedral region can be divided into several simplices, and the integrand on each simplex is continuous, so the integrations over these simplices can be computed by using the classical quadrature rules directly. In fact, we solve
\begin{equation}
    \hat{u}_h(x) =
  \sum_{j=1}^{J} \hat{U}_j \phi_j(x) = \sum_{j=1}^{J_\Omega} \hat{U}_j \phi_j(x) + \sum_{j=J_\Omega + 1}^J g(\widetilde x_j) \phi_j(x)
    \label{eq:hatu}
\end{equation}
by dealing with a {modification} of the weak formulation \eqref{eq:linear1.5}, {i.e.}:
\begin{equation}
    \sum_{j=1}^{J_\Omega} A_h(\phi_{j},\phi_{i})\hat{U}_j =\widetilde{\widetilde{F}}_h(\phi_{i}) := F_h(\phi_{i}) - \sum_{j=J_\Omega+1}^{J} A_h(\phi_{j},\phi_{i})g(\widetilde x_j)\quad \mbox{for $i=1,\ldots,J_\Omega$},
    \label{eq:modified_linear}
\end{equation}
where
\begin{equation}
\begin{aligned}\label{eq:Ah}
A_h(u,v) =& \int_{\Omega}\int_{\Omega\cap B_{\delta,h}(x)}
            (u(y)-u(x))(v(y)-v(x))\psi(x,y)\,dy\,dx
           \\& + 2\int_{\Omega}u(x)v(x)\bigg(\int_{\Omega_{\mathcal I}\cap B_{\delta,h}(x)}\psi(x,y)dy\bigg) \,dx
            \quad\forall\, u\in V^h,\,v\in V_c^h
\end{aligned}
\end{equation}
and
\begin{equation}\label{eq:Fh}
F_h(v) = \int_{\Omega}v(x) \bigg(f(x) +
2\int_{\Omega_{\mathcal I}\cap B_{\delta,h}(x)} g(y)\psi(x,y)dy
  \bigg) \,dx \quad\forall\, v\in V_c^h.
\end{equation}

By such polyhedral domain approximation, 
the numerical integration over each simplex can be computed by using classical quadrature rules without loss of accuracy. 
In \cite{d2021cookbook}, the ball approximation {strategies about} how to choose an appropriate $B_{\delta,h}(x)$ in 2D has been discussed in detail, including nocaps, barycenter, overlaps, approxcaps, exactcaps and shifted-center.


In \cite{d2021cookbook}, D'Elia et al. show that the geometric error, i.e. the {estimate of $L^2$-error} between $u_h$ and $\hat{u}_h$, can be bounded by the error of approximating the ball with a polytope. Du et al. \cite{Yin} point out that this error is also determined by the properties of the kernel function $\psi(x,y)$ defined on the $B_\delta(x)$. 
Before presenting the proposition, we denote $\Delta B_{\delta,h}(x)=(B_\delta(x)\setminus B_{\delta,h}(x))\cup(B_{\delta,h}(x)\setminus B_\delta(x))$,  $\forall\,x\in\Omega$ and $B'_{\delta,h}(x):=\big\{y\in\Omega|x\in B_{\delta,h}(y)\big\}$, $\Delta B'_{\delta,h}(x)=(B_\delta(x)\setminus B'_{\delta,h}(x))\cup(B'_{\delta,h}(x)\setminus B_\delta(x))$, $\forall\,x\in\Omega\cup\Omega_\mathcal{I}$,

\begin{proposition}[\cite{d2021cookbook}]\label{prop:error1}
Let $B_\delta(x)$ denote the $\ell^2$-ball in $n$D and $B_{\delta,h}(x)$ {is} its approximation, and let $u_h$ and $\hat{u}_h$ denote the corresponding finite element solutions obtained from \eqref{eq:u} and \eqref{eq:hatu}, respectively. {Assume the kernel function $\psi(x,y)$ satisfies \eqref{ass:kernel} and is integrable for all $y\in\Delta B_{\delta,h}(x)$ and $y\in\Delta B'_{\delta,h}(x)$. If all inner and outer integrals in \eqref{eq:linear1.5} and \eqref{eq:modified_linear} are exactly evaluated.} Then,
\begin{equation}\label{eq:energy-diff1}
\|u_h-\hat{u}_h \|_{L^2(\hat{\Omega})}\leq K \;\Big(\sup_{x\in\Omega}\,\big(
     \int_{\hat{\Omega}\cap\Delta B_{\delta,h}(x)} \psi(x,y)\,dy\big)+\sup_{x\in\hat{\Omega}}\,\big(
     \int_{\Omega\cap\Delta B'_{\delta,h}(x)} \psi(x,y)\,dy\big)\Big),
\end{equation}
where $K$ is a positive constant that depends on $\|f\|_{L^2(\Omega)}$ and $\|g\|_{L^2(\Omega_{\mathcal{I}})}$ but is independent of $\delta$ and $h$.
\end{proposition}

If we further assume $\psi(x,y)$ is a smooth kernel function satisfies $\psi(x,y):=C(n)\frac{\psi_0(|x-y|/\delta)}{\delta^{2+n}}$ like in \cite{d2021cookbook,Yin}, where $C(n)\in\mathbb{R}^+$ is a constant depending on $n$ to make $\int_{B_\delta(x)}\gamma(x,y)\cdot|x-y|^2 dy = n$, then the Proposition \ref{prop:error1} can be written in a simpler form as:
\begin{equation}\label{eq:energy-diff2}
\|u_h-\hat{u}_h \|_{L^2(\hat{\Omega})}\leq \frac{K}{\delta^{2+n}} \;\Big(\sup_{x\in\Omega}\,|\hat{\Omega}\cap\Delta B_{\delta,h}(x)|+\sup_{x\in\hat{\Omega}}\,|\Omega\cap\Delta B'_{\delta,h}(x)|\Big).
\end{equation}

The definition of shifted-center strategy discussed in \cite{d2021cookbook} can be {generalized} to $n$-dimension directly. This shifted-center strategy can be paired with any of the ball approximation strategies as in \cite{d2021cookbook}. For other strategies, the definition of intersection and barycenter approximation strategy can be generalized to higher dimensions quite naturally. However, the ball approximation strategies that use an inscribed polyhedral domain to approximate $B_\delta(x)$, such as the nocaps strategy and the approxcaps strategy, seem unable to be directly generalized from their 2-dimensional definition. 
In higher dimensions, putting all intersection cases between an element (simplex) and a ball into consideration is both theoretically and programmatically cumbersome, which brings great difficulties to the finite element implementation of {ball approximation strategies}.

In particular, we note that the ball approximation strategies that {generate an} inscribed polytope of $B_\delta(x)$ all employ an inscribed polytope $B_{\delta,h}(x)$ with edges' length $\mathcal{O}(h)$. We will later prove theoretically that this ``inscribed polytope approximation'' will not affect the convergence order of the solution when we use linear bases. 

In \cite{d2021cookbook}, it has been shown that the order of geometric error arising from the nocaps strategy is of $\mathcal{O}(h^2)$ in 2D by applying the proposition \ref{prop:error1}, and the barycenter strategy is of $\mathcal{O}(h^\alpha)$ for $\alpha\in[1,2]$. For higher dimensions, we can obtain similar estimates of the geometric error based on proposition \ref{prop:error1}, {but the key point is to estimate the error of approximating the ball with a polytope:}
\begin{theorem}\label{thm:normalapprox}
Let $B_\delta(x)$ denote the $n$-dimensional $\ell^2$-ball and $B_{\delta,h}(x)$ be  its  approximation. Assume $B_{\delta-h}(x)\subset B_{\delta,h}(x)\subset B_{\delta+h}(x)$ holds for all $x\in\Omega$, then:
$$
\sup_{x\in\Omega}\,|\hat{\Omega}\cap\Delta B_{\delta,h}(x)|+\sup_{x\in\hat{\Omega}}\,|\Omega\cap\Delta B'_{\delta,h}(x)|\leq\mathcal{O}(h\delta^{n-1}).
$$
\end{theorem}
\begin{proof}
It is easy to check that $\big(\Omega\cap B_{\delta-h}(x)\big)\subset B'_{\delta,h}(x)\subset \big(\Omega\cap B_{\delta+h}(x)\big)$ is also satisfied according to definition of $B'_{\delta,h}(x)$. As a result, $\Delta B_{\delta,h}(x)\subset B_{\delta+h}(x)\setminus B_{\delta-h}(x)$ for $\forall x\in\Omega$, and $\Omega\cap\Delta B'_{\delta,h}(x)\subset \Omega\cap\big(B_{\delta+h}(x)\setminus B_{\delta-h}(x)\big)$ for $\forall x\in\hat{\Omega}$. This can directly derive
$$|\Delta B_{\delta,h}(x)|+|\Omega\cap\Delta B'_{\delta,h}(x)|\leq2|B_{\delta+h}(x)|-2|B_{\delta-h}(x)|=\mathcal{O}(h\delta^{n-1}).$$
The proof is completed. \end{proof}

Theorem \ref{thm:normalapprox} also holds for barycenter, overlap, shifted-center and all the other polynomial approximation strategies considered in \cite{d2021cookbook} and this article, because these strategies all satisfy $\partial B_{\delta,h}(x)\subset\bigcup_{\mathcal{E}_i\cap\partial B_{\delta}(x)\neq\varnothing}\mathcal{E}_i$, and the edge of each simplex is $\mathcal{O}(h)$, which derives $B_{\delta-h}(x)\subset B_{\delta,h}(x)\subset B_{\delta+h}(x)$. The estimate of the order of geometric error for some of these strategies may not be optimal. Besides, for inscribed polyhedral approximations such as nocaps strategy and approxcaps strategy, we have a higher-order estimate:

\begin{theorem}\label{thm:inscribedapprox}
Let $B_\delta(x)$ denote the $n$-dimensional $\ell^2$-ball and $B_{\delta,h}(x)$ be an inscribed polyhedral approximation of $B_\delta(x)$. Assume the maximum edge length of $B_{\delta,h}(x)$ is less than $h$, and every face of this polytope is a $(n-1)$-dimensional simplex, then:
$$
\sup_{x\in\Omega}\,|\hat{\Omega}\cap\Delta B_{\delta,h}(x)|+\sup_{x\in\hat{\Omega}}\,|\Omega\cap\Delta B'_{\delta,h}(x)|\leq\mathcal{O}(h^2\delta^{n-2}) .
$$
\end{theorem}
\begin{proof}
Following the proof of theorem \ref{thm:normalapprox}, it is easy to find that we only need to prove $\delta-c\leq|y-x|\leq\delta$ holds for all $y\in \partial B_{\delta,h}(x)$, where $c=\mathcal{O}(h^2\delta^{-1})$.

The $|y-x|\leq\delta$ is definitely satisfied, as we consider inscribed polyhedral approximation here, and we only need to prove the left inequality. In fact, because every face of this polytope is a $(n-1)$-dimensional simplex, we state that if $y\in \partial B_{\delta,h}(x)$, then $y$ must be in the convex hull of some of this polynomial vertices $\{v_i\}_{i\in\mathcal{I}}$, which satisfies $|v_i-v_j|\leq h,\,\forall i,j\in\mathcal{I}$, because they are all in the same $(n-1)$-dimensional simplex. Without loss of generality, we assume $x$ is the original point. Thus the proof of this theorem is equivalent to the following Lemma \ref{lem:nocaps}.
\end{proof}
\begin{lemma} \label{lem:nocaps}
If $|v_i|=\delta,\,\forall i\in\mathcal{I}$, {and} $|v_i-v_j|\leq h,\,\forall i,j\in\mathcal{I}$. Then, for any point in the convex hull of $\{v_i\}_{i\in\mathcal{I}}$ (i.e. $y=\sum_{i\in\mathcal{I}}\alpha_iv_i,\,\sum_{i\in\mathcal{I}}\alpha_i=1$),  we have
$$
|y|>\delta-\mathcal{O}(h^2\delta^{-1}).
$$
\end{lemma}
\begin{proof}
By the definition of $y$, we have
\begin{equation*}
|y|^2=y^Ty=(\sum_{i\in\mathcal{I}}\alpha_iv_i)^T(\sum_{i\in\mathcal{I}}\alpha_iv_i) =\sum_{i\in\mathcal{I}}\alpha_i^2\delta^2+\sum_{i\neq j\in\mathcal{I}}\alpha_i\alpha_jv_i^Tv_j.
\end{equation*}
Noticing $|v_i-v_j|\leq h$, we have
\begin{equation*}
v_i^Tv_j =-\frac{(v_i-v_j)^T(v_i-v_j)}{2}+\delta^2
\geq\delta^2-\frac{h^2}{2}.
\end{equation*}
Hence, we further have
\begin{equation}
\begin{aligned}
|y|^2&\geq\sum_{i\in\mathcal{I}}\alpha_i^2\delta^2+\sum_{i\neq j\in\mathcal{I}}\alpha_i\alpha_j(\delta^2-\frac{h^2}{2})>\delta^2-\frac{h^2}{2}.
\end{aligned}
\end{equation}
Finally, Taylor's expansion shows that
$$
|y|>\sqrt{\delta^2-\frac{h^2}{2}}=\delta\sqrt{1-\frac{h^2}{2\delta^2}}\geq\delta - \frac{1}{4}h^2\delta^{-1}.
$$
The proof is completed.
\end{proof}

{With the increase of dimension, the implementation difficulty and computational cost of ``polyhedral approximation'' of nocaps and approxcaps strategies increases faster than barycenter strategy. The nocaps strategy we use in this paper is simplified from the nocaps strategy defined in \cite{d2021cookbook}, but it is much more suitable for implementation in 3D and performs well in numerical experiments, as we will see in Section \ref{sec:numerical}. First} we present a strategy that can split the simplex satisfying $\mathcal{E}_k\cap B_\delta\neq\emptyset$ into smaller simplex so that we can build a polyhedral approximation of domain $\mathcal{E}_k\cap B_\delta$. 

\begin{definition}[simplex's dividing strategy]\label{def:simplexdiv}
For an $n$-{dimensional} simplex $\mathcal{E}_k$ with $n+1$ vertices $\{v_i\}_{i=1}^{n+1}$ and edge length less that $h$, and an exact $n$-{dimensional} open $\ell^2$-ball $B_\delta(x)$ with radius $\delta$, and center at $x$. We can subdivide $\mathcal{E}_k$ by $\partial B_\delta(x)$ in the following ways:
\begin{itemize}
\item If all vertices are in $B_\delta(x)$, we say $\mathcal{E}_k$ is {\bf inside} the $B_\delta(x)$. Note this means $\mathcal{E}_k\subset B_\delta(x)$;
\item If all vertices are not in $B_\delta(x)$, we say $\mathcal{E}_k$ is {\bf outside} the $B_\delta(x)$. Note this doesn't mean we have $\mathcal{E}_k\cap B_\delta(x)=\varnothing$, since $\mathbb{R}^n-B_\delta(x)$ is not a convex domain and $\mathcal{E}_k$ is a convex domain;
\item If there are $m$ vertices inside the $B_\delta(x)$, and $n+1-m$ vertices outside the $B_\delta(x)$, we subdivide the simplex in the following ways. For simplicity, we assume $V_p:=\{v_i\}_{i=1}^{m}\in B_\delta(x)$ and $V_q:=\{v_i\}_{i=m+1}^{n+1}\in\mathbb{R}^n\setminus B_\delta(x)$. Then for those $m\times(n+1-m)$ different line segments decided by $V_p\times V_q$, each of them  has and only has one intersection point with $\partial B_\delta(x)$, and we {denote} them as $\{p_i\}_{i=1}^{m(n+1-m)}$. Then, the convex hull of $\{p_i\}_{i=1}^{m(n+1-m)}\cup\{v_i\}_{i=1}^{m}$, which we denote as $\mathcal{E}_k^0$, satisfies $\mathcal{E}_k^0\subset B_\delta(x)\cap\mathcal{E}_k$. Obviously $\mathcal{E}_k^0$ is a polytope that can be divided into smaller $n$-dimensional simplices $\{\mathcal{E}_j^*\}_{j\in\mathcal{J}^x_k}$. 
\end{itemize}
\end{definition}

This kind of simplex dividing strategy only considers the relationship between simplex and ball through vertices and is already quite complex for implementation for dimensions higher than $3$. Regardless of its complexity in programming, we are now able to get an ``inscribed polytope approximation" of $B_\delta(x)$ by Definition \ref{def:simplexdiv}. It is time to present our nocaps ball approximation strategy as follows:

\begin{definition}[nocaps strategy]\label{def:nocaps}
Given a set of $n$D simplices $\{\mathcal{E}_k\}_{k\in\mathcal{K}}$ with $n+1$ vertices $\{v_i\}_{i=1}^{n+1}$ and edge length less that $h$, and an $n$-dimensional open $\ell^2$-ball $B_\delta(x)$ with radius $\delta$ and center at $x$. For each simplex $\mathcal{E}_k$ satisfing $\mathcal{E}_k\cap \partial B_\delta(x)\neq\varnothing$, we are able to get $\{\mathcal{E}_j^*\}_{j\in\mathcal{J}^x_k}$ by following the dividing strategy in Definition \ref{def:simplexdiv}. Now we can approximate $B_\delta(x)$ in the following way
$$
B_{\delta,h}(x):=\bigcup_{\mathcal{E}_k\subset B_\delta(x)}\mathcal{E}_k+\bigcup_{\mathcal{E}_k\cap \partial B_\delta(x)\neq \varnothing}{\bigcup_{j\in\mathcal{J}^x_k}\mathcal{E}_j^*}.
$$
\end{definition}
The $B_{\delta,h}(x)$ defined in nocaps strategy is convex in 2D, but usually not convex in higher dimensions, see the {examples of } illustrations in Figures \ref{fig:2dinscribe} and \ref{fig:3dinscribe}.

\begin{figure}[htp]
    \centering
    \subfigure[2D Barycenter]{
    \includegraphics[scale=0.15]{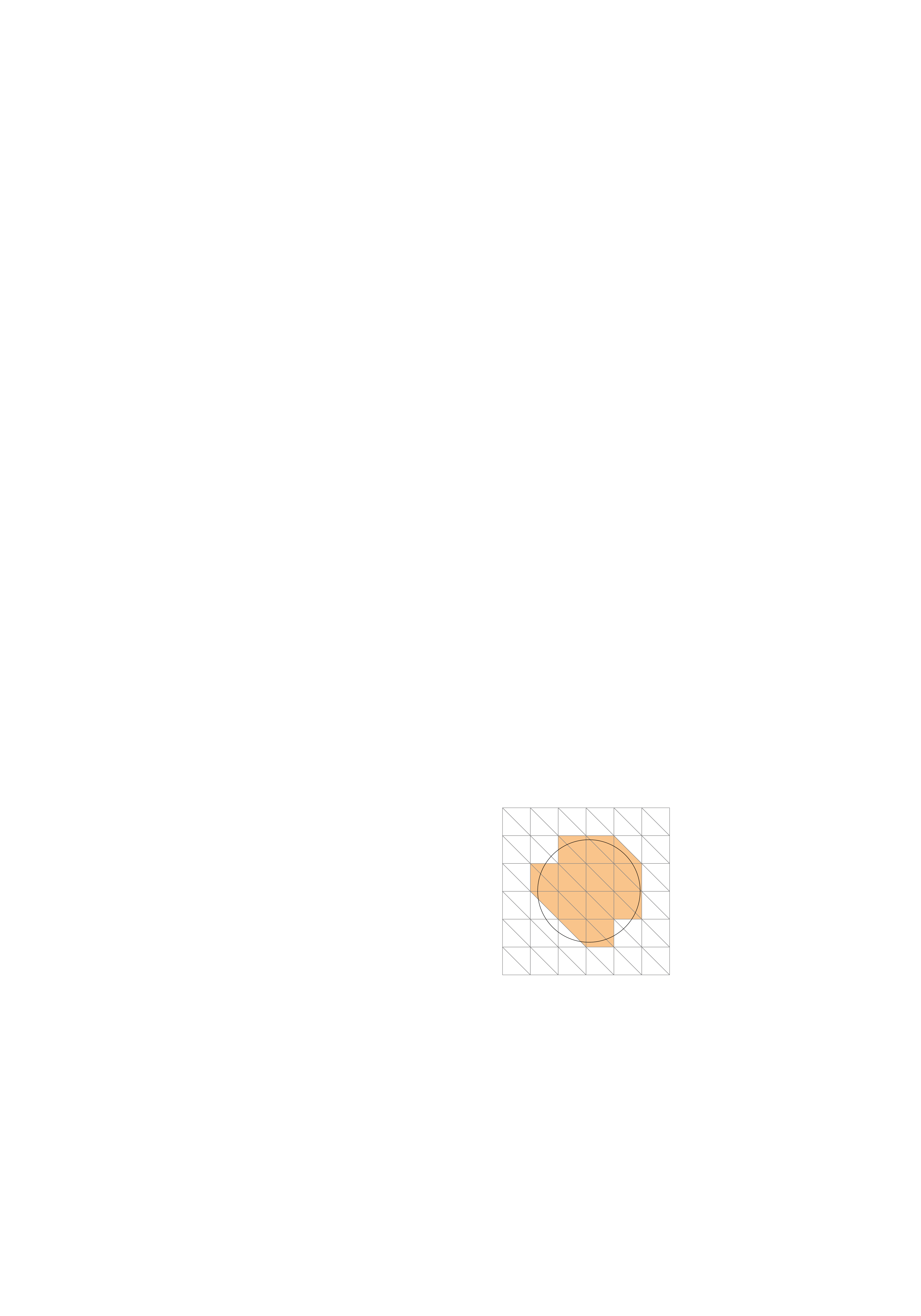}
    \label{fig:2dbary}
    }
    \subfigure[2D overlap]{
    \includegraphics[scale=0.15]{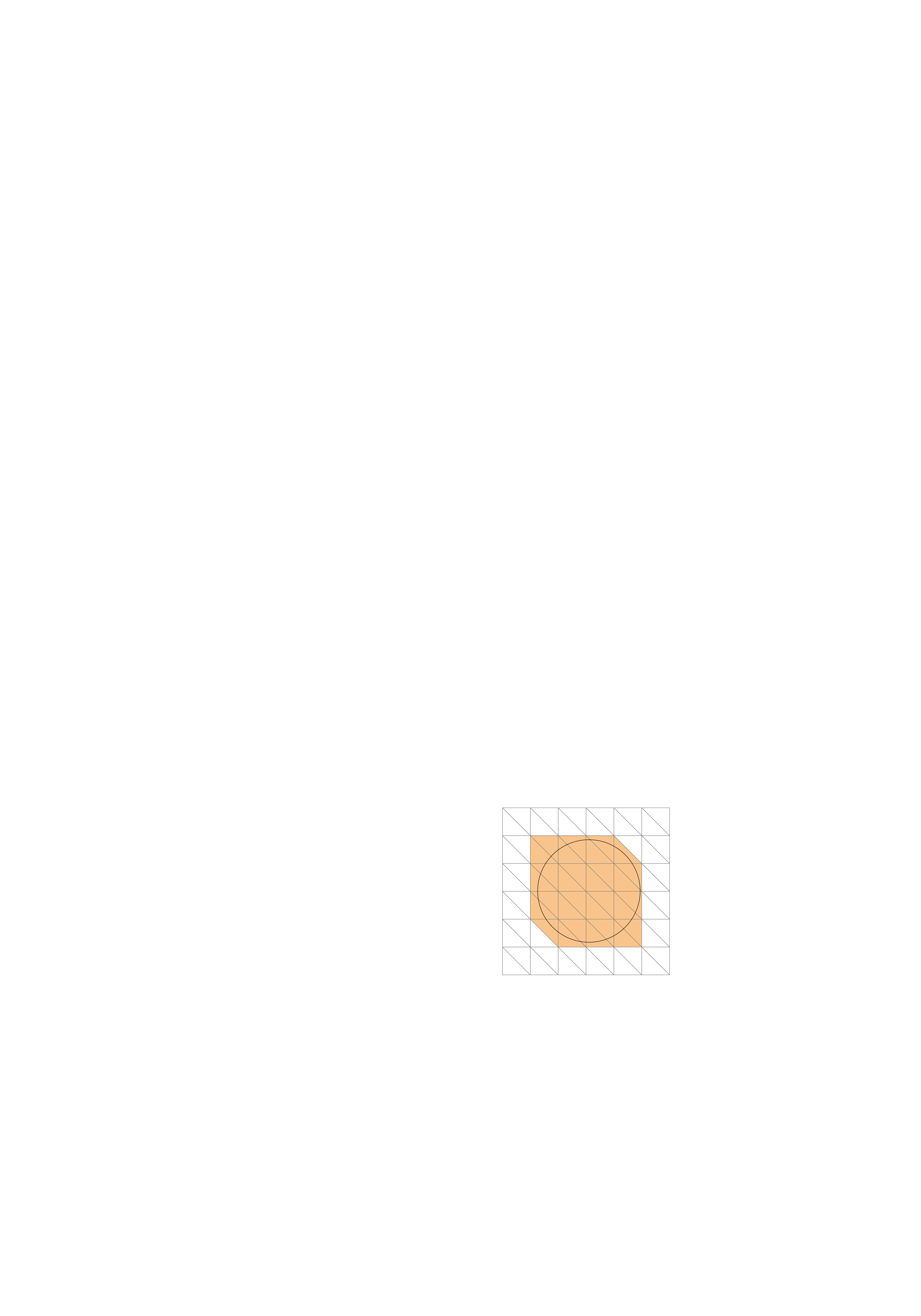}
    \label{fig:2doverlap}
    }
    \subfigure[2D nocaps]{
    \includegraphics[scale=0.3]{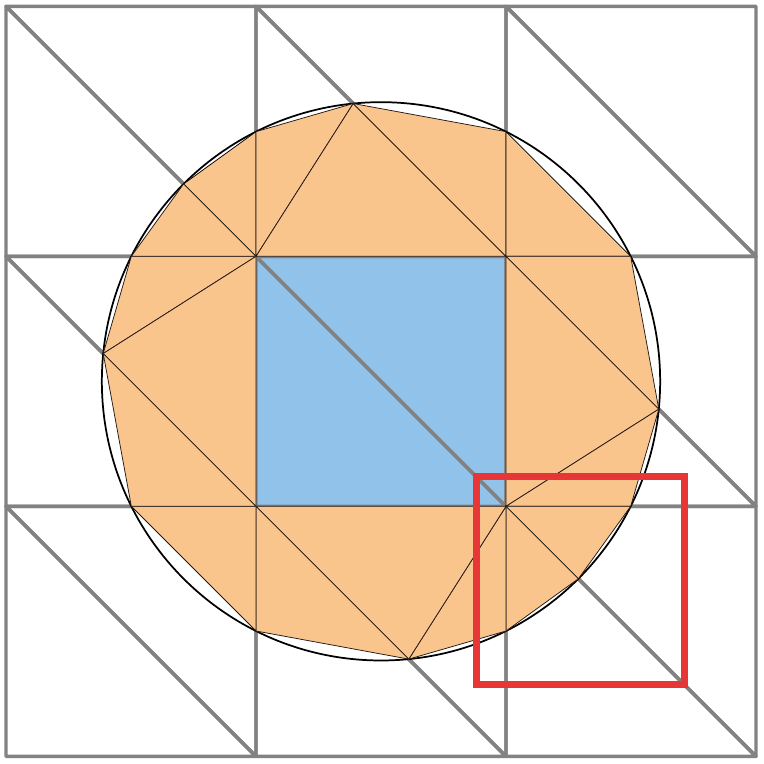}
    \label{fig:2dinscribe}
    }
    \subfigure[2D approxcaps]{
    \includegraphics[scale=0.3]{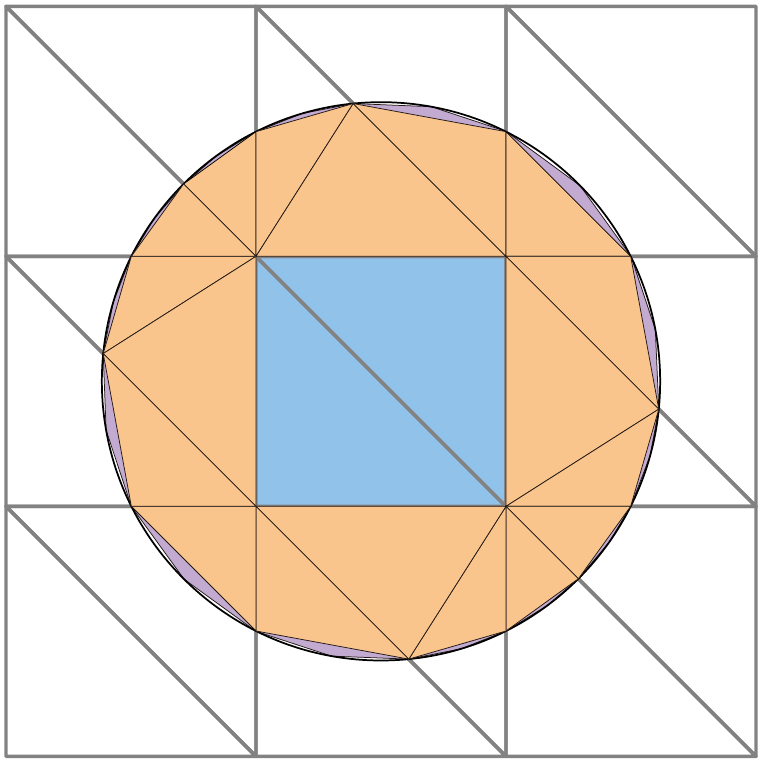}
    \label{fig:2dapproxcaps}
    }\\
    \centering
    \subfigure[3D Barycenter]{
    \includegraphics[height = 60pt, width = 60pt]{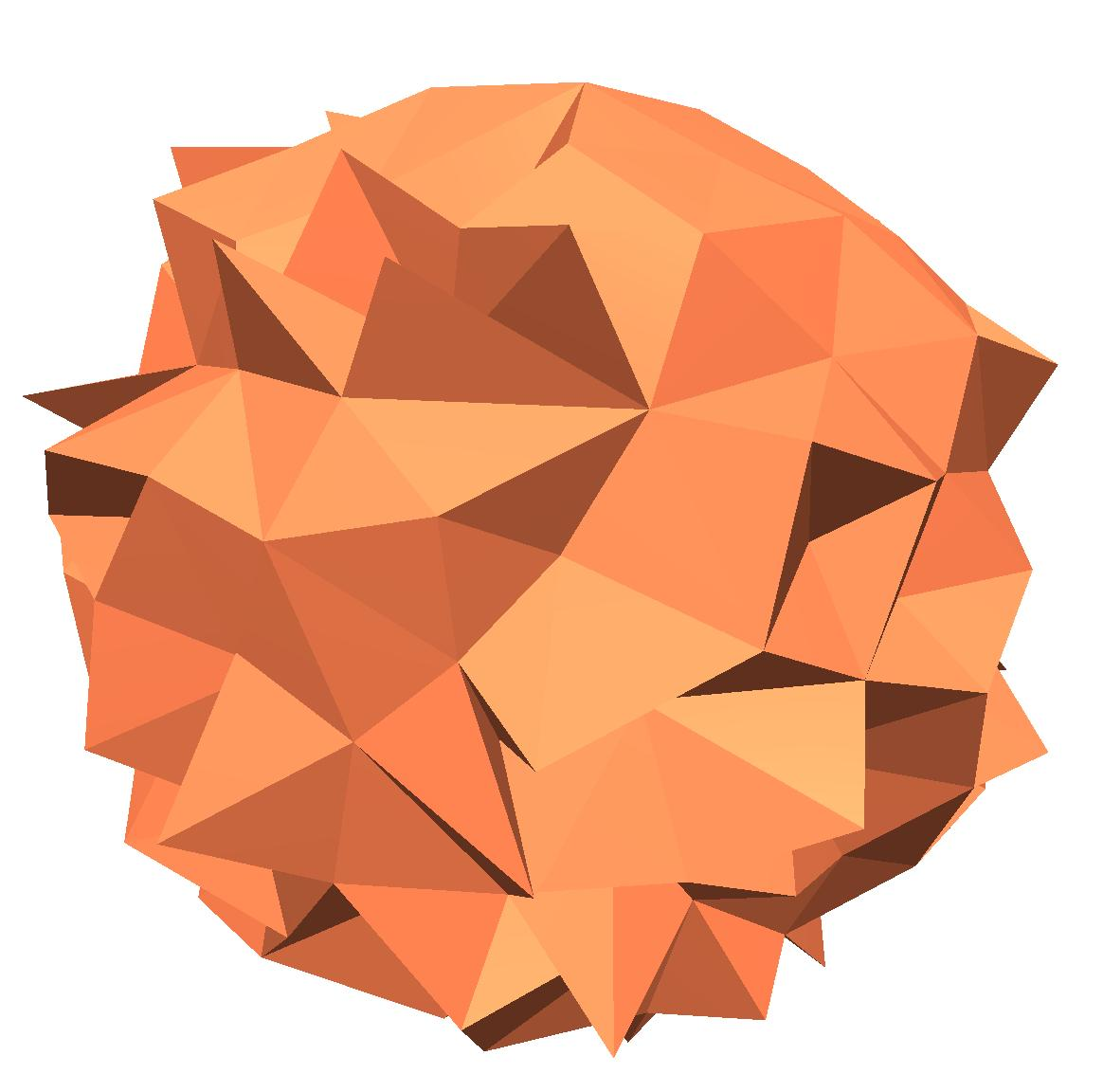}
    \label{fig:3dbary}
    }
    \subfigure[3D overlap]{
    \includegraphics[height = 60pt, width = 60pt]{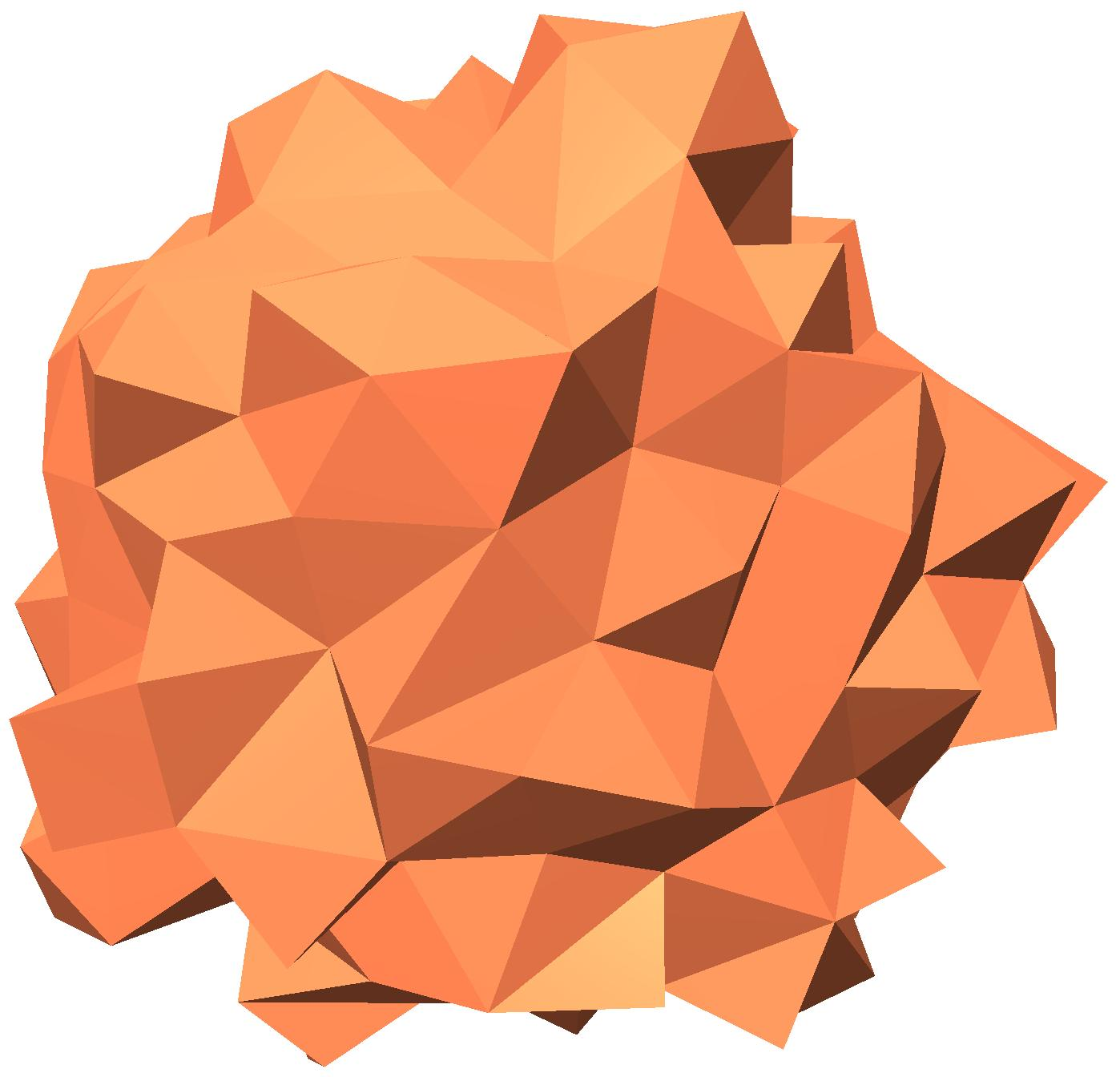}
    \label{fig:3doverlap}
    }
    \subfigure[3D nocaps]{
    \includegraphics[height = 60pt, width = 60pt]{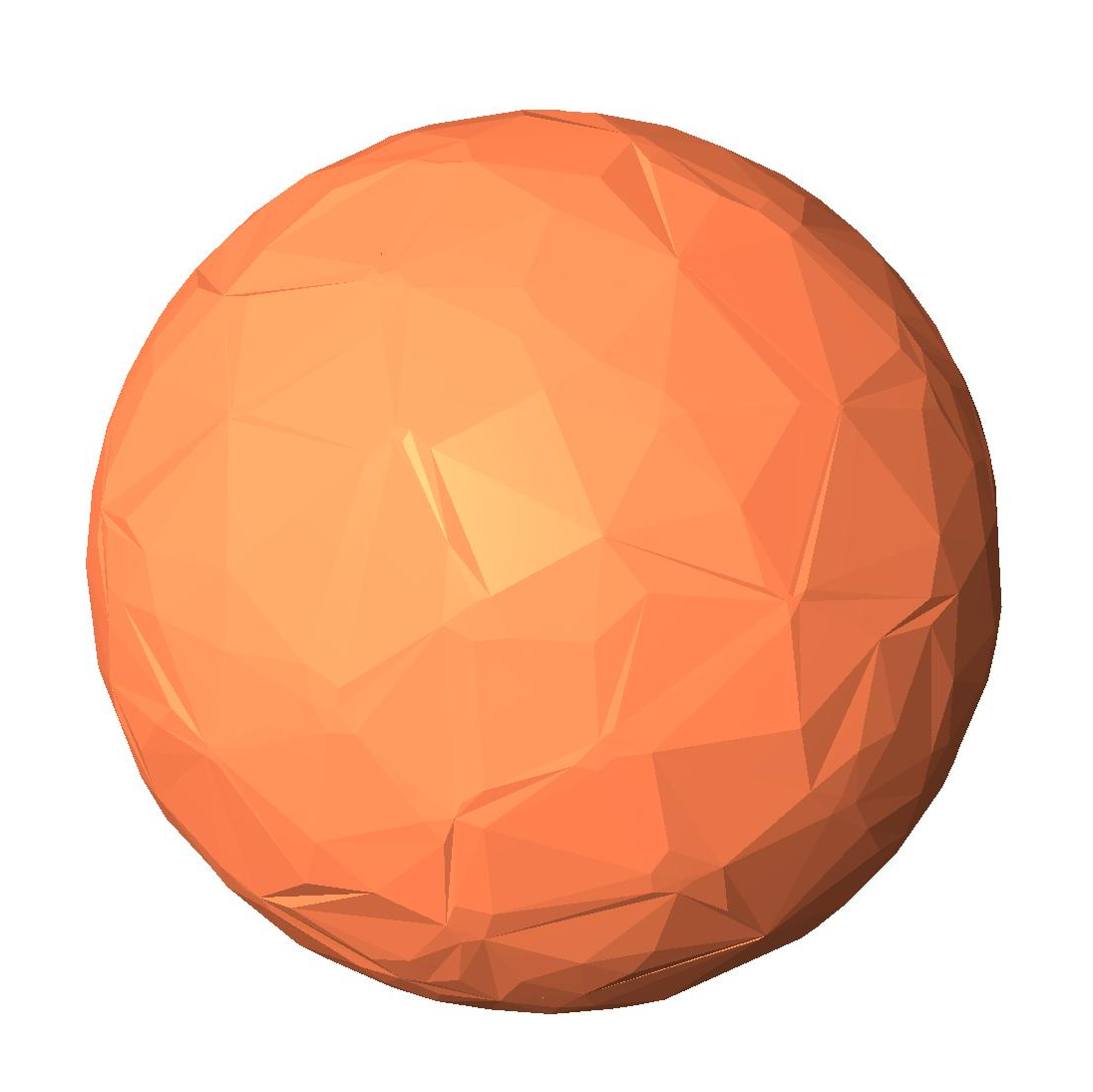}
    \label{fig:3dinscribe}
    }
    \subfigure[3D approxcaps]{
    \includegraphics[height = 60pt, width = 60pt]{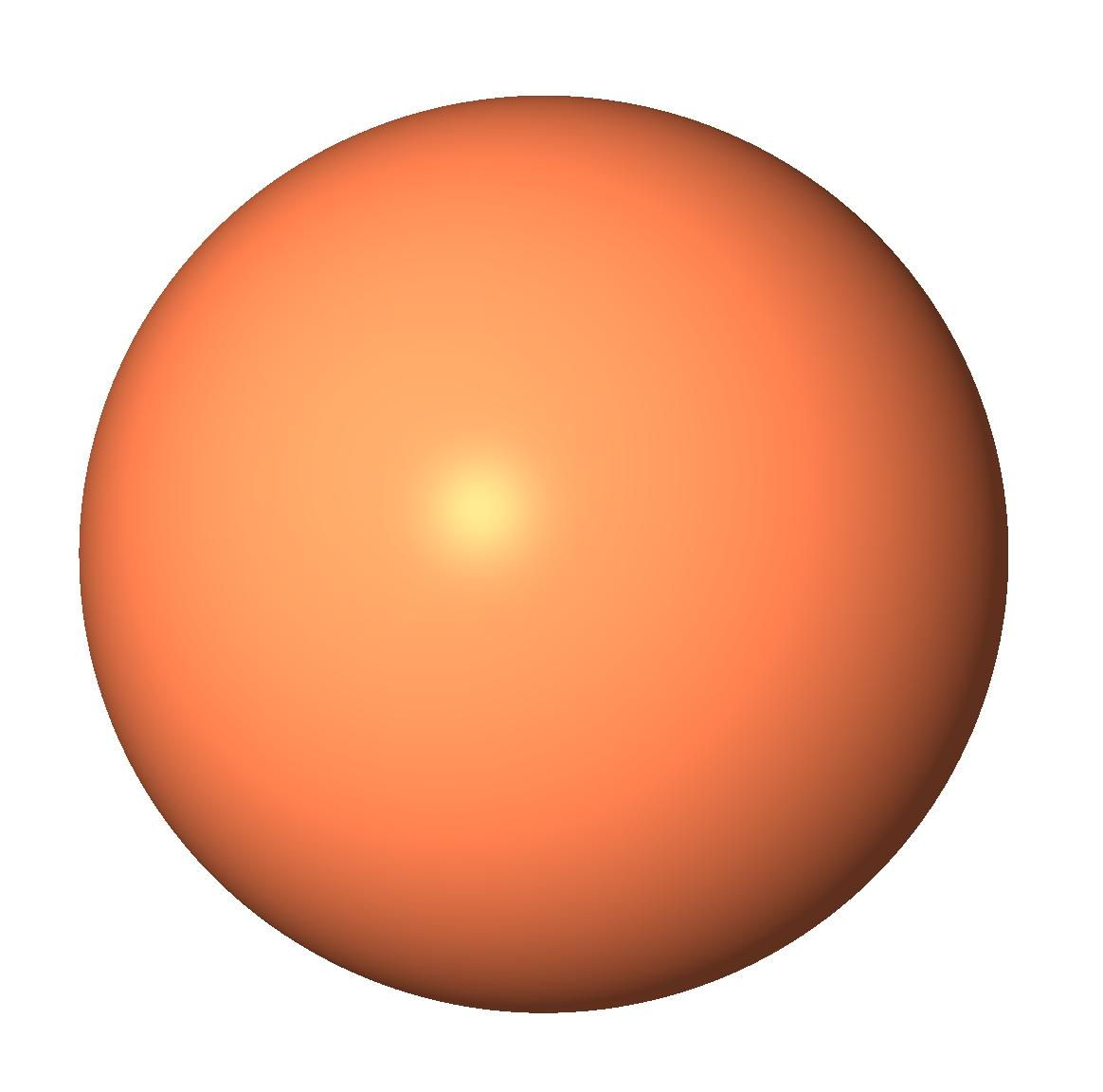}
    \label{fig:3dapproxcaps}
    }
    \caption{
    (a) (e): Approximation by elements of which the barycenter lies within the ball.
    (b) (f): Approximation by finite elements that intersect the ball.
    (c) (g): Approximation by an inscribed polytope without caps.
    (d) (h): Approximation by an inscribed polytope with subdivided caps. }
    \label{fig:approxball}
\end{figure}

Recalling the nocaps strategy of $n$-dimension we develop in Definition \ref{def:nocaps} and the approxcaps strategy, nocaps strategy of 2D in \cite{d2021cookbook}, it is obvious that the assumption in Theorem \ref{thm:inscribedapprox} are all satisfied. Because for each face $\mathcal{F}_i$ on $\partial B_{\delta,h}(x)$, there exists an element $\mathcal{E}_j$ such that $\mathcal{F}_i\subset\mathcal{E}_j$ with the diameter of $\mathcal{F}_i$ less than $h$. So, $\mathcal{F}_i$ is a $(n-1)$-dimensional simple polytope that can be subdivided into some $(n-1)$-dimensional {simplices}, with every simplex's edge length less than $h$. In short, $B_{\delta,h}(x)$ can be viewed as an $n$-dimensional polytope with every face of this polytope is an $(n-1)$-dimensional simplex with length less than $h$.

Similarly, we define a simplified 3D approxcaps strategy based on the nocaps strategy defined in Definition \ref{def:nocaps}. As we can see in Figure \ref{fig:3DcapsCases}, the relationship between a ball and the tetrahedron is complicated. We only consider the cases showed in Figure \ref{fig:caps1v}, \ref{fig:caps2v} and \ref{fig:caps3v}. By considering the convex hull of vertices of the yellow part and the midpoint of each curve of the blue part in Figure \ref{fig:caps1v}, \ref{fig:caps2v} and \ref{fig:caps3v}, we are able to construct a polytope that approximate $\mathcal{E}_k\cap \partial B_\delta(x)$ better. This approxcaps strategy is much more suitable for implementation in 3D compared to the approxcaps strategy in \cite{d2021cookbook}, and is already able to reduce the geometric error {of approximate ball} very well in the actual experiment.


\begin{figure}[htp!]
    \centering
    \subfigure[1 vertex]{
    \includegraphics[scale=0.035]{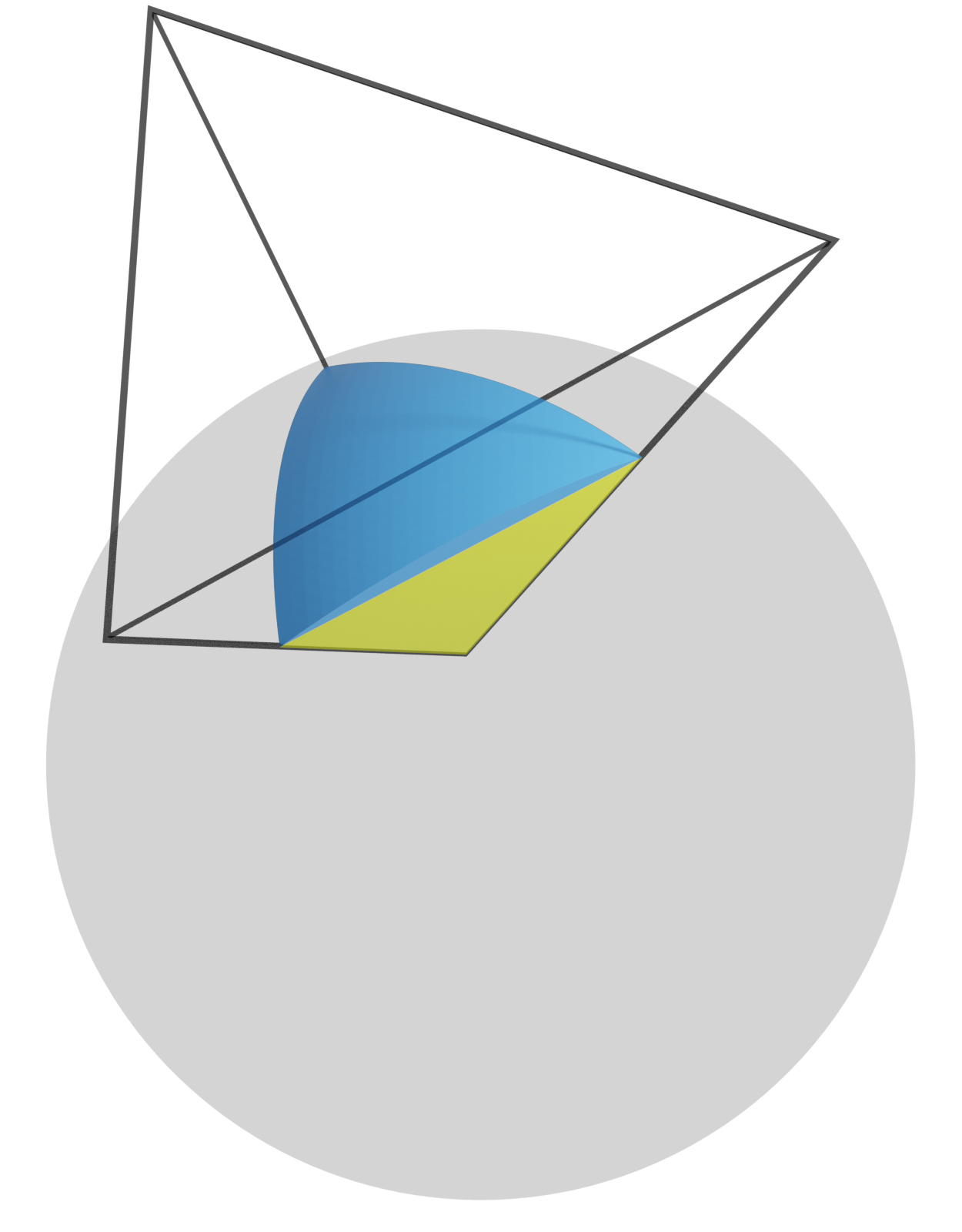}
    \label{fig:caps1v}
    }
    \subfigure[2 vertices]{
    \includegraphics[scale=0.035]{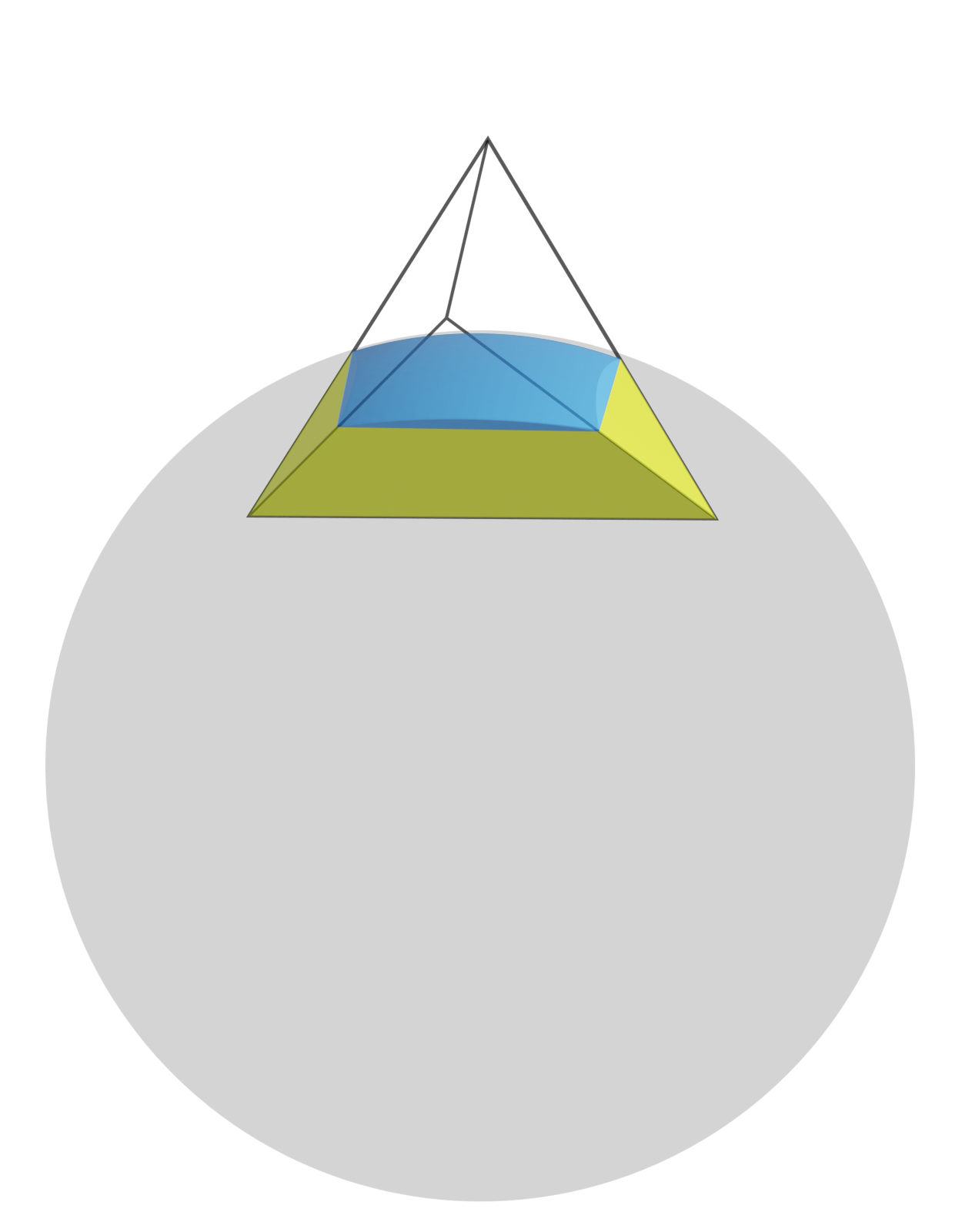}
    \label{fig:caps2v}
    }
    \subfigure[3 vertices]{
    \includegraphics[scale=0.035]{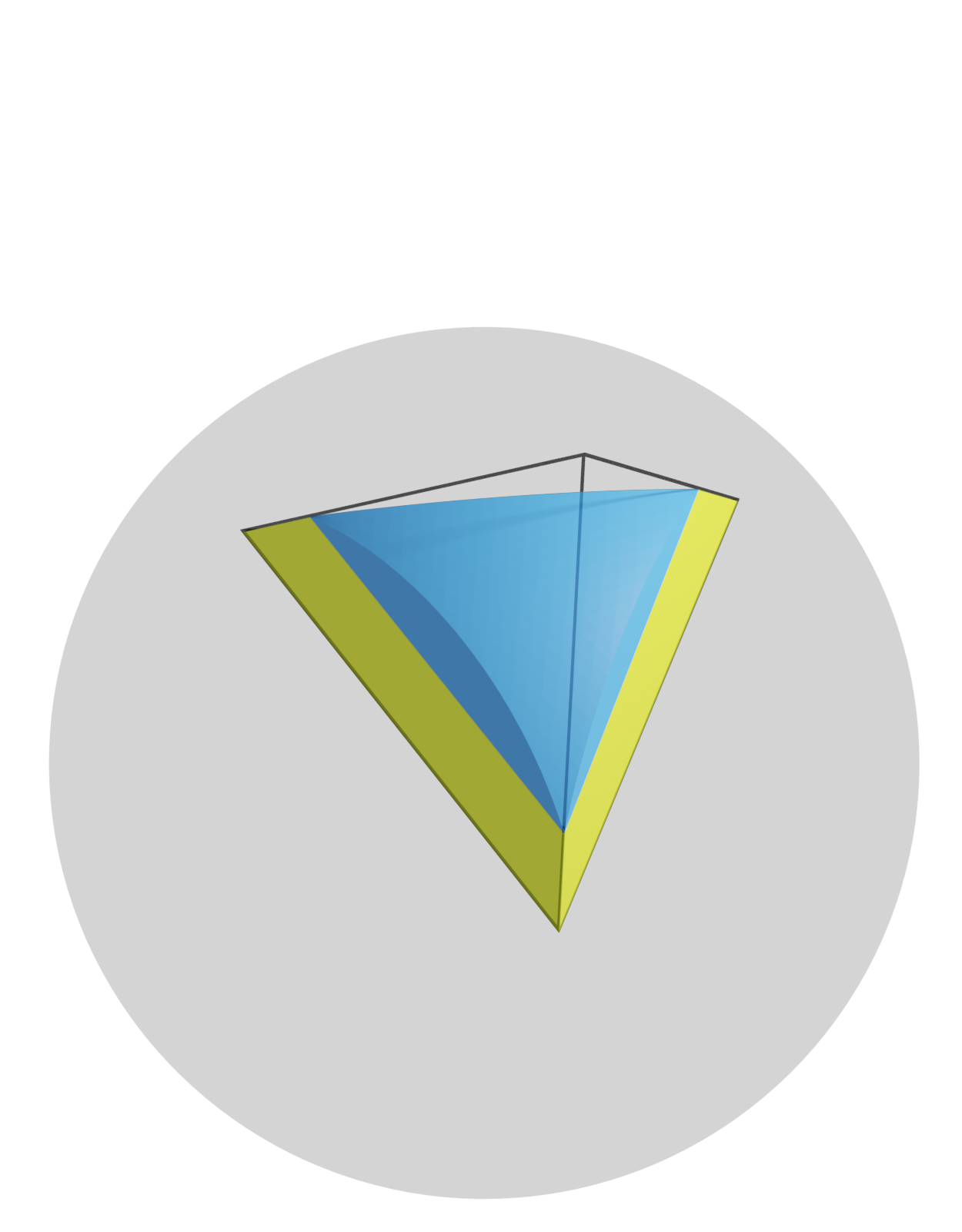}
    \label{fig:caps3v}
    }
    \subfigure[edge]{
    \includegraphics[scale=0.035]{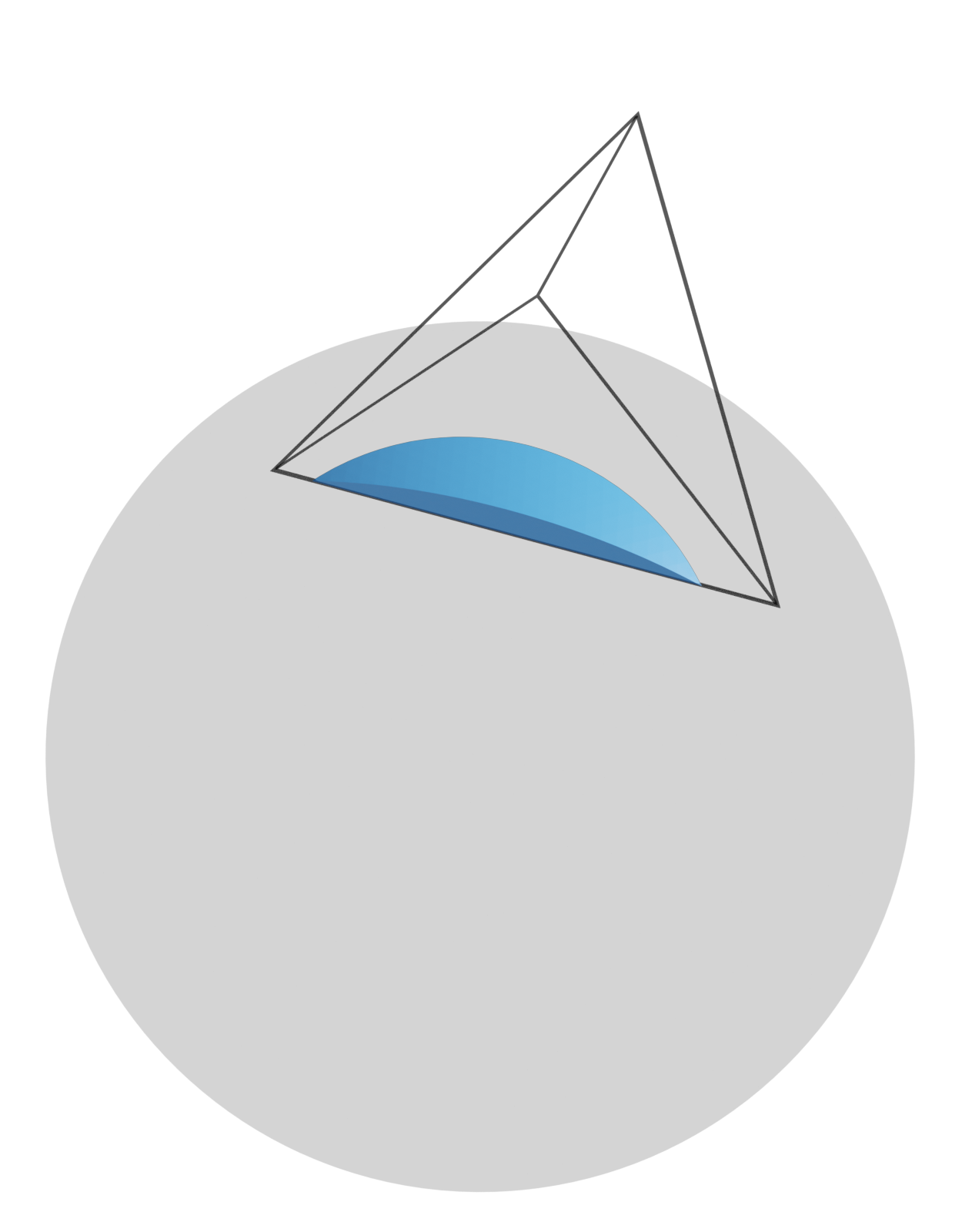}
    \label{fig:banana}
    }
    \subfigure[face]{
    \includegraphics[scale=0.035]{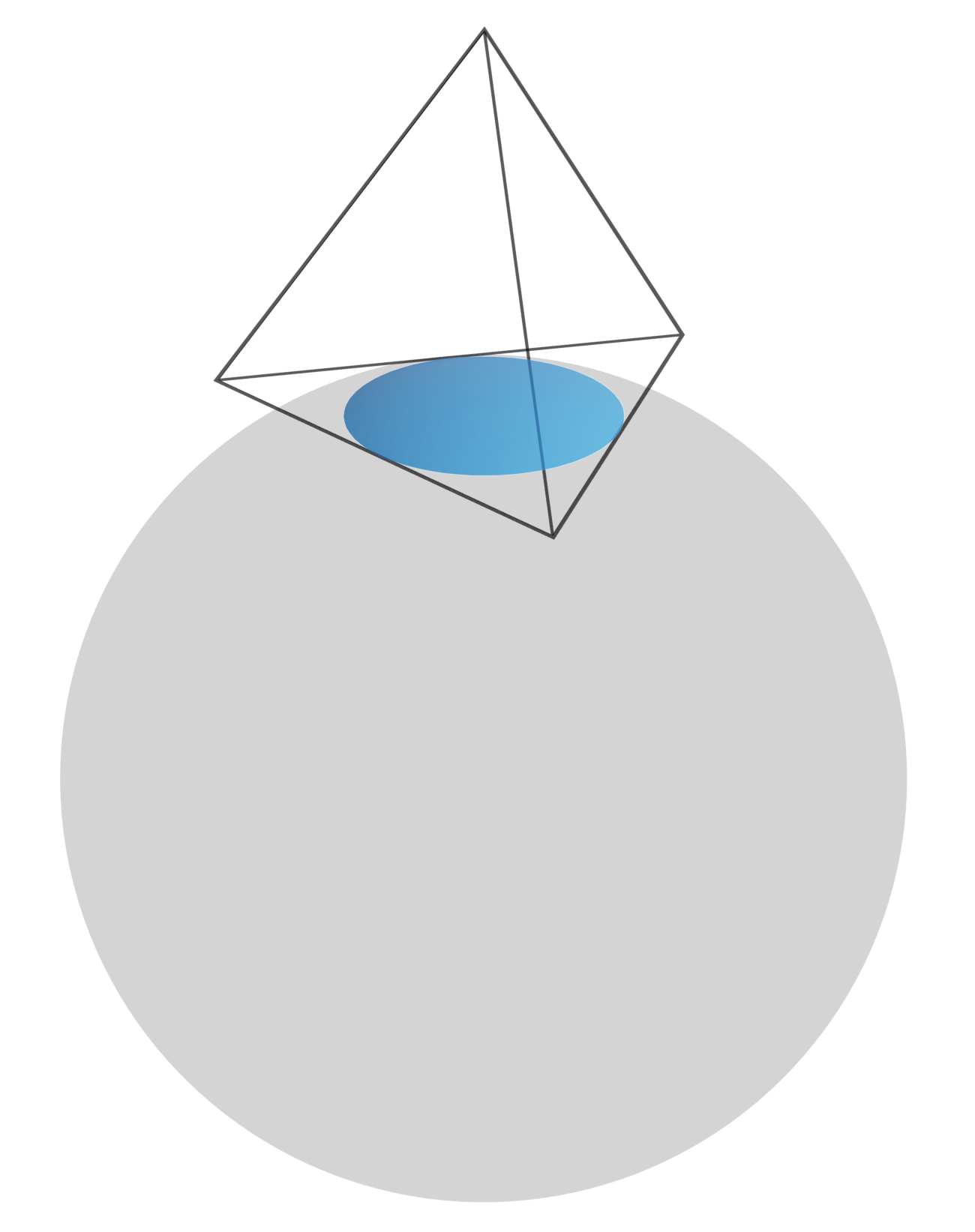}
    \label{fig:capscycle}
    }
    \subfigure[others]{
    \includegraphics[scale=0.035]{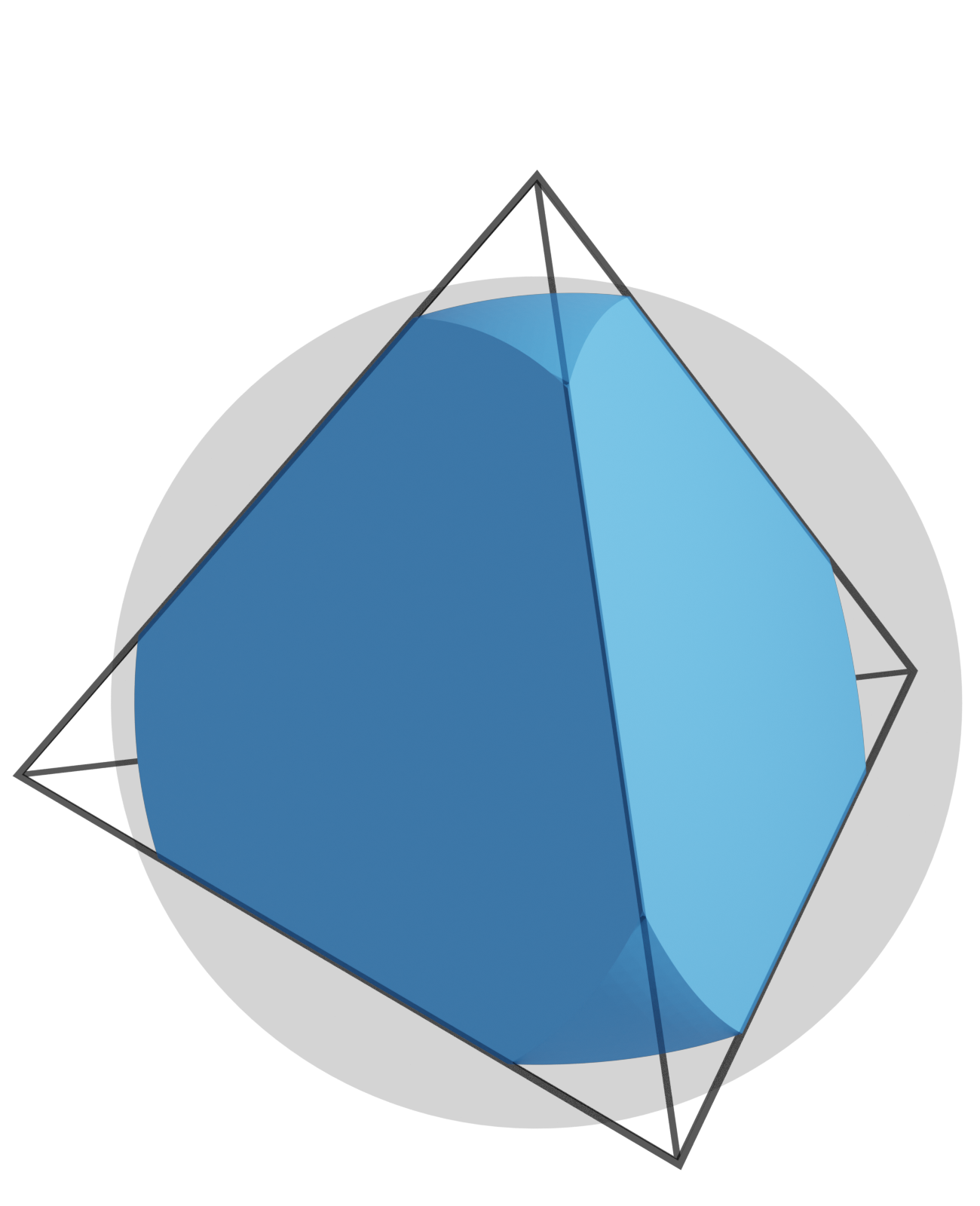}
    \label{fig:capsothercases}
    }
    \caption{six different cases for the tetrahedron (black line) and the ball (gray shadow). The fullcaps are colored by blue and the newly generated cells are colored by yellow. }
    \label{fig:3DcapsCases}
\end{figure}

\section{Implementation of Nonlocal FEM}\label{sec:implementation}

In this section, we first introduce the definition of combinatorial map theory rigorously. 
Then, we introduce the iterators designed for fast neighborhood queries and dynamic mesh modifications that are implemented in an Object-Oriented Approach. After that, we give a general interface for constructing the {polytope that approximates ball}. {The fullcaps strategy is also introduced.} Finally, we present a parallel assembly process of the nonlocal problem \eqref{eq:linear1.5}.

\subsection{Combinatorial Map}
The combinatorial map (C-map) is a mathematical model representing the topology of the subdivision of orientable objects, which is consistently defined in any dimension. The initial definition of the combinatorial map is given in \cite{cmap:lienhardt1991,cmap:lienhardt1994}, but it allows only to represent objects without boundaries. This definition was extended in \cite{cmap:2007,cmap:2010} to represent objects with boundaries, based on the concepts of partial permutations and partial involutions. 
First, we strictly introduce the theory of combinatorial mapping starting with the concept of ``dart".

\begin{definition}[dart/cell-tuple]
Consider a nD quasi-manifold $\mathcal{K}$, a cell-tuple is an ordered sequence of cells:
$$d:=([c^n],[c^{n-1}],\dots,[c^{1}], [c^{0}]),$$
where $[c^i]$ is an $i$-cell of $\mathcal{K}$, and the cell-tuple $d$ is defined in the order of decreasing dimensions such that $[c^{i-1}]\prec [c^{i}]$ for all $0 < i \leq n$. This cell-tuple is also referred to as ``dart''.
\end{definition}
For the sake of economy {of} expression, the mapping from dart $d$ to its $i$-dimensional cell is denoted by $C_i(d)$, and the mapping from cell $[c]$ to {one of} its darts is denoted by $D([c])$. {In the implementation of our FEM, this dart can be chosen by the user freely because we will only use $D(\cdot)$ as initialization in our algorithm.}

For these cell-tuples corresponding to $\mathcal{K}$, the two cell-tuples are said to be $i$-adjacent if they share all but the $i$-dimensional cell. 
In fact, we can define a set of $n+1$ mappings $\{\alpha_i(\cdot)\}_{i=0}^{n}$ called partial perturbations. Intuitively, we first denote $\epsilon$ as a null and $B$ as the finite set that contains all cell tuples corresponding to $\mathcal{K}$. The partial permutation $\alpha_i$ related to the quasi-manifold $\mathcal{K}$ is a map from $B\cup \{\epsilon\}$ to $B\cup \{\epsilon\}$, defined based on the $i$-adjacency relations of the cell-tuples:
\begin{itemize}
\item $\alpha_i(\epsilon)=\epsilon$;
\item $\forall\, d\in B, \alpha_i(d)=d'$ if there exists a $d'$ that is $i$-adjacent to $d$, otherwise $\alpha_i(d)=\epsilon$.
\end{itemize}
{These $\{\alpha_i(\cdot)\}_{i=0}^{n}$ are uniquely defined on $B$. For a given partial permutation $f$, the inverse of it is defined as:}
\begin{itemize}
\item $f^{-1}(\epsilon)=\epsilon$;
\item $\forall\, d\in B, f^{-1}(d)=d'$ if there exists a $d'$ that satisfy $f(d')=d$, otherwise $f(d)=\epsilon$.
\end{itemize}

In implementation, one can define $\epsilon$ as an empty pointer, and define a modified $\alpha_n$ in the following {way}:
$$
\dot{\alpha}_n(d)\left\{
\begin{aligned}
&= \alpha_n(d),\, &&\mbox{if $\alpha_n(d)\neq\epsilon$,}
\\
&= d,\, &&\mbox{if $\alpha_n(d)=\epsilon$,}
\end{aligned}
\right.
\qquad\mbox{for $d \in B$}.
$$
An example of using the modified $\alpha_n$ is shown in Figure \ref{fig:mapsa}, where a 2D geometric object is expressed by darts, and they interact with each other by $\alpha_0$, $\alpha_1$, and $\alpha_2$.

\begin{figure}[htp!]
    \centering
    \subfigure[mappings between darts]{
    \includegraphics[scale=0.35]{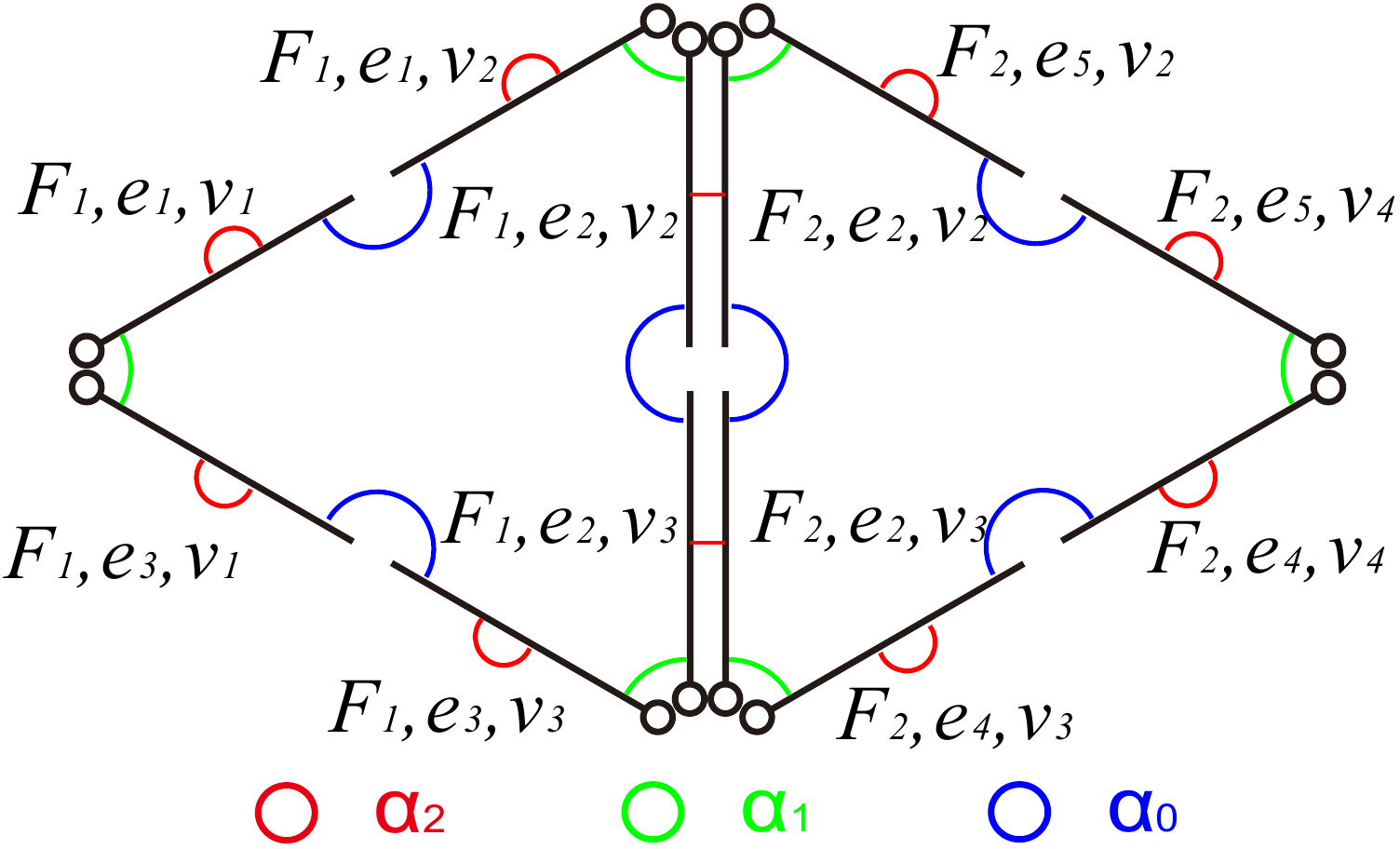}
    \label{fig:mapsa}
    }
    \subfigure[combinatorial map]{
    \includegraphics[scale=0.35]{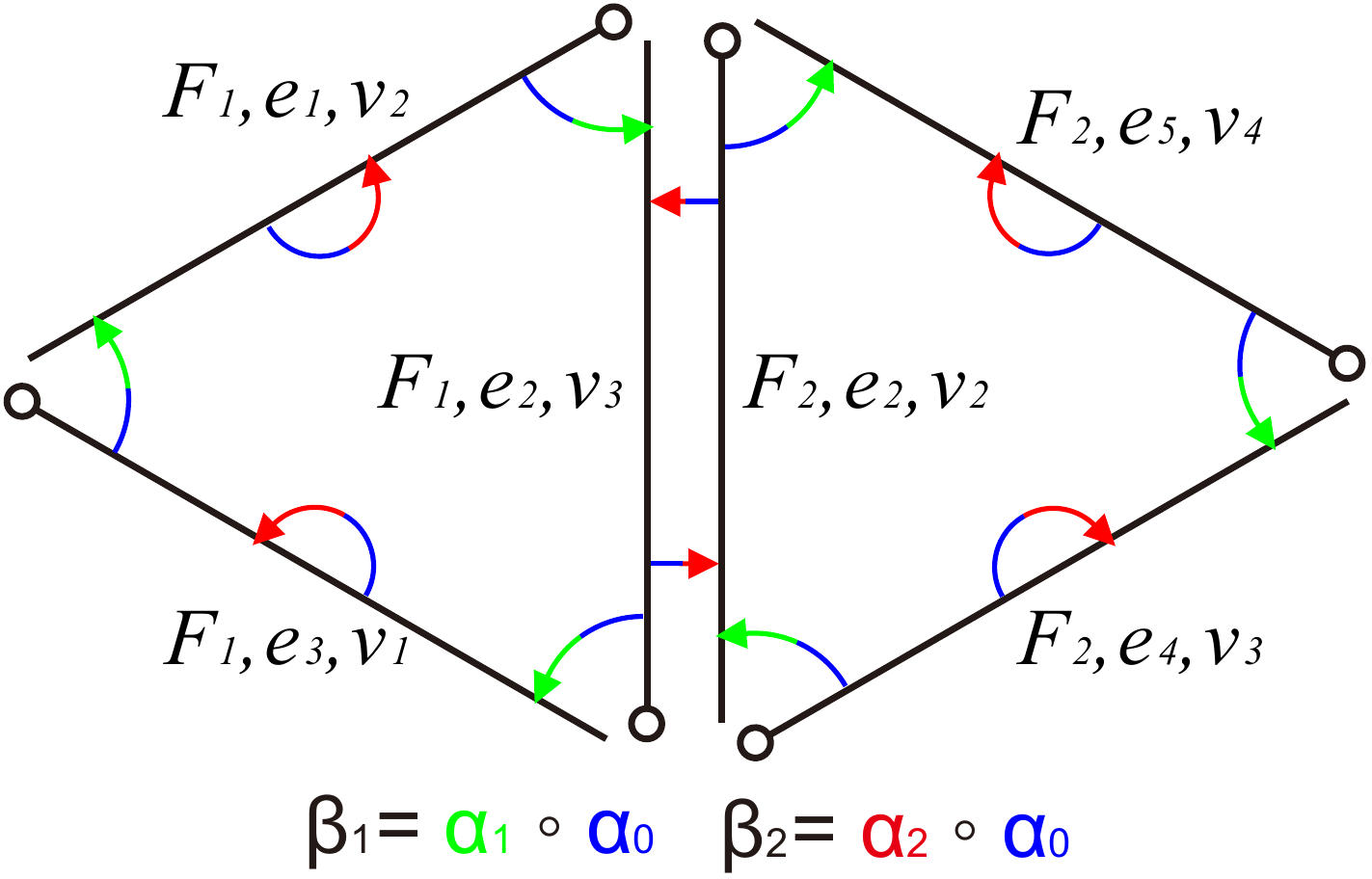}
    \label{fig:mapsb}
    }
    \caption{(a) shows $\alpha_0$, $\alpha_1$, $\alpha_2$ between the darts corresponding to a triangulation of a geometry by denoting a dart as (F,e,v). (b) shows the C-map, and mappings $\beta_1$, $\beta_2$. }
\end{figure}

We define the partial permutations
$$ \beta_i = \alpha_i \circ \alpha_0,\quad\forall\,0\leq i\leq n,$$
which connect two darts from different $0$-adjacent cell-tuples pairs.
Now we are able to give the definition of a combinatorial map in nD. The selected cell-tuples and the mapping $\beta_i$ can form an algebra called C-map. We have the following definition:
\begin{definition}[Combinatorial map]
When $n\geq2$, consider an orientable quasi-manifold $\mathcal{K}$. An $n$-dimensional C-map is an algebra $C=(R\cup\{\epsilon\},\beta_1,\cdots,\beta_n)$. Here $R$ is some cell-tuples that are selected by a given orientation of $\mathcal{K}$
\end{definition}
More details about the definition of $R$ can be found in \cite{cmap:2007,cmap:2010} or in our supplement. Similarly, we can define $\dot{\beta}_i$ by following a similar way of defining $\dot{\alpha}_i$ in the implementation. A 2D example of C-map is showed in Figure \ref{fig:mapsb}, and a 3D example of C-map is showed in Figure \ref{fig:3dmaps}.


\begin{figure}[htp!]
    \label{fig:3dmaps}
    \centering
    \subfigure[]{
    \includegraphics[scale=0.35]{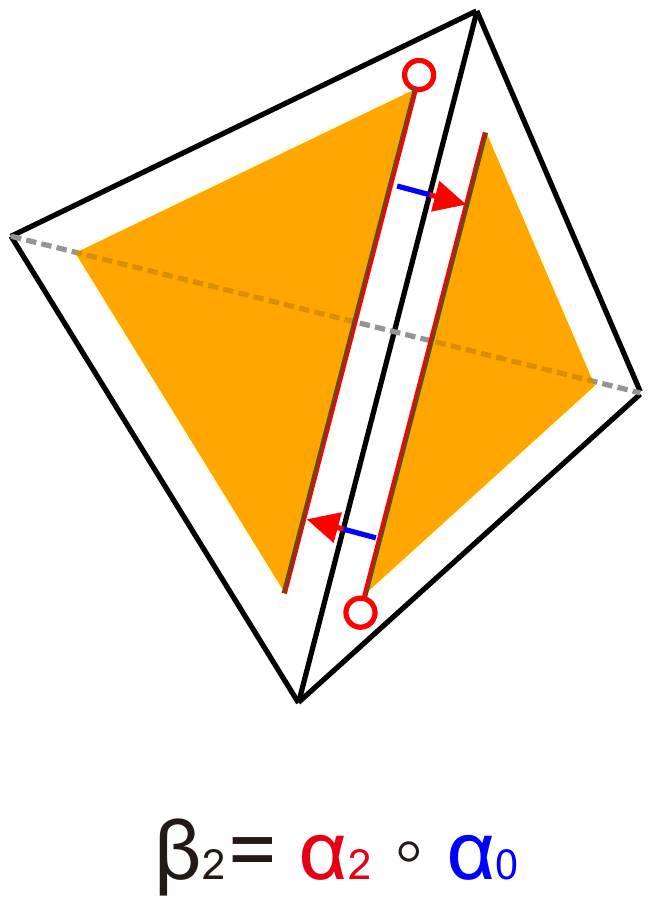}
    \label{fig:3dmapsa}
    }
    \subfigure[]{
    \includegraphics[scale=0.35]{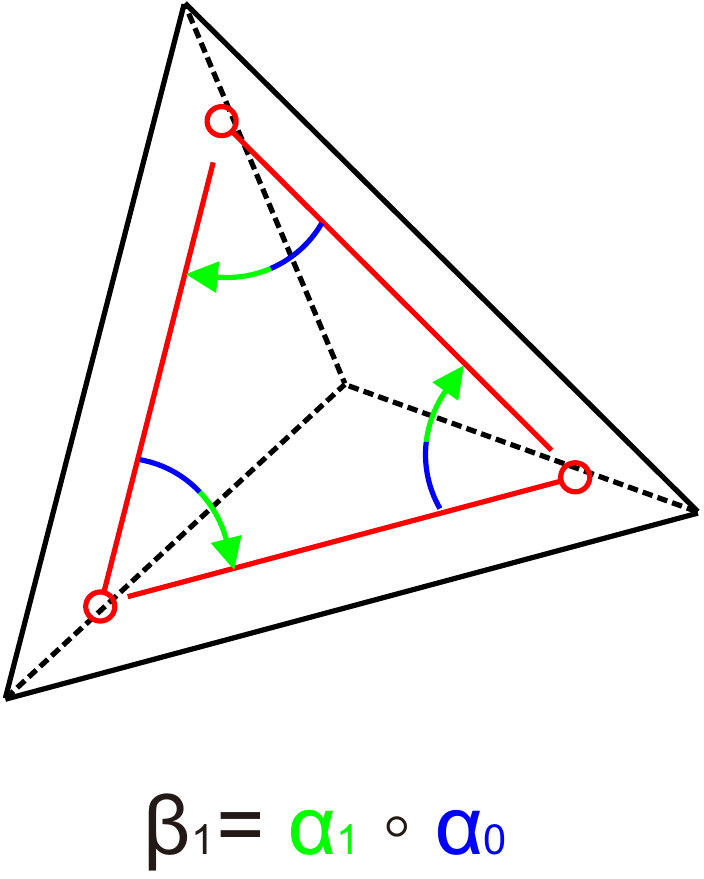}
    \label{fig:3dmapsb}
    }
    \subfigure[]{
    \includegraphics[scale=0.35]{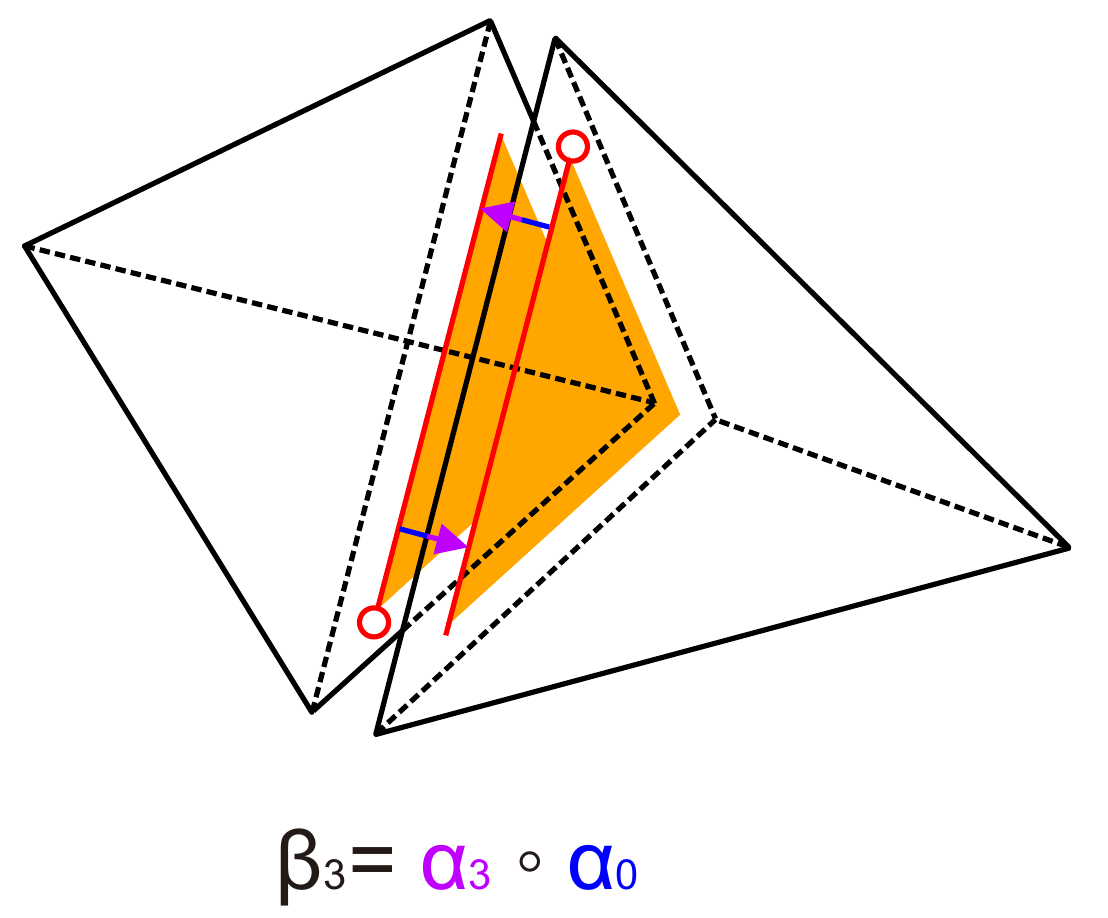}
    \label{fig:3dmapsc}
    }
    \caption{combinatorial map for an object consisting of two triangles. (a) the mapping $\beta_2$ associate two darts that have the common edge and volume but different faces. (b) the mapping $\beta_1$ associate two darts that have the common face and volume but different edges. (c) the mapping $\beta_3$ associate two darts that have the common face and edge but different volumes. }
\end{figure}

The concrete implementation of the C-map is achieved by following its mathematical definition and using the Object-Oriented Programming method. In an $n$-dimensional C-map, the cell is defined as an object with one pointer $D$ to dart and some attributes (such as coordinates of $0$-cell, material of this $0$-cell, etc.). The dart is defined to be an object with $n$ pointers $\beta_1, \beta_2, \dots, \beta _n$ and $n+1$ pointers $C_0, C_1, C_2,\dots,C_n$, where $\beta_i$ points to its $i$-adjacent dart, and $C_i$ points to its $i$-dimensional cell. The pointer $\beta_n$ is replaced by an empty pointer (i.e. null: $\epsilon$) if the dart is at the boundary of this quasi-manifold. For readers who want to use C-map to implement other meshes such as quadrilateral mesh in 2D, they may refer to \cite{cmap:2010}.

\subsubsection{Dynamical Mesh Modification}

C-map provides efficient tools to locate the darts {that} need to be modified. In this paper, we only care about how to use C-map to efficiently {generate a polytope $B_{\delta,h}(x)$ to approximate the ball $B_\delta(x)$}. The orbit defined as follows is used to efficiently accomplish this.

\begin{definition}[Orbit]\label{def:orbit}
Consider a given C-map $C=(R,\beta_1,\cdots,\beta_n)$, and a set of partial permutations $\{f_1,\cdots,f_k\}$ {defined on $R$}. {The set of darts that can be reached from $d$ through $f_i$, i.e. $\langle f_1,\cdots,f_k\rangle(d)\setminus\epsilon=\{f(d)|f\in\langle f_1,\cdots,f_k\rangle\}\setminus\epsilon$, is called the orbit of $d\in D$ related to $\{f_1,\cdots,f_k\}$.} Here $\langle f_1,\cdots,f_k\rangle$ is the group generated by $\{f_1,\cdots,f_k\}$.
\end{definition}

The {importance of defining orbit is to provide a tool that can efficiently find the darts associated with a given $i$-cell on the C-map}, for which we have the following theorem:
\begin{theorem}[\cite{cmap:2010}]\label{thm:orbit}
Assume $n\geq 2$. Consider a given $n$-dimensional orientable quasi-manifold $\mathcal{K}$ and its C-map $C=(R,\beta_1,\cdots,\beta_n)$. Let $d\in R$ be a dart, and $[c^i]$ is the $i$-cell of $d$. {If} the quasi-manifold $\mathcal{K}$ {satisfies} the following constraint:
\begin{itemize}
\item For any two $n$-cells $[c_1^n],[c^n]\in\mathcal{K}$, if there is an $p$-cell $[c^p]$ satisfying $[c^p]\prec[c_1^n],[c^n]$, then there is a series of $n$-cells and $(n-1)$-cells that are separated from each other:
    $$
    [c_1^n],[c_1^{n-1}],[c_2^n],[c_2^{n-1}],\cdots,[c_k^n],[c_k^{n-1}],[c^n]=[c_{k+1}^n],
    $$
    such that $[c_i^{n-1}]$ is the face of $[c_i^n]$ and $[c_{i+1}^n]$, and it satisfies $[c^p]\prec[c_i^{n-1}]$ for $1\leq i\leq k$.
\end{itemize}
Then it can be proved that
\begin{itemize}
\item $\{d'\in R|C_0(d')=[c^0]\}=\langle \{ \beta_i \circ \beta_j | \forall i,j: 1\leq i < j \leq n \} \rangle (d)\setminus\epsilon;$
\item $\{d'\in R|C_i(d')=[c^i]\}=\langle \beta_1,...,\beta_{i-1}, \beta_{i+1},...,\beta_n \rangle(d)\setminus\epsilon,\quad\forall\,1 \leq i \leq n.$
\end{itemize}
\end{theorem}



\begin{figure}[htp]
    \centering

    \subfigure[vertex orbit]{
    \includegraphics[scale=0.3]{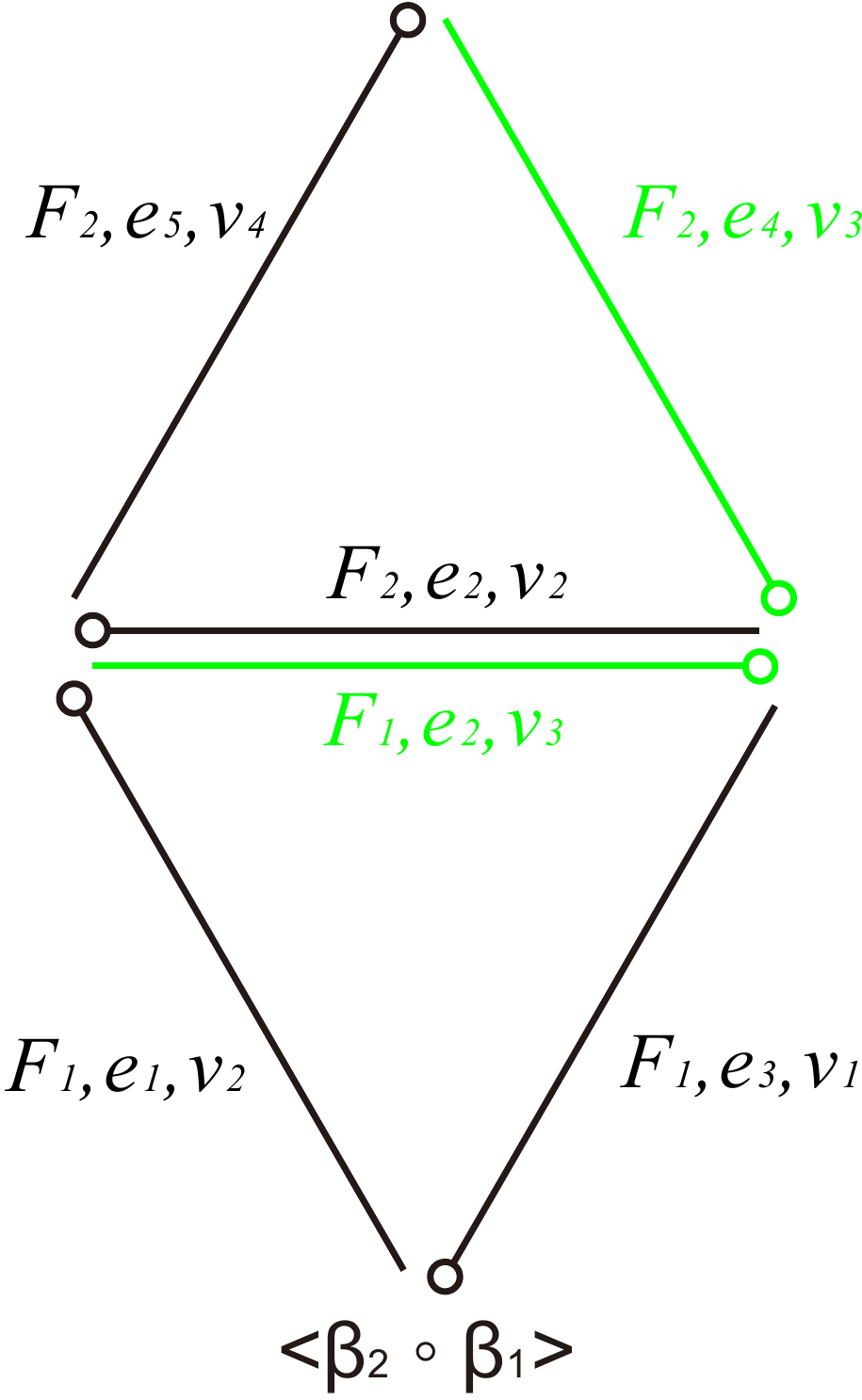}
    \label{fig:vertex}
    }
    \subfigure[edge orbit]{
    \includegraphics[scale=0.3]{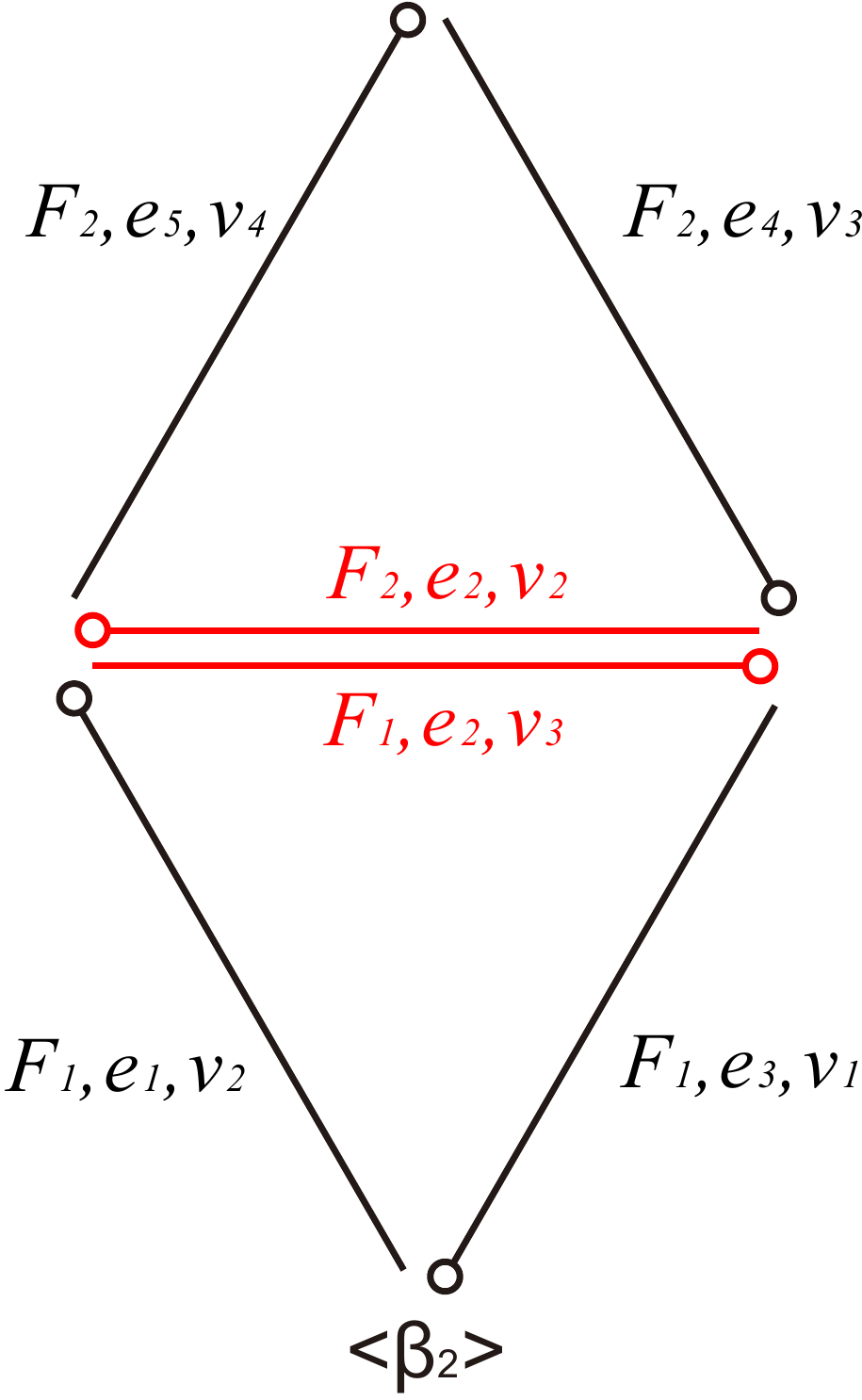}
    \label{fig:edge}
    }
    \subfigure[face orbit]{
    \includegraphics[scale=0.3]{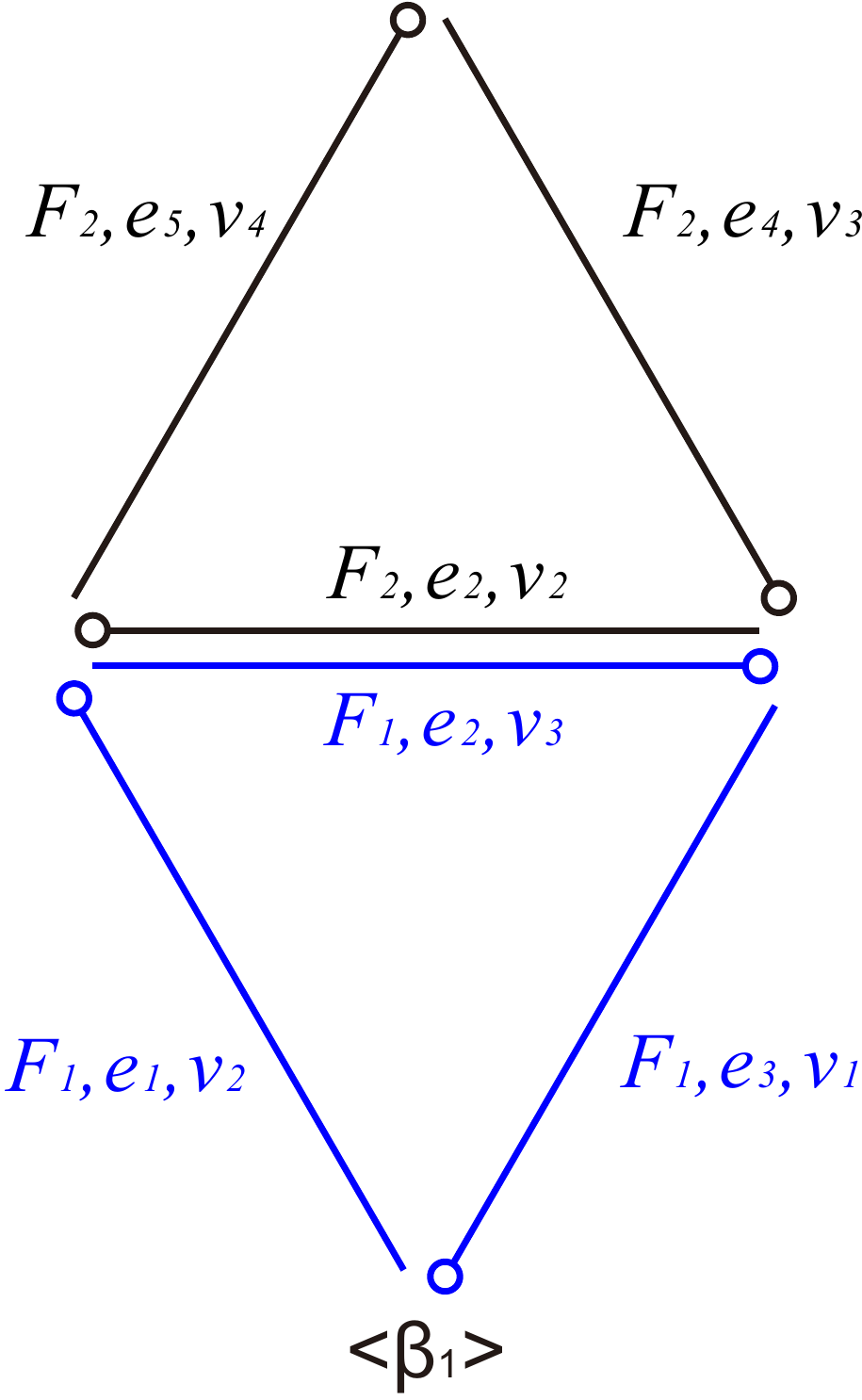}
    \label{fig:face}
    }
    \quad
    \caption{Three different types of orbit. (a), (b) and (c) shows the darts that contain the vertex $v_3$, edge $e_2$ and face $F_1$, respectively.}\label{fig:orbit}
\end{figure}


The orbit provides a tool to efficiently query cells on the mesh. If we need to remove or subdivide an $i$-cell, all the darts that are related to it can be accessed by traversing the orbit.  A 2D example given in Figure \ref{fig:orbit} shows that if this geometry object breaks at the edge $e_2$, all the darts (colored in red) that need to be updated can be accessed by the orbit $\langle \beta_2 \rangle (F_1,e_2,v_3)$. In the rest of this paper, we use $\beta_n$ in our implementation of the query algorithm. 
The ``neighboring tetrahedrons iterator'' presented in Algorithm \ref{alg:Tet2Tetiterator} is an iterator on $3$-dimensional mesh for traversing the neighboring tetrahedrons that are adjacent to the given tetrahedron. More efficient neighborhood iterators that are useful for nonlocal problems can be found in \cite{cmap:2010}.

\begin{algorithm}[!ht]
 \caption{neighboring tetrahedrons iterator}
 \label{alg:Tet2Tetiterator}  
 \begin{algorithmic}[1]
     \REQUIRE A given tetrahedron $t$;
        \ENSURE the tetrahedrons that are adjacent to $t$
        \STATE $d_0 \leftarrow D(t)$
        \STATE $\mathbf{if}$ $\beta_3(d_0)$ $\neq$ $\epsilon$ $\mathbf{then}$ output $C_3(\beta_3(d_0))$
        \STATE $d \leftarrow d_0$
        \REPEAT
        \STATE $d_2 \leftarrow \beta_3(\beta_2(d))$
        \STATE $\mathbf{if}$ $d_2$ $\neq$ $\epsilon$ $\mathbf{then}$ output $C_3(d_2)$
        \STATE $d \leftarrow \beta_1(d)$
        \UNTIL{$d = d_0$}
 \end{algorithmic}
\end{algorithm}

\subsection{Approximate Ball and Quadrature Rules}

As we have mentioned above, we {adopt a polytope} $B_{\delta,h}(x)$ to approximate the Euclid ball $B_{\delta}(x)$, and many approximation methods have been proposed in \cite{approx}.
For example,
\emph{barycenter}: by finite elements of which the barycenter lies within the ball; \emph{overlap}: by finite elements that intersect the ball; \emph{Inside}: by finite elements that wholly inside the ball; \emph{nocaps}: by simplices that is subdivided from \emph{overlap}; and \emph{approxcaps}: by simplices from \emph{nocaps} and the approximation of these caps. For convenient illustration, the examples of the 2D or 3D situation have been given in Figure \ref{fig:approxball}.

The implementation of the approximation \emph{nocaps} and \emph{approxcaps} is not as easy as that of \emph{overlap} and \emph{barycenter}, because the approximation strategies \emph{nocaps} and \emph{approxcaps} {consist of additional cells subdividing for the cells that intersect the ball.} As shown in Figure \ref{fig:2dapproxcaps}, the blue triangles belong to the finite element cells, and each of the orange triangle is part of a finite element cell. To alleviate these difficulties encountered in the finite element assembly process and make the process more efficient, we will use C-map data structure introduced in the above subsection. We here introduce how to construct those approximations with the C-map.

Algorithm \ref{alg:iterator} is a general interface provided for constructing $B_{\delta,h}(x)$. For a given $n$-cell $[c^n]$ and one of its quadrature point $p$, the cells that are adjacent to $[c^n]$ are traversed in a breadth-first way. The cells that satisfy $\mathcal{E}_k\subset B_{\delta}(p)$ are pushed into the queue $Q$. The cells that are not entirely included in the $B_\delta(x)$ will be specially treated according to the choice of ball approximation strategies. The polytope $B_{\delta,h}(x)$ consists of the newly generated cells and the cells that are fully contained. For example, if we use the nocaps strategy in the Definition \ref{def:nocaps}, the related algorithm is given in Algorithm \ref{alg:nocapsiterator}.

\begin{algorithm}[!ht]
 \caption{construction of approximate ball $B_{\delta,h}(\cdot)$ } \label{alg:iterator}
 \begin{algorithmic}[1]
     \REQUIRE A given tetrahedron $\mathcal{E}_0$ and a point $p$ in $\mathcal{E}_0$;
        \ENSURE  the approximate ball $B_{\delta,h}(p)$
        \STATE $Q$ : a queue to preserve the current searching path.
        \STATE push $\mathcal{E}_0$ into $Q$;
        \STATE output $\mathcal{E}_0$;
        \WHILE{$Q$ not empty}
        \STATE $\mathcal{E}_1$ $\leftarrow$ pop the head of $Q$
        \FOR{each cell $\mathcal{E}_2$ adjacent to $\mathcal{E}_1$}
        \IF{$\mathcal{E}_2$ is not visited \&\& $\mathcal{E}_2 \cap B_\delta(p) \neq \emptyset$}
        \STATE push $\mathcal{E}_2$ into $Q$
        \IF{$\mathcal{E}_2 \subset B_\delta(p)$}
        \STATE output $\mathcal{E}_2$
        \ELSIF{$\mathcal{E}_2 \cap B_\delta(p) \neq \emptyset$}
        \STATE output the subdivision of  $\mathcal{E}_2 \cap B_\delta(p)$ according to the choice of ball approximation strategy
        \ENDIF
                \ENDIF
        \ENDFOR
        \ENDWHILE
 \end{algorithmic}
\end{algorithm}

\begin{algorithm}[!ht]
 \caption{construction of approximate ball $B_{\delta,h}(\cdot)$ according to Definition \ref{def:nocaps}} \label{alg:nocapsiterator}
 \begin{algorithmic}[1]
     \REQUIRE A given tetrahedron $\mathcal{E}_0$ and a point $p$ in $\mathcal{E}_0$;
        \ENSURE  the approximate ball $B_\delta^\sharp(p)$
        \STATE $Q$ : a queue to preserve the current searching path.
        \STATE push $\mathcal{E}_0$ into $Q$;
        \WHILE{$Q$ not empty}
        \STATE $\mathcal{E}_1$ $\leftarrow$ pop the head of $Q$
        \FOR{each cell $\mathcal{E}_2$ adjacent to $\mathcal{E}_1$}
        \IF{$\mathcal{E}_2$ is not visited}
        \STATE push $\mathcal{E}_2$ into $Q$
        \IF{$\mathcal{E}_2 \subset B_\delta(p)$}
        \STATE output $\mathcal{E}_2$
        \ELSIF{$\mathcal{E}_2 \cap B_\delta(p) != \varnothing$}
        \STATE $P$ $\leftarrow$ the intersection points of $\partial B_\delta(p)$ and $\mathcal{E}_2$
        \IF{$P \neq \varnothing$ }
        \STATE $I$ $\leftarrow$ the vertices of $\mathcal{E}_2$ inside $B_\delta(p)$
        \STATE $U_I$ $\leftarrow$ calculate the convex hull of $I \cup P$
        \STATE $T_I$ $\leftarrow$ subdivision of the convex hull $U_I$ to a number of simplices
        \STATE the simplices in $T_I$ inherit the basis function and material of $\mathcal{E}_2$
        \STATE output $T_I$
        \ENDIF
        \ENDIF
        \ENDIF
        \ENDFOR
        \ENDWHILE
 \end{algorithmic}
\end{algorithm}

\subsubsection{Nocaps with Gauss Quadrature Rules}
In the \emph{nocaps} approximation, the finite elements that satisfy $\mathcal{E}_k\cap \partial B_{\delta}(p)\neq\emptyset$, i.e. the elements that are not entirely included will be subdivided into some new cells. The newly generated cells are expected to be compatible with the original cells and inherit some attributes. A newly temporal C-map is generated for representing the \emph{approxcaps} and \emph{nocaps} ball $B_{\delta,h}(x)$. A 2D example is shown in Fig. \ref{fig:dart in approx}, where the black darts belong to the original element of the mesh $\mathcal{T}_{\hat{\Omega}}^h$, and the blue darts belong to the cell-decomposition of {approximate ball} $B_{\delta,h}(x)$.

\begin{figure}
    \centering
    \includegraphics[scale = 0.4]{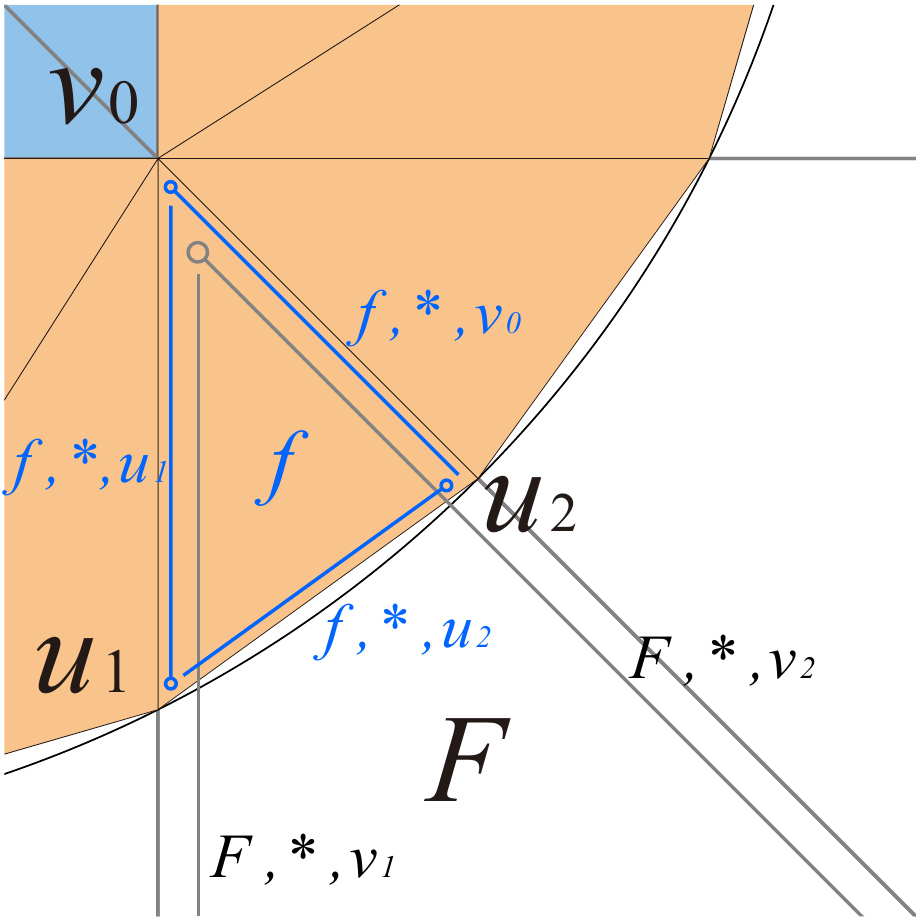}
    \caption{This is an enlargement of the red box in Fig. \ref{fig:2dinscribe}. In the process of construction approximate ball, cell $f$ is part of the subdivision of $F$. The darts colored in blue belong to the newly generated cell $f$.}
    \label{fig:dart in approx}
\end{figure}

The mapping $desk: B_{\delta,h}(x) \rightarrow \mathcal{T}_{\hat{\Omega}}^h$ is defined to drawback the cells from the {approximate ball} to the finite element mesh, such that the newly generated cells can inherit some attributes (basis function, material, etc.) from its parent's finite elements. The life cycle of the newly generated cells in an approximate ball should be consistent, i.e. these cells will be simultaneously destructed.

For quadrature rules used for outer and inner integration, we both use $4$-point Gauss quadrature rule that has a degree of precision $2$ in tetrahedron, instead of the quadrature rules such as KEAST6 based on the Keast Rule, or using the quadrature rule of tetrahedron in \cite{AS64} for outer integration, or the Dunavant 7-point rule used in \cite{AC21}. 
The selection of this quadrature rule is based on ensuring that the error of inner integration and the outer integration will not affect the convergence order of the finite element solution, and at the same time using the least quadrature points required to obtain this accuracy. One can find discussions in \cite{d2021cookbook,Vollmann2019} {about} quadrature rules, and it can be easily extended from the $2$-dimensional case to the $n$-dimensional case.

\subsubsection{Fullcaps with Monte Carlo integration}

For the ``nocaps'' approximation, the integrals over the caps are {ignored for the simplicity of programming}. The study in \cite{approx} provides a 2D strategy named ``approxcaps'' that uses a number of triangles to approximate the caps. However, ``approxcaps'' is difficult to generalize to 3D due to programming difficulties. Even in 2D, further approximation of the $B_\delta(x)$ leads to more computational operation and difficulties in implementation. By using Monte Carlo integrals, we can easily {compute an acceptable result of the integral over the complex region}. Therefore, we {propose} a new approximation strategy called ``fullcaps'', which adopts the ``Combined Geometry Via Boolean Operations'' for representing the caps and Monte Carlo quadrature rules of the {integrals over caps}.

The idea of ``fullcaps'' is mainly to deal with an element $\mathcal{E}_k$ that satisfy $\mathcal{E}_k\cap\partial B_\delta(x)\neq\emptyset$. The $\mathcal{E}_k\cap B_\delta(x)$ is subdivided into a number of (maybe zero) tetrahedrons $\{\mathcal{E}_{ki}\}$ and an additional region called fullcap. The fullcap is represented by a combined geometry via boolean operations, namely, $$
\text{fullcap}  := (\mathcal{E}_k- \cup_i \mathcal{E}_{ki}) \cap B_\delta(x).
$$
As shown in Figure \ref{fig:3DcapsCases}, we provide six different intersection cases of the tetrahedron and Euclidean ball. The fullcaps are colored by blue and the newly generated cells are colored by yellow.

A tetrahedron is an explicit geometric representation that can be used to quickly generate sample points. The Euclidean ball $B_{\delta}(p)$ is an implicit geometric representation that can quickly determine whether a point is inside the geometry. Therefore we adopt Monte Carlo method to compute the {integrals over} those fullcaps. Although this method brings white noise, the improvement of integration accuracy is enough to offset the random error brought by white noise because fullcaps make a very small contribution to the whole integral. More importantly, compared to the approxcaps, the fullcaps stratgy is much easy to be implemented, and its fullcaps ball approximation algorithm is given in Algorithm \ref{alg:fullcaps}.

\begin{algorithm}[!ht]
 \caption{construction of fullcaps and ball $B_{\delta,h}(\cdot)$ } \label{alg:fullcaps}
 \begin{algorithmic}[1]
     \REQUIRE A given tetrahedron $\mathcal{E}_0$ and a point $p$ in $\mathcal{E}_0$;
        \ENSURE  the approximate ball $B_{\delta,h}(p)$
        \STATE $Q$ : a queue to preserve the current searching path.
        \STATE push $\mathcal{E}_0$ into $Q$;
        \WHILE{$Q$ not empty}
        \STATE $\mathcal{E}_1$ $\leftarrow$ pop the head of $Q$
        \FOR{each cell $\mathcal{E}_2$ adjacent to $\mathcal{E}_1$}
        \IF{$\mathcal{E}_2$ is not visited}
        \STATE push $\mathcal{E}_2$ into $Q$
        \IF{$\mathcal{E}_2 \subset B_\delta(p)$}
        \STATE output $\mathcal{E}_2$
        \ELSIF{$\mathcal{E}_2 \cap B_\delta(p) != \varnothing$}
        \STATE $P$ $\leftarrow$ the intersection points of $\partial B_\delta(p)$ and $\mathcal{E}_2$
        \IF{$P \neq \varnothing$ }
        \STATE $I$ $\leftarrow$ the vertices of $\mathcal{E}_2$ inside $B_\delta(p)$
        \STATE $O$ $\leftarrow$ the vertices of $\mathcal{E}_2$ outside $B_\delta(p)$
        \STATE $U_I$ $\leftarrow$ calculate the convex hull of $I \cup P$
        \STATE $T_I$ $\leftarrow$ subdivision of the convex hull $U_I$ to a number of simplices
        \STATE the simplices in $T_I$ inherit the basis function and material of $\mathcal{E}_2$

        \STATE $U_O$ $\leftarrow$ calculate the convex hull of $O \cup P$
        \STATE $T_O$ $\leftarrow$ subdivision of the convex hull $U_O$ to a number of simplices
        \STATE $Fullcaps$ $\leftarrow$ the simplices in $T_O$ combines the $B_\delta(p)$ via boolean operations
        \STATE the elements in $Fullcaps$ inherit the basis function and material of $\mathcal{E}_2$
        \STATE output $Fullcaps$ and $T_I$
        \ELSE
        \STATE $Fullcaps$ $\leftarrow$ $\mathcal{E}_2$ combines $B_\delta(p)$ via boolean operations
        \STATE output $Fullcaps$
        \ENDIF

        \ENDIF
        \ENDIF
        \ENDFOR
        \ENDWHILE
 \end{algorithmic}
\end{algorithm}

The accuracy of the Monte Carlo integration depends on the sampling method. If the points are directly uniformly sampled from the tetrahedron, the probability of the points inside the fullcap may be relatively small. {This} will decrease the integral accuracy, and one may need amount of samples to raise the accuracy of integration, which may also raise the cost of computation. {The} sampling methods and integration methods {over} fullcaps need to be further studied.

\subsection{Assembly process}
In this section, we introduce an efficient method to assemble the stiffness matrix and right-hand side vector of the linear system \eqref{eq:modified_linear}.

Suppose that we have an $n$-dimensional mesh with two domains $\Omega$ and $\Omega_\mathcal{I}$ such that $B_{\delta,h}(x) \subset \Omega \cup \Omega_\mathcal{I},\,\forall\,x \in \Omega$. The maximum, average and minimum mesh size is denoted by $h_{\max}$, $h_{\text{avg}}$ and $h_{\min}$, respectively. Besides, we have $g(x):\Omega_\mathcal{I} \rightarrow R$ and $f(x):\Omega \rightarrow R$. With those settings, the linear system of the finite element discretization \eqref{eq:modified_linear} is uniquely determined mathematically.

For assembling the linear system, the task in hand is how to efficiently compute the entries of the stiffness matrix $A_h$ and the components of the right-hand side vector $\widetilde{\widetilde{F}}_h$. Similar to the local cases, matrix $A_h(\cdot,\cdot)$ is sparse, but its sparsity is less than that of the local problem. In fact, for a given element, there are $\mathcal{O}(\frac{\delta^n}{h^n})$ interacted elements. However, it is difficult to know in advance whether $A_h(\phi_j,\phi_j)$ is zero because elements that make non-zero contributions do not have to be adjacent like in local problems. 
For nonlocal problems, the query of elements in the intersection domain and their associated basis functions can be very complex. Therefore, the main idea of our algorithm is to search the pairs of finite elements that may make non-zero contributions, then traverse the basis functions $\phi_i$ and $\phi_j$ that pertained to this pair of elements, and compute their contributions to the linear system.

More precisely, we first traverse the finite elements in $\mathcal{T}_\Omega$. The outer integral is the summation of integrals {over} those elements. So we have that
\begin{equation}
        A_h(\phi_j,\phi_i) = \sum_{\mathcal{E}_n\in K_\Omega} \int_{\mathcal{E}_n}W(x)+N(x)dx,
\end{equation}
where
\begin{equation}
\label{eq:notreplaced}
\begin{split}
    W(x) = &\sum_{\mathcal{E}_m \in \Omega \cap B_{\delta,h}(x)}
\int_{\mathcal{E}_m} (\phi_{j}(y) - \phi_{j}(x)) (\phi_{i}(y) - \phi_{i}(x) )\psi(x,y)dy,\\
N(x) = &\sum_{\mathcal{E}_m \in \Omega_\mathcal{I} \cap B_{\delta,h}(x)}
2 \phi_{j}(x) \phi_{i}(x)\int_{\mathcal{E}_m}\psi(x,y)dy.
\end{split}
\end{equation}

For each finite element $\mathcal{E}_n$, we generate the quadrature points and weights $(p_k, \omega_k)$. So the integral {over} the element $\mathcal{E}_n$ can be written in the following form
\begin{equation}
    \int_{\mathcal{E}_n}W(x)+N(x)dx = \sum_{p_k\in \mathcal{E}_n} \omega_k (W(p_k)+N(p_k)).
\end{equation}

Now, the implementation difficulty of computing numerical integrals over $B_\delta(x)$ arises from the computation of the inner integral $\int_{\mathcal{E}_m}(\cdot)$ in $W(x)$ and $N(x)$. As presented and discussed in the former sections, we adopt the {polytope} $B_{\delta,h}(\cdot)$ to replace the Euclid ball $B_\delta(x)$ in integral computation. Therefore, for each quadrature point $p_k$, we generate the {polytope} $B_{\delta,h}(p_k)$, and the inner {integrals are} now over a series of simplices. The process of {generating $B_{\delta,h}(p_k)$} can follow Algorithm \ref{alg:iterator}.

The function $\phi_{(n,i)}(x)$ is defined to be $\phi_i(x) \mathcal{X}_{\mathcal{E}_n}(x)$, where $\mathcal{X}_{\mathcal{E}_n}(x)$ is the indicative function. The $\phi_{(n,i)}(x)$ is a linear function on the element $\mathcal{E}_n$ when we use Lagrange linear bases. Replacing the basis functions $\phi_i(\cdot)$ in equation \eqref{eq:notreplaced} by $\phi_{(n,i)}(\cdot)$, we have
\begin{equation}
\label{eq:replaceBasis}
\begin{split}
    W(x) = &\sum_{\mathcal{E}_m \in \Omega \cap B_{\delta,h}(x)}
\int_{\mathcal{E}_m} (\phi_{(m,j)}(y) - \phi_{(n,j)}(x)) (\phi_{(m,i)}(y) - \phi_{(n,i)}(x) )\psi(x,y)dy,\\
N(x) = &\sum_{\mathcal{E}_m \in \Omega_\mathcal{I} \cap B_{\delta,h}(x)}
2 \phi_{(n,j)}(x) \phi_{(n,i)}(x)\int_{\mathcal{E}_m}\psi(x,y)dy .
\end{split}
\end{equation}
By considering the nonzero contribution of these elements, we have the following items:
\begin{itemize}
\item When $\mathcal{E}_m\in\Omega\cap B_{\delta,h}(x)$, we have $N(x)=0$. And $W(x)\neq0$ only when $\phi_i$ and $\phi_j$ are two basis function pertained to $\mathcal{E}_n$ or $\mathcal{E}_m$.
\item When $\mathcal{E}_m\in\Omega_\mathcal{I}\cap B_{\delta,h}(x)$, we have $W(x)=0$. And $N(x)\neq0$ only when $\phi_i$ and $\phi_j$ are pertained to $\mathcal{E}_n$.
\end{itemize}
How to make good use of the geometric relationship between these elements is important for fast assmebly process. To alleviate the cost of additional judgements, the topological relations of the finite elements could be used to predict the relation between quadrature points and basis functions. This is not that difficult in the implementation once we construct the relationship between basis function and elements properly. Some details about considering the boundary layer $\Omega_\mathcal{I}$ are already discussed in \cite{d2021cookbook}.



It should be pointed out that we use Euclid coordinates instead of area coordinates because the newly generated cells in the approximate ball inherit the basis function of their parents. Because the kernel $\psi$ is unable to be computed directly from area coordinates, the transformation from the area coordinate to the Euclid coordinate is repeatedly invoked if we use area coordinates. Euclid coordinates bring convenience to integration over the approximate ball, so it is better to use them here.

The second term $\widetilde{\widetilde{F}}_h(\cdot)$ in \eqref{eq:modified_linear} is relatively easy to compute because the inner integral is independent of the basis functions, and we do not need to judge the relationship between the basis functions and the elements as in assembling $A_h$. 

One can see that in the process of considering outer integration point $p_k$, both the assembly of the stiffness matrix and the construction of the right-hand vector require the same approximate ball $B_{\delta,h}(x)$. Therefore, the assembly of the right-hand vector can be carried out simultaneously with the assembly of the stiffness matrix. By combing all the discussions above, the pseudo-code of this process is presented in Algorithm \ref{alg:main}.

\begin{algorithm}[!ht]
	\caption{CG Elliptic equation} \label{alg:main}  %
	\begin{algorithmic}[1]%
	    \REQUIRE Suitable Mesh $\mathcal{T}^h_{\hat{\Omega}}$, $\delta$, volume constraint $g(x)$ on $\Omega_\mathcal{I}$, force $f(x)$;
        \ENSURE  The solution of the nonlocal Poisson equation
        \STATE a Tuples array is needed for storing the contributions to $A$
        \STATE a Pairs array is needed for storing the contributions to $b$
        \FOR{each element $\mathcal{E}_n$ in $\mathcal{T}^h_{{\Omega}}$ parallelly}
        \STATE Generate the Gauss quadrature Points $P$ and weights $W$ for $\mathcal{E}_n$
        \FOR{each point $p \in P$, and weight $w\in W$}
        \STATE Generate the approximate ball $B_{\delta,h}(p)$ for $p$

        \FOR{each element $\mathcal{E}_m \in B_{\delta,h}(p)$ parallelly}
        \STATE //calculate the contributions of $\mathcal{E}_n$ and $\mathcal{E}_m$, and save them in Tuples and Pairs.

            \FOR{each pair of $\phi_i$ and $\phi_j$ that pertained to $\mathcal{E}_n$ and $\mathcal{E}_m$}
                \IF{$\widetilde{x}_i \in \Omega$ and $\widetilde{x}_j\in \Omega$ }
                    \IF{$\mathcal{E}_m \in \mathcal{T}^h_{{\Omega}}$}
                    \STATE $a \leftarrow$ $\int_{\mathcal{E}_m} (\phi_j(y) - \phi_j(p)) (\phi_i(y) - \phi_i(p) )\psi(p,y)dy$
                    \ELSIF{$\mathcal{E}_m \in \mathcal{T}^h_{{\Omega_{\mathcal{I}}}}$}
                    \STATE $a \leftarrow$$2 \phi_j(p) \phi_i(p)\int_{\mathcal{E}_m}\psi(p,y) dy$

                    \STATE $c \leftarrow$$2\phi_i(p)\int_{\mathcal{E}_m} g(y) \psi(p,y) dy$

                    \ENDIF
                    \STATE save $( i\textbf{ , }  j \textbf{ , }  a * w)$ to Tuples
                    \STATE save $( i\textbf{ , }  c * w)$ to Pairs

                    \ENDIF

                    \IF{$\widetilde{x}_i \in \Omega$ and $\widetilde{x}_j \in \partial\Omega$}
                    \IF{$\mathcal{E}_m \in \mathcal{T}^h_{{\Omega}}$}
                    \STATE $c \leftarrow$ $\int_{\mathcal{E}_m} (\phi_j(y) - \phi_j(p)) (\phi_i(y) - \phi_i(p) )\psi(p,y)dy$
                    \ELSIF{$\mathcal{E}_m  \in  \mathcal{T}^h_{{\Omega_{\mathcal{I}}}}$}
                    \STATE $c \leftarrow$ $2 \phi_j(p) \phi_i(p)\int_{\mathcal{E}_m}\psi(p,y) dy$
                    \ENDIF

                    \STATE save $(j\textbf{ , }  -c* w *g(p_j))$ to Pairs
                \ENDIF
            \ENDFOR

        \STATE save $( i \textbf{ , } \phi_i(p) * f(p) * w)$ to Pairs
        \ENDFOR
        \ENDFOR
        \ENDFOR
        \STATE set Sparse matrix A from Tuples
        \STATE set Right-hand side vector b from Pairs
        \STATE uses the conjugate gradient method to solve $Ax = b$
	\end{algorithmic}
\end{algorithm}

\subsection{Parallelizing}
Many steps in the traditional finite element algorithm can be decomposed into a series of vectorization operations, such as computing numerical integration in the matrix assembly process and Matrix-Vector Multiplication during the solution process. Hardware and software development in computer science provide many supports for these vectorization operations.
However, FEM for nonlocal problems cannot be easily decomposed into a series of vectorization operations, which brings  challenges to the parallelization of the assembly process. As we have described in the previous sections, the construction of approximate ball during the assembly process involves recursive breadth-first search and mesh modification, which can not be decomposed into a series of vectorization operations directly. Therefore, we take a different approach here to parallelize the assembly process of nonlocal problem's linear system.

Operations in the assembly process can be divided into two categories based on granularity. The first type is coarse-grained: a program is split into several relatively large tasks. Each task can perform more complex calculations. These coarse-grained tasks can be parallelized by the distributed and multi-core system. In our assembly process, constructing the approximate ball is one such coarse-grained task that involves lots of branches and unaligned memory access. The second type is fine-grained: a program is broken down into a number of relatively small tasks. There exists the same instruction sequence and few branches in these small tasks. During our assembly, computing the integrals over different simplices of the approximate ball is a fine-grained task. Such tasks are suitable for computers with SIMD architectures, such as vector arithmetic instructions (AVX SSE, etc.) and General Purpose Graphics Processing Units (GPGPU).

Our assembly process can be split into a series of operations. First, allocate the finite elements in $\mathcal{T}_\Omega$ dynamically to several threads in a load-balanced way. This process is corresponding to line 3 in Algorithm \ref{alg:main}. Second, for each element $\mathcal{E}_k$, the corresponding thread generates the quadrature point $p$ and the approximate ball $B_{\delta,h}(p)$. The information of tetrahedrons in $B_{\delta,h}(p)$ is prepared in an array. Third, for the elements in this array, we use vectorization operations to calculate the contributions to the linear system and return a tuple array.
This process is corresponding to line 7 in Algorithm \ref{alg:main}. Last, bitonic sort and reduce operations are iteratively invoked on the tuples arrays until all threads are terminated. At the end of the process, we have a stiffness matrix $A$ stored in the coordinate format (COO) and a right-hand side vector $F$.

In \cite{PS20},  the authors present an algorithm specifically designed to directly assemble  sparse matrices in a multi-threaded shared memory setting, which enables a fast and efficient solution for nonlocal problems. In addition, the asynchronous and task-based solution is implemented in \cite{DJ20}. For nonlocal problems, when computing on distributed CPU or distributed memory systems, the details of parallelism need to be further studied, which is necessary for solving large-scale nonlocal problems.


\section{Experiment and benchmark}\label{sec:numerical}
In this section, we provide some numerical experiments for further illustrations of the accuracy and efficiency of our algorithm. These numerical examples cover both 2D and 3D cases, and involve various types of ball approximation strategies. 
The accuracy of our algorithm is evaluated by the $L^2$-error, and efficiency is evaluated by the peer-to-peer (P2P) execution time.

For computing the convergence rates, we construct a variety of manufactured meshes with different mesh sizes, including uniform meshes and unstructured meshes. Examples of the mesh used in the numerical experiments are presented in Figure \ref{fig:mesh}. The minimum step size of these meshed varies gradually from $0.025$ to $0.00731$, while the average step size varies from $0.025$ to $0.00731$.


All the ball approximation strategies can be used for {correct} solution when $\delta \ll h$. 
In our experiment, we choose the $\delta$ to be 3 $\sim$ 7 times larger than the grid sizes.

\begin{figure}[htb]
\centering
\subfigure[uniform 2D mesh]{
\includegraphics[scale=0.05]{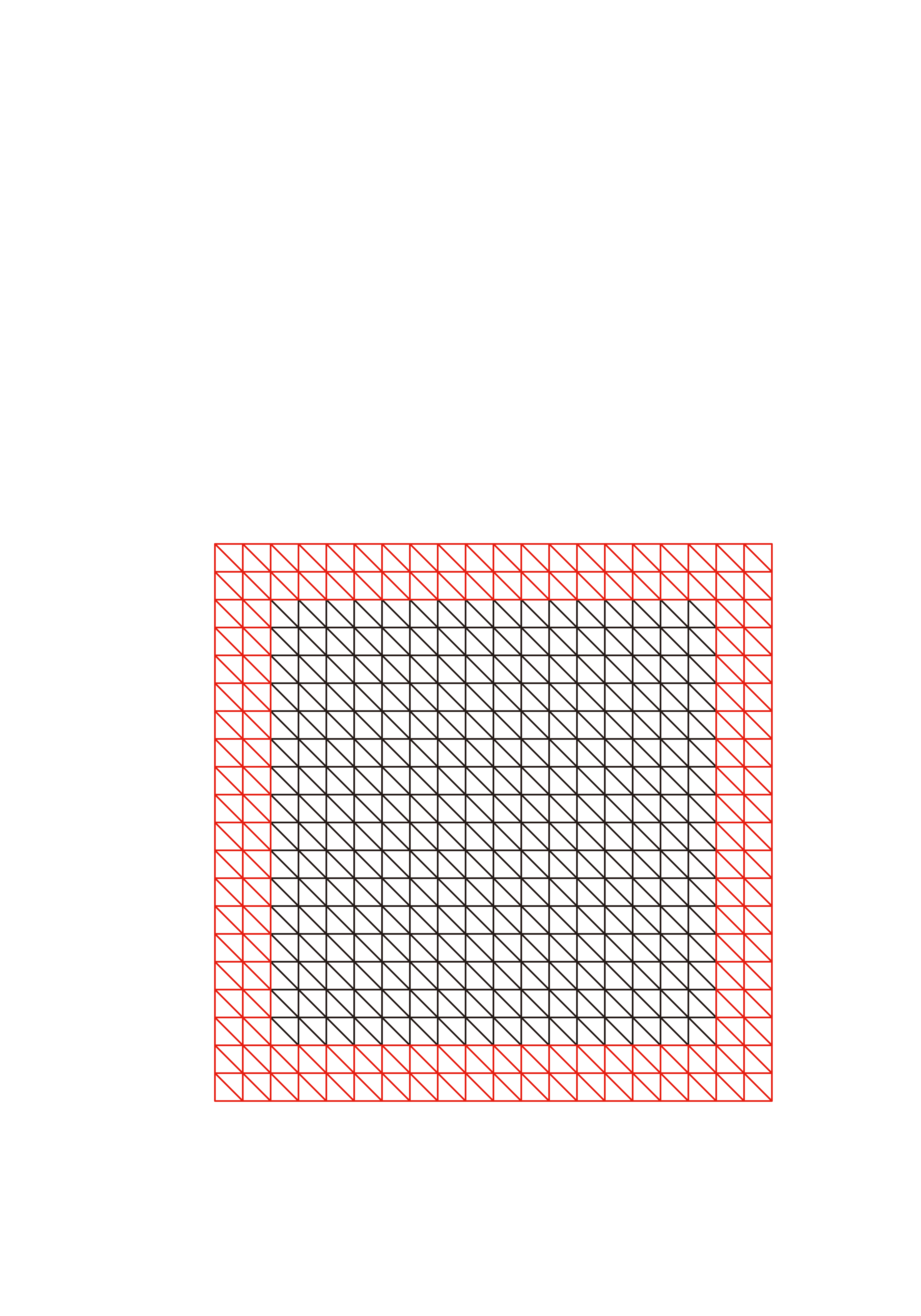}
}
\subfigure[unstructured 2D]{
\includegraphics[scale=0.05]{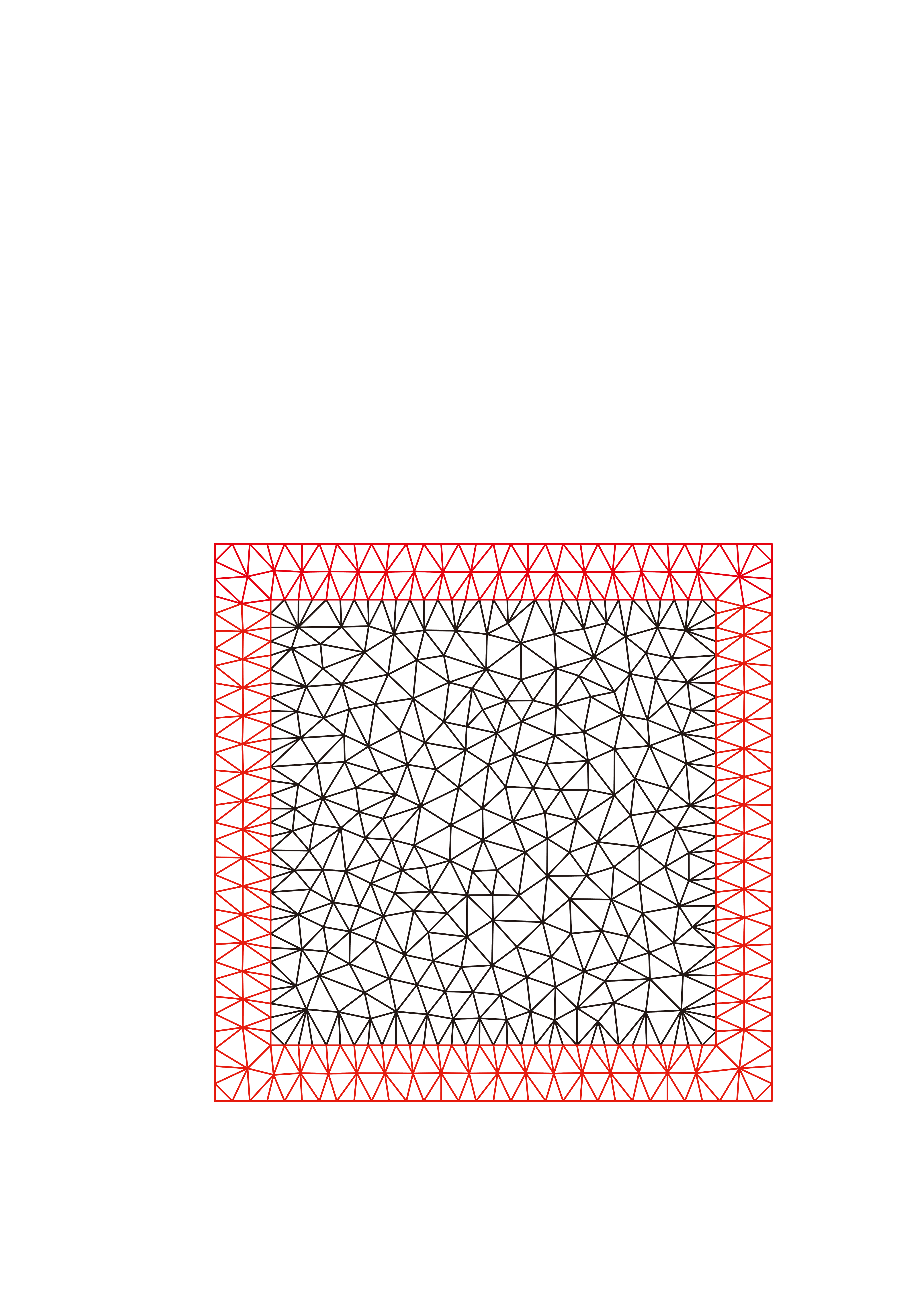}
}
\subfigure[unstructured 3D mesh]{
\includegraphics[scale=0.07]{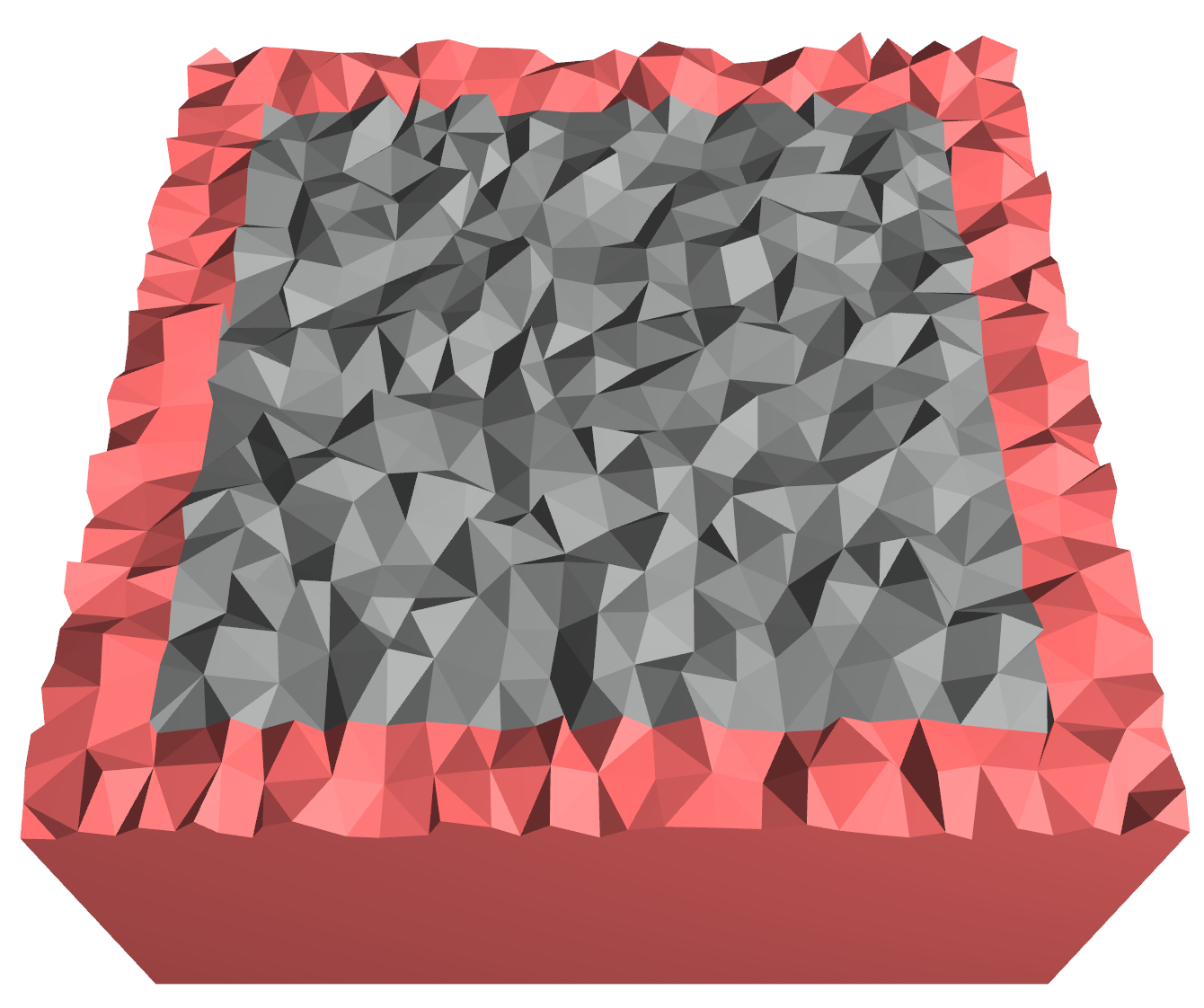}
}
\caption{ (a) Uniform mesh. (b) and (c) Unstructured meshes.}
\label{fig:mesh}
\end{figure}

\subsection{2D numerical experiments}
We take $\Omega = (0, 1)^2$, and $\gamma(x,y) = C \cdot \mathcal{X}_{B_\delta(x)}(y)$ with $C$ making sure $\int_{B_\delta(x)} \gamma(x,y) dy = 2$, and choose the manufactured solution given in \cite{d2021cookbook} as
$u(x) = x_1^2x_2 + x_2^2.$
The external force is computed by $f(x) = -\mathcal{L}u $ and the nonlocal Dirichlet volume constraint is taken as $g(x) = u(x) \text{ for }  x \in \Omega_\mathcal{I}$.

We choose the \emph{overlap, inside, barycenter, nocaps} approximation strategies in our 2D experiments and compare their accuracy and efficiency. 


\begin{table*}[htb]  \caption{ $L_2$ errors of 2-D numerical results }
 \label{table_2D}
\centering
 \begin{tabular}{lllllllll}
\toprule
   &dof &$K_\Omega$ &$h$ &inside &overlap &barycenter &nocaps  \\
\midrule

&1521   &5000    &0.0227  &4.06E-02&2.27E-02  &4.91E-04 &9.51E-04\\
&6241   &20000   &0.0114  &1.71E-02&1.33E-02  &2.28E-04 &1.69E-04  \\
&25281  &80000   &0.0057  &8.01E-03&7.00E-03  &5.84E-05 &3.70E-05 \\
&101761 &320000  &0.0028  &3.63E-03&3.24E-03   &1.47E-05& 9.26E-06 \\

\bottomrule
\end{tabular}
\end{table*}

\begin{figure}[ht]
\centering
\subfigure{
\includegraphics[scale=0.45]{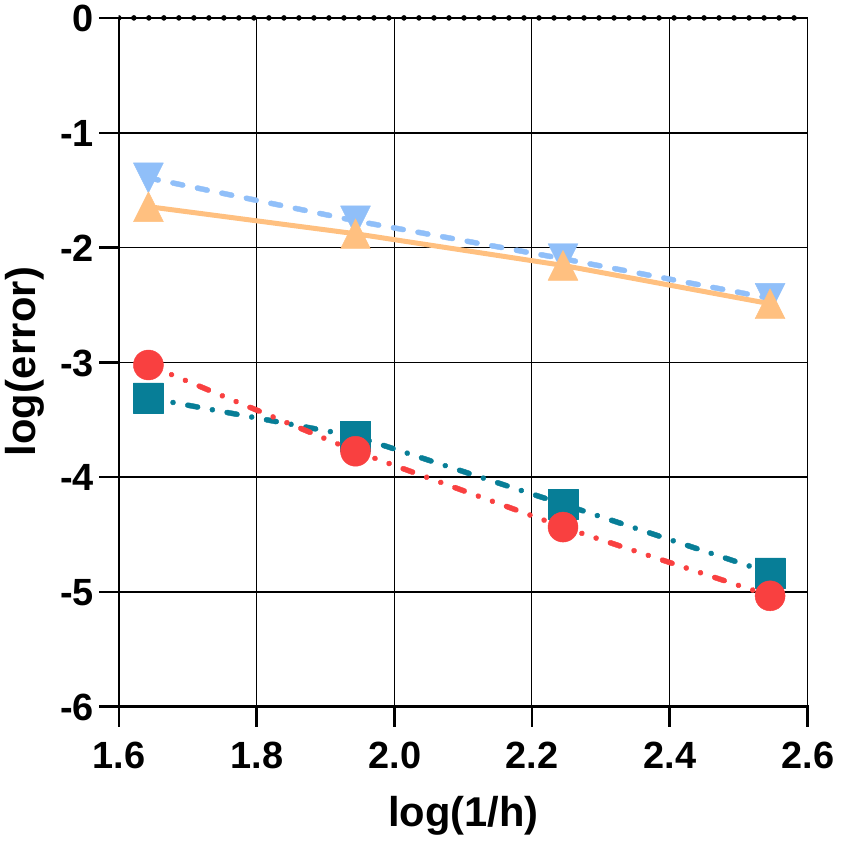}
}
\subfigure{
\includegraphics[scale=0.45]{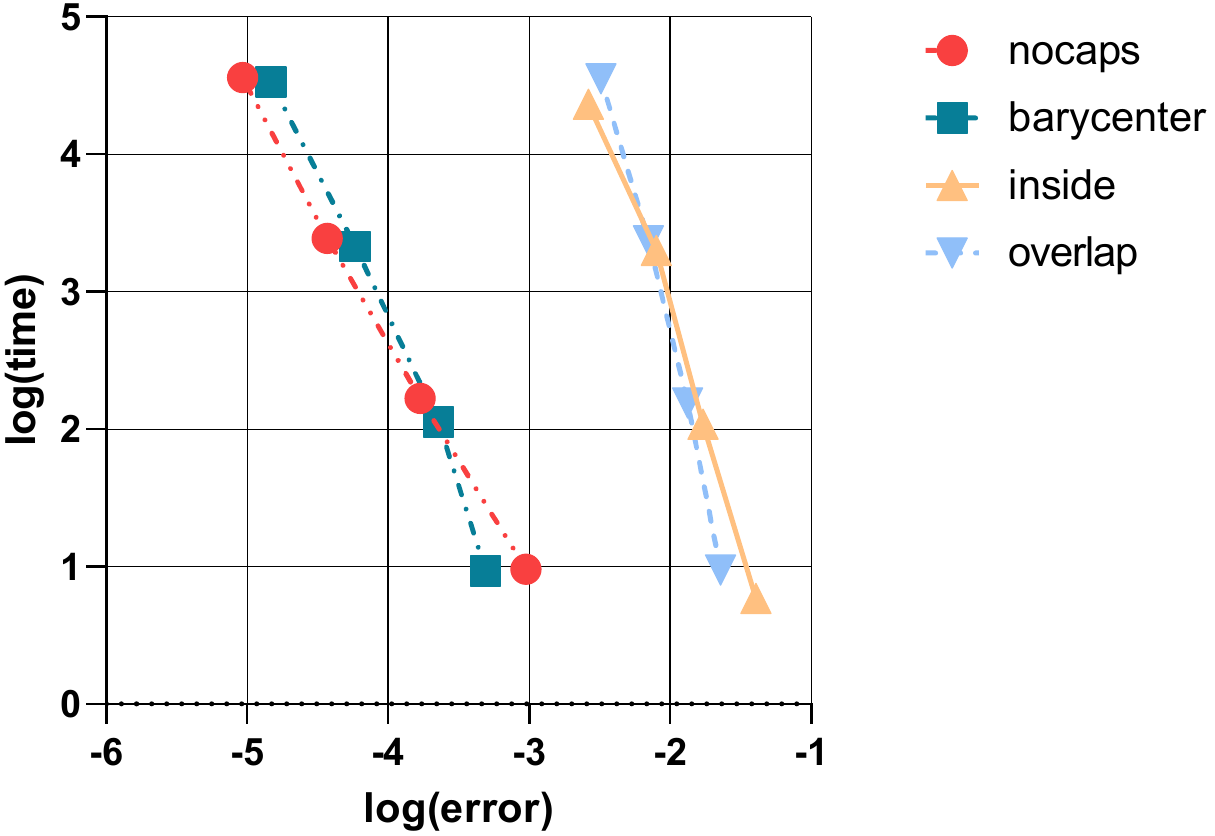}
}
\caption{Errors vs. average mesh sizes (left) and assembly times vs. errors (right)}
\label{fig:accur2D}
\end{figure}

We evaluate the convergence rates in $L^2$-norm. As predicted by the theory in Section 3, we observe  second-order convergence rates for  ``barycenter'' and ``nocaps'' ball approximations, and  first-order convergence rates for ``inside'' and ``overlap'' ball approximations in Table \ref{table_2D}. Figure \ref{fig:accur2D} plots the errors and assembly times, which shows the line lower left, the more effective the approximate strategy is.

\subsection{3D numerical experiments}
Similar to 2D case, we take $\Omega = (0, 1)^3$ and $\gamma(x,y) = C * \mathcal{X}_{B_\delta(x)}(y)$ with $\delta = 0.1$ and $C$ making sure $\int_{B_\delta(x)} \gamma(x,y) dy = 3$. The manufactured solution is taken as
$u(x) = (1-x_1)(1-x_2)(1-x_3)x_1x_2x_3.$  The \textit{overlap, inside, barycenter, nocaps, fullcaps} approximation strategies are investigated. 


\begin{figure}[ht]
\centering
\subfigure{
\includegraphics[scale=0.45]{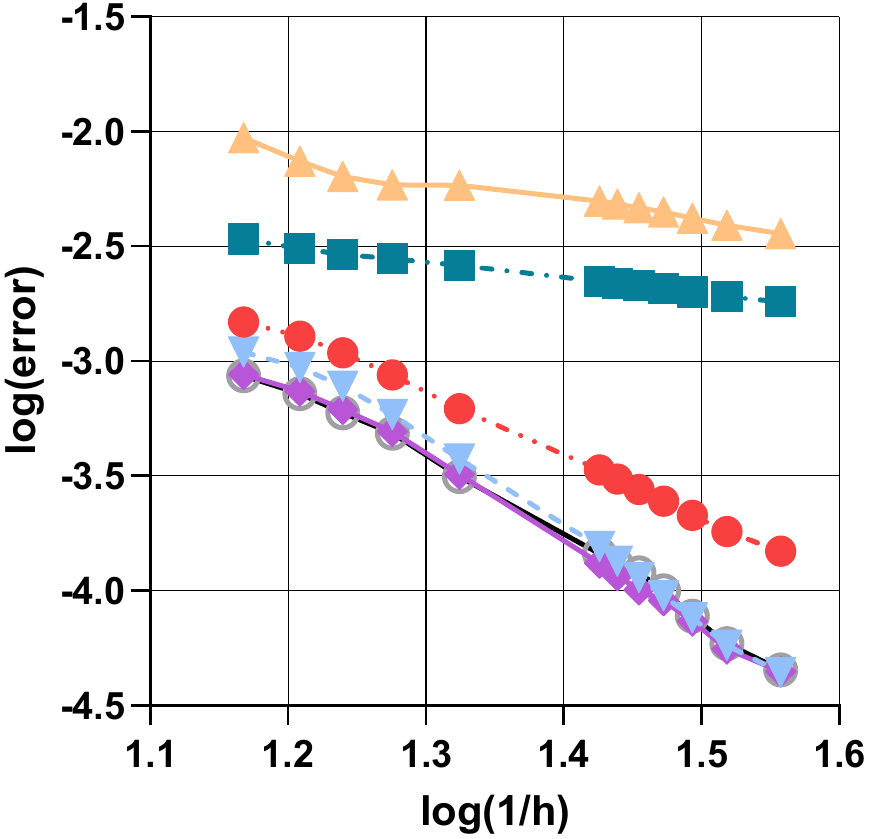}
}
\subfigure{
\includegraphics[scale=0.45]{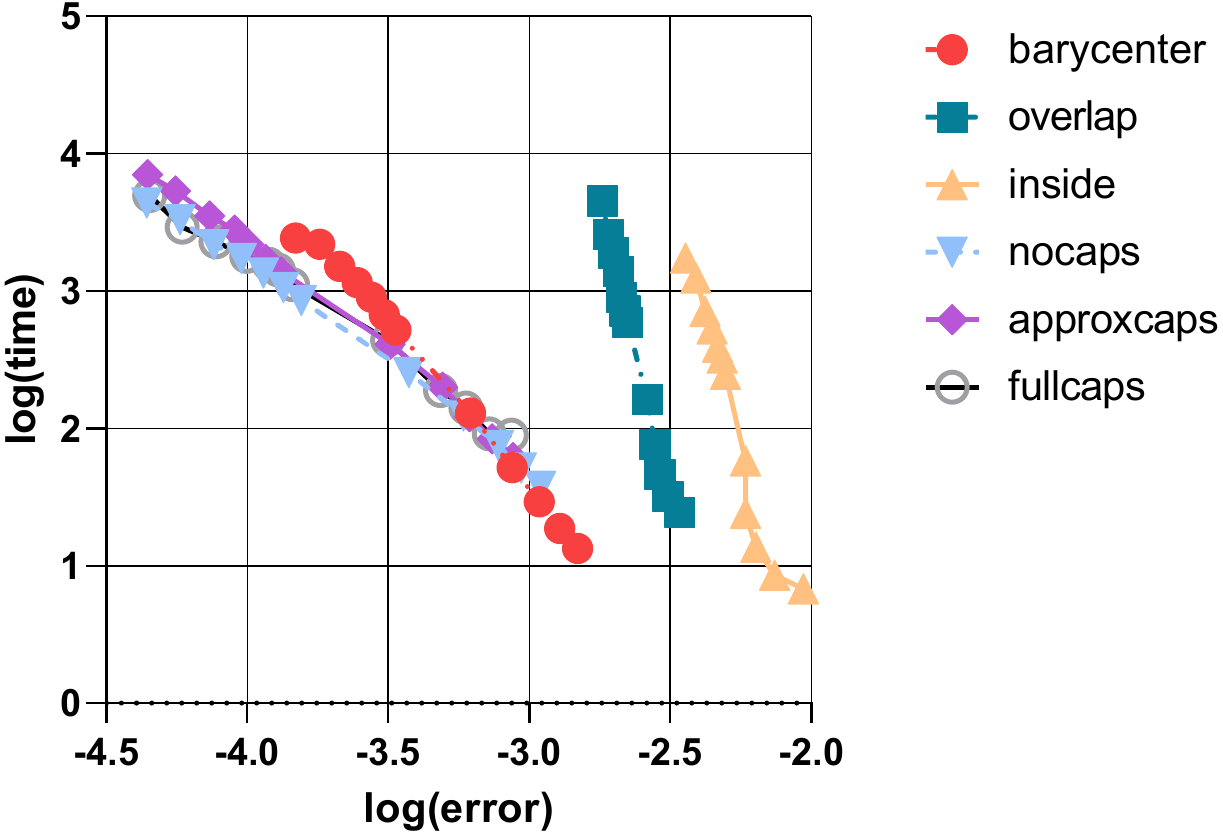}
}
\caption{Errors vs. average mesh sizes (left) and assembly times vs. errors (right) that are given in Tables \ref{table_3D_error} and \ref{table_3D_time} with the legend numbers corresponding to the numbering of columns in those tables.}
\label{fig:accur}
\end{figure}

Table \ref{table_3D_error} shows the convergence rates of the finite element approximation $u_{h,\sharp}$ with
$$\sharp \in \{overlap, inside, barycenter, nocaps, fullcaps\}.$$
One can observe a {less than 2-order convergence rates for the ``barycenter'', ``overlap'' and ``inside'' ball approximations}, and {3-order convergence rates for the ``nocaps'', ``approxcaps'' and ``fullcaps'' ball approximations}.

{In order to reveal the relationship between the numerical accuracy and the calculation cost of our algorithms, we define the ratio $\lambda$ to evaluate the effectiveness of convergency, i.e.
    $$\lambda_n = -\frac{\log(e_{n-1}/e_{n})}{\log(t_{n-1}/ t_n)}.$$
{Here $e_n$ is the $L^2$ error and $t_n$ is the execution time for the $n$th numerical experiment. } Figure \ref{fig:accur} and Table \ref{table_3D_time} show that the ``fullcaps'', ``nocaps'' and ``approxcaps'' approximation makes a better efficiency ratio ($\lambda \approx 0.667$) than the other approximations. This is, it requires a triple cost to double the accuracy when we take the ``fullcaps'', ``nocaps'' and ``approxcaps'' strategies, while the other strategies require more cost to double the accuracy. }

\begin{table*}[htb]  \caption{ $L_2$ errors of 3-D numerical results}
 \label{table_3D_error}
 \centering
 \begin{tabular}{llllllllll}
\toprule
  &$h_{avg}$ &barycenter &overlap &inside &nocaps &approxcaps &fullcaps  \\
\midrule
&0.0680   &1.48E-03&3.41E-03  &9.38E-03 &1.10E-03 &8.76E-04   &8.64E-04 \\
&0.0619    &1.28E-03&3.10E-03 &7.44E-03  &9.39E-04 &7.40E-04     &7.22E-04  \\
&0.0576    &1.09E-03&2.92E-03 &6.36E-03 &7.78E-04 &6.16E-04     &5.97E-04  \\
&0.0530    &8.70E-04&2.80E-03 &5.86E-03 &5.85E-04 &4.94E-04     &4.85E-04  \\
&0.0474   &6.20E-04&2.62E-03 &5.84E-03 &3.75E-04 &3.21E-04     &3.15E-04  \\
&0.0375   &3.37E-04&2.23E-03 &4.97E-03 &1.56E-04 &1.32E-04     &1.45E-04 \\
&0.0364 &3.06E-04&2.19E-03 &4.83E-03  &1.34E-04 &1.16E-04     &1.29E-04 \\
&0.0351 &2.76E-04&2.13E-03 &4.66E-03 &1.14E-04 &1.01E-04     &1.20E-04 \\
&0.0337  &2.45E-04&2.07E-03 &4.46E-03 &9.52E-05 &9.05E-05     &9.99E-05 \\
&0.0321  &2.13E-04&1.10E-03 &4.20E-03 &7.62E-05 &7.37E-05    &7.76E-05  \\
&0.0303  &1.81E-04&1.91E-03 &3.91E-03 &5.78E-05  &5.58E-05    &5.89E-05 \\
&0.0277   &1.49E-04&1.82E-03 &3.59E-03 &4.41E-05  &4.44E-05     &4.51E-05  \\
\bottomrule
\end{tabular}
\end{table*}

\begin{table*}[htb]  \caption{P2P time of 3-D numerical results}
 \label{table_3D_time}
 \centering
 \begin{tabular}{lllllllll}
\toprule
  &$h_{avg}$ &barycenter &overlap &inside &nocaps & approxcaps &fullcaps \\
\midrule
&0.0680  &13.42  &24.6       &6.79    &37.84     &61.42      &89.22 \\
&0.0619   &18.7   &31.92      &8.54   &51.48      &83.71      &91.31  \\
&0.0576   &29.25  &46.49      &13.67    &75.31   &121.61     &140.43  \\
&0.0530  &51.98  &75.88      &24.07   &123.22    &199.33     &188.77  \\
&0.0474   &128.94 &162.25     &57.90    &255.46   &409.98     &440.01 \\
&0.0375  &520.6  &589.55     &247.79   &852.50   &1379.48    &1113.79 \\
&0.0364 &671.97 &715.83     &324.21  &1060.67   &1699.53    &1785.25\\
&0.0351 &900.12 &892.45     &389.74   &1350.33  &2187.53    &1963.69 \\
&0.0337 &1153.1 &1387.56    &529.00   &1722.16  &2782.56    &2242.36 \\
&0.0321 &1507.2 &1890.53    &713.25  &2166.29   &3528.13    &2743.72 \\
&0.0303 &2192.72&2559.84    &1238.67  &3308.84  &5331.29    &3658.44\\
&0.0277 &2432.24&4514.34    &1732.44  &4351.03  &7014.38    &4954.51\\
\bottomrule
\end{tabular}
\end{table*}

\section{Conclusions}
\label{sec:conclusions}
In this paper,  a general framework {of FEM for solving} $n$-dimensional nonlocal modeling is discussed, and some measures are taken to alleviate some of the computational challenges brought by nonlocality. For example, we use ball approximation strategies to improve the accuracy of numerical integration and reduce the error of computation. We use the improved combinatorial map to express the topological structure of mesh and some iterators for fast neighborhood queries and dynamic mesh modifications. Besides, we provide a general algorithm for constructing the $n$-dimensional approximate ball, which alleviates the memory requirement and simplifies the operations in {ball approximation} from the engineering point of view. To increase the accuracy of the inner integration, we use combined geometry via boolean operations to represent the caps. Therefore, we proposed the new strategy named ``fullcaps'' to approximate the interaction domain and Monte Carlo sampling for the integration over fullcaps. The new ball approximation strategy ``fullcaps'' is superior to other approximations when $\delta \sim h$. In addition, we provide a method to parallelize the assembly process of finite element linear system, which can achieve a significant acceleration on modern muti-core computers and SIMD devices.

D'Elia et al. have given in \cite{approx} a 2D nonlocal problems' finite element solution procedure, as well as the quadrature rules, ball approximation strategies, and the corresponding error analysis. But,  there are few FEM implementations of higher dimensional nonlocal models up to now. Our work is the first concrete implementation for solving the 3D nonlocal problem on unstructured meshes with a parallel strategy. Higher dimensional nonlocal problems can be implemented by nD combinatorial map and corresponding topological iterators, with the same algorithm structure in 2D and 3D. Although the difficulty of implementation and the possible computational cost are high, it is still worth of developing an efficient implementation of FEM for solving nD nonlocal problems for its practical applications.



In the future, there are several points worth improving on our work. First, high precision quadrature rules for singular kernel functions are required for more engineering modeling. Second, there is no unified efficient algorithm in computational geometry for subdividing the simplex into polytopes in high-dimensional space.

\bibliographystyle{siamplain}
\bibliography{references}
\end{document}


\maketitle

\section{Introduction to combinatorial mapping}

First, we strictly introduce the theory of combinatorial mapping starting with the definition of the orientation of quasi-manifold \cite{G67}.

\begin{definition}[Quasi-manifold]\label{def:quasi-manifold}
Consider a nD simplicial complex
$$
\mathcal{K}:=\bigcup_{i=0}^{n}\bigcup_{k=1}^{\mathcal{S}_i}\{[c_k^i]\},
$$
where every $[c_k^i]$ is an $i$-dimensional simplex called $i$-cell. The number of all $i$-cell in $\mathcal{K}$ is $\mathcal{S}_i<\infty$. We also call $\mathcal{K}$ a cellular decomposition of the polytope $|\mathcal{K}|:=\bigcup_{i=0}^{n}\bigcup_{k=1}^{\mathcal{S}_i}[c_k^i]$ in the remainder of this paper. A nD quasi-manifold is a nD simplicial complex that satisfies:
\begin{itemize}
\item[(i)] For all $[c]\in\mathcal{K}$, there exists $n$-cell $[c^n]\in\mathcal{K}$ contains $[c]$ as a face of it, i.e. $[c]\preceq[c^n]$;
\item[(ii)] The $n$-cells of $\mathcal{K}$ are assembled along $(n-1)$-cells, such that each $(n-1)$-cell is incident at most two $n$-cells;
\item[(iii)] For any two $n$-cells $[c_1^n],[c^n]\in\mathcal{K}$, there are a series of $n$-cells and $(n-1)$-cells that are separated from each other:
    $$
    [c_1^n],[c_1^{n-1}],[c_2^n],[c_2^{n-1}],\cdots,[c_k^n],[c_k^{n-1}],[c^n]=[c_{k+1}^n],
    $$
    such that $[c_i^{n-1}]$ is the face of $[c_i^n]$ and $[c_{i+1}^n]$ for $1\leq i\leq k$.
\end{itemize}
\end{definition}
In fact, each $(n-1)$-cell is only able to be incident at one or two $n$-cell, and $(n-1)$-cell is incident at one $n$-cell only when it is on the boundary of polytope $|\mathcal{K}|$. Two $i$-cells are called adjacent if there exists an $(i-1)$-cell incident to both $i$-cells. The third item of Definition \ref{def:quasi-manifold} guarantees the $|\mathcal{K}|$ is a strongly connected manifold.

Any permutation of the vertices of the simplex $[c]\in\mathcal{K}$ determines an orientation of $[c]$ \cite{G67}. For every $n$-cell $[c^n]$, one of its orientation is denoted as ${\bf c}^n$, and the opposite orientation is $-{\bf c}^n$. For example, consider $[c^n]$ with vertices $v_0,v_1,\cdots,v_n$, the orientation that determined by the permutation of this vertices can be defined as ${\bf c}^n:=v_0v_1\cdots v_n$. For its $i$-th face $[c_i^{n-1}]$, we can define its orientation induced from the orientation ${\bf c}^n$ as ${\bf c}^{n-1}_i:=(-1)^iv_0v_1\cdots\hat{v}_i\cdots v_n$. After these preparations, the definition of orientable quasi-manifold is given as follows.

\begin{definition}[Orientable quasi-manifold \cite{G67}]\label{def:orient-quasi-manifold}
Consider the $\mathcal{K}$ defined in Definition \ref{def:quasi-manifold}. We call $\mathcal{K}$ an orientable quasi-manifold if it is able to define orientations on every $n$-cell of $\mathcal{K}$, which satisfies:
\begin{itemize}
\item If $[c_1^n]$ and $[c_2^n]$ are two adjacent $n$-cells, and $[c^{n-1}]$ is the $(n-1)$-cell that incident to $[c_1^n]$ and $[c_2^n]$, then the orientations on these two $n$-cells ${\bf c}_1^n$ and $-{\bf c}_2^n$ can induce the same orientation on $[c^{n-1}]$.
\end{itemize}
In this way, we are able to determine two different orientation of an orientable quasi-manifold $\mathcal{K}$. For any two $n$-cells $[c_1^n],[c_2^n]\in\mathcal{K}$, we call their orientations ${\bf c}_1^n$ and ${\bf c}_2^n$ are compatible if they are subject to the same orientation of $\mathcal{K}$, otherwise they are called incompatible.
\end{definition}
The closure of $\hat{\Omega}$ we consider in \eqref{eq:weak} is a subset of a $n$-dimensional ball, {and its triangulation $\mathcal{T}_{\Omega\cup\Omega_{\mathcal{I}}}^h$ is also an orientable quasi-manifold \cite{O06}.}

\begin{figure}[ht]
    \centering
    \subfigure[cell decomposition]{
    \includegraphics[scale=0.4]{fig/combinary_map/combatinal_map_4.pdf}
    \label{fig:cellDec}
    }
    \subfigure[incidence graph]{
    \includegraphics[scale=0.55]{fig/combinary_map/incidence.pdf}
    \label{fig:incidence}
    }
    \caption{(a) Shows a cell decomposition of a given geometry, where the cell $F_1$ is glued by $e_1$, $e_2$, $e_3$. (b) Shows the incidence relations about the cells in this geometry.}
\end{figure}

For a $n$-dimensional cellular decomposition, the incidence between cells can be expressed as an incidence graph, as the cells are incident to another if and only if the incidence graph exists a path starting from the cell of the higher dimension to the lower dimensional cell. For example, Figure \ref{fig:cellDec} shows a simple cellular decomposition of a 2D object and its incidence graph in Figure \ref{fig:incidence}. In this way, there exists several paths starting from the $n$-cells of the highest dimension to the $0$-cells in the bottom, such that each cell of the path (except the $n$-cell) is the face of the previous cell in the path.

{Now we are  able to define the cell-tuples in $n$-dimension \cite{cmap:2010} basing on incidence graph.} For a given quasi-manifold $\mathcal{K}$, a cell-tuple is an ordered sequence of cells:
$$d:=([c^n],[c^{n-1}],\dots,[c^{1}], [c^{0}]),$$
where $[c^i]$ is an $i$-cell of $\mathcal{K}$, and the cell-tuple $d$ is defined in the order of decreasing dimensions such that $[c^{i-1}]\prec [c^{i}]$ for all $0 < i \leq n$. 
For the sake of economy {of} expression, the mapping from dart $d$ to its $i$-dimensional cell is denoted by $C_i(d)$, and the mapping from cell $[c]$ to {one of} its darts is denoted by $D([c])$. {In implementation of our FEM, this dart can be chosen by user freely because we will only use $D(\cdot)$ as an initialization in our algorithm.}

For these cell-tuples corresponding to $\mathcal{K}$, the two cell-tuples are said to be $i$-adjacent if they share all but the $i$-dimensional cell. For example, in Figure \ref{fig:cellDec}, the $(F_1,e_2,v_2)$ and $(F_1,e_1,v_2)$ are $1$-adjacent, because they only have $1$-cell different. In fact, we can define a set of $n+1$ mappings $\{\alpha_i(\cdot)\}_{i=0}^{n}$ called partial perturbations. Intuitively, we first denote $\epsilon$ as a null and $B$ as the finite set that contains all cell tuples corresponding to $\mathcal{K}$. The partial permutation $\alpha_i$ related to the quasi-manifold $\mathcal{K}$ is a map from $B\cup \{\epsilon\}$ to $B\cup \{\epsilon\}$, defined based on the $i$-adjacency relations of the cell-tuples:
\begin{itemize}
\item $\alpha_i(\epsilon)=\epsilon$;
\item $\forall\, d\in B, \alpha_i(d)=d'$ if there exists a $d'$ that is $i$-adjacent to $d$, otherwise $\alpha_i(d)=\epsilon$.
\end{itemize}
{These $\{\alpha_i(\cdot)\}_{i=0}^{n}$ are uniquely defined on $B$. For a given partial permutation $f$, the inverse of it is defined as:}
\begin{itemize}
\item $f^{-1}(\epsilon)=\epsilon$;
\item $\forall\, d\in B, f^{-1}(d)=d'$ if there exists a $d'$ that satisfy $f(d')=d$, otherwise $f(d)=\epsilon$.
\end{itemize}

As what showed in \cite{cmap:2007,cmap:2010}, the $\{\alpha_i\}_{i=0}^{n-1}$ are bijections from $B$ to $B$, and $\alpha_n$ is also a bijection from $B$ to $B$ if $\mathcal{K}$ is a quasi-manifold without boundary. 
The $\{\alpha_i\}_{i=0}^{n}$ are all called partial involutions, because they all satisfy $\alpha_i^2(d)=d$ if and only if $\alpha_i(d)\neq\epsilon$. This means $\alpha_i=\alpha_i^{-1}$ by following the definition of inverse of partial permutations. With these preparations, an algebraic relationship can be defined over these cell-tuples, called generalized map (G-map):
\begin{definition}[Generalized map \cite{cmap:2010}]\label{def:G-map}
Assume $n\geq1$. For a given quasi-manifold $\mathcal{K}$, we define cell-tuples and their corresponding $\alpha_i(\cdot)$ according to the above rules. The G-map of $\mathcal{K}$ is an algebra $G=(B\cup\{\epsilon\},\alpha_0,\cdots,\alpha_n)$. It has the following properties:
\begin{itemize}
\item $B$ is the finite set that contains all cell tuples corresponding to $\mathcal{K}$;
\item $\forall\, i:0\leq i\leq n$, $\alpha_i$ is a partial involution on $B\cup\{\epsilon\}$;
\item $\forall\, i,j:0\leq i<i+2\leq j\leq n$, we have $\alpha_i$ and $\alpha_j$ convertable, which derives $\alpha_i\circ\alpha_j$ is a partial involution on $B\cup\{\epsilon\}$.
\end{itemize}
\end{definition}

\begin{figure}[htp!]
    \label{fig:maps}
    \centering
    \subfigure[generalized maps]{
    \includegraphics[scale=0.35]{fig/combinary_map/map_1.pdf}
    \label{fig:mapsa}
    }
    \subfigure[combinatorial map]{
    \includegraphics[scale=0.35]{fig/combinary_map/map_3.pdf}
    \label{fig:mapsb}
    }
    \caption{(a) shows the G-map, $\alpha_0$, $\alpha_1$, $\alpha_2$ corresponding to the cellular decomposition of Figure \ref{fig:cellDec}. (b) shows the C-map, $\beta_1$, $\beta_2$ by denoting a dart as (F,e,v). }
\end{figure}

In an implementation of G-map data structure, one can define $\epsilon$ as an empty pointer, or define a modified $\alpha_n$ in the following {way}:
$$
\dot{\alpha}_n(d)\left\{
\begin{aligned}
&= \alpha_n(d),\, &&\mbox{if $\alpha_n(d)\neq\epsilon$,}
\\
&= d,\, &&\mbox{if $\alpha_n(d)=\epsilon$,}
\end{aligned}
\right.
\qquad\mbox{for $d \in B$}.
$$
An example of using the modified $\alpha_n$ is shown in Figure \ref{fig:mapsa}, where a 2D geometric object is expressed by darts, and those partial permutations defined in generalized maps are expressed as colored lines between those darts. In this way, we are able to express a $n$-dimensional mesh by a G-map.

\begin{figure}[htp!]
    \label{fig:3dmaps}
    \centering
    \subfigure[]{
    \includegraphics[scale=0.4]{fig/combinary_map/3D/beta2.pdf}
    \label{fig:3dmapsa}
    }
    \subfigure[]{
    \includegraphics[scale=0.4]{fig/combinary_map/3D/beta1.pdf}
    \label{fig:3dmapsb}
    }
    \subfigure[]{
    \includegraphics[scale=0.4]{fig/combinary_map/3D/beta3.pdf}
    \label{fig:3dmapsc}
    }
    \caption{combinatorial map for an object consisting of two triangles. (a) the mapping $\beta_2$ associate two darts that have the common edge and volume but different faces. (b) the mapping $\beta_1$ associate two darts that have the common face and volume but different edges. (c) the mapping $\beta_3$ associate two darts that have the common face and edge but different volumes. }
\end{figure}

By Definition \ref{def:orient-quasi-manifold}, it is obvious that G-map can be used to represent both orientable or non-orientable quasi-manifolds. In fact, a dart $d\in B$ can determine an orientation in its corresponding $n$-cell $C_n(d)$:
\begin{definition}
For a given dart $d=([c^n],[c^{n-1}],\cdots,[c^1],[c^0])$, where $[c^0]\prec [c^1]\prec\cdots\prec [c^{n-1}]\prec [c^n]$, we can select a series of vertices by following the rule:
\begin{itemize}
\item For $0<i\leq n$, we define $a_i$ as the vertex of $[c^i]$ that doesn't {lie at} $[c^{i-1}]$, and define $a_0$ as the only vertex {lies at} $[c^0]$.
\end{itemize}
In this way, we call the orientation ${\bf c}^n:=a_na_{n-1}\cdots a_0$ as $[c^n]$'s orientation determined by $d$, say $\sigma(d)$.
\end{definition}
In the remainder of this paper, {we say $\sigma(d)=\sigma(d'),\,\forall\, d,d'\in B$ if they are compatible orientations of an orientable quasi-manifold}.

We state the G-map determined by an orientable quasi-manifold $\mathcal{K}$ can be viewed as a bipartite graph. Because for every given $d\in B$, we have $\sigma(d)=-\sigma(\alpha_i(d)),\,\forall\,0\leq i\leq n-1$, because $\alpha_i$ {is a rearrangement of
two vertices in} $C_n(d)$. And for the $\alpha_n$, the orientation of $C_{n-1}(d)$ induced by $\sigma(d)$ and $\sigma(\alpha_n(d))$ are the same, which also means $\sigma(d)=-\sigma(\alpha_n(d))$. So according to the Definition \ref{def:orient-quasi-manifold}, the two orientations of $\mathcal{K}$ can sperate $B$ into two disjoint and independent sets, and two darts belong to the same set if and only if they determine two compatible orientations {of} $\mathcal{K}$ respectively. Because all permutations $\{\alpha_i\}_{i=0}^n$ in G-map can change the orientation of $\mathcal{K}$, so the G-map can be viewed as a bipartite graph if $\{\alpha_i\}_{i=0}^n$ are the edges in the graph.

By taking advantage of the property that a G-map of an orientable quasi-manifold $\mathcal{K}$ can be viewed as a bipartite graph, the memory and operations we need in the implementation of this data structure can be halved \cite{cmap:2010} by using another data structure called combinatorial map (C-map).  For each pair of $0$-adjacent cell-tuples in G-map, they determine two different orientations {of} $\mathcal{K}$. {By} only keeping the cell-tuple that is compatible with a given orientation of  $\mathcal{K}$ is enough for expressing {the topology of} $\mathcal{K}$. The darts that determine a compatible orientation {of} $\mathcal{K}$ are defined as $R:=\{d'\in B|\sigma(d')=\sigma(d)\}$. We define the partial permutations
$$ \beta_i = \alpha_i \circ \alpha_0,\quad\forall\,0\leq i\leq n,$$
which connect two compatible darts from $0$-adjacent cell-tuples pairs. The properties of $\beta_i$ are given by the following lemma:
\begin{lemma}\label{lem:beta_property}
Assume $n\geq2$. For a given orientable quasi-manifold $\mathcal{K}$, it has a G-map $G=(B\cup\{\epsilon\},\alpha_0,\cdots,\alpha_n)$. By choosing $\forall\, d\in B$ ($d$ is definitely belonging to a $0$-adjacent cell-tuple pair), we fix $\sigma(d)$ as the orientation of $\mathcal{K}$, and define $R:=\{d'\in B|\sigma(d')=\sigma(d)\}$. By defining $\{\beta_i\}_{i=1}^n$ according to the above rules, we have $\beta_i$ satisfying the following properties:
\begin{enumerate}
\item $\beta_i\,(2\leq i\leq n-1)$ are involutions on $R$ without a fixed point, $\beta_n$ is a partial involution on $R\cup\{\epsilon\}$ without a fixed point (we don't consider fixed point such as $\beta_n(\epsilon)=\epsilon$);
\item $\forall\, i,j:1\leq i<i+2\leq j\leq n$, $\beta_i\circ\beta_j$ is a partial involution on $R\cup\{\epsilon\}$;
\item Consider the group generated by $\{\beta_i\}_{i=1}^n$, i.e. $S:=\langle\beta_1,\cdots,\beta_n\rangle$, it holds $S(d):=\langle\beta_1,\cdots,\beta_n\rangle(d)=R$;
\item $S=\langle\{\alpha_i\circ\alpha_j\}_{i,j=0}^n\rangle$.
\end{enumerate}
\end{lemma}
\begin{proof}
For item \emph{1}, since the G-map satisfies the second and third items in Definition \ref{def:G-map}, we have
$$
\beta_i^2=(\alpha_i\circ\alpha_0)^2=(\alpha_0\circ\alpha_i)^{-2}=Id,\quad\,\forall\, 2\leq i\leq n-1.
$$
The case of $\beta_n$ can be considered similarly, as we only need to consider $\epsilon$ in addition. For a given dart $d\in R$, and noting $2\leq n$, we have:
$$
\beta_n^2=(\alpha_n\circ\alpha_0)^2=(\alpha_n)^{2}\circ(\alpha_0)^{2}=(\alpha_n)^{2},
$$
which means $\beta_n$ is a partial involution. One can use similar method to prove item \emph{2} in Lemma \ref{lem:beta_property}.

For item \emph{3}, we have $S(d)\subset R$ holds obviously. So, we only need to prove $R\subset S(d)$, which can be derived from the item \emph{4} directly. {This is} because $\mathcal{K}$ is a connected complex, and every $\alpha_i$ can change its orientation to the opposite one. To prove item \emph{4}, we only need to notice:
$$
\alpha_i\circ\alpha_j=\alpha_i\circ\alpha_0\circ\alpha_0\circ\alpha_j=\beta_i\circ\beta_j^{-1}.
$$
The proof is completed.
\end{proof}

\section{The proof of the Theorem about orbit}

The {importance of defining orbit is to provide a tool that can efficiently find the darts associated with a given $i$-cell on the C-map}, for which we have the following theorem:
\begin{theorem}\label{thm:orbit}
Assume $n\geq 2$. Consider a given $n$-dimensional orientable quasi-manifold $\mathcal{K}$ and its C-map $C=(D,\beta_1,\cdots,\beta_n)$. Let $d=([c^n],[c^{n-1}],\dots,[c^{1}], [c^{0}])\in D$ be a dart, and $[c^i]$ is the $i$-cell of $d$. {If} the quasi-manifold $\mathcal{K}$ {satisfies} the following constraint:
\begin{itemize}
\item For any two $n$-cells $[c_1^n],[c^n]\in\mathcal{K}$, if there is an $p$-cell $[c^p]$ satisfying $[c^p]\prec[c_1^n],[c^n]$, then there is a series of $n$-cells and $(n-1)$-cells that are separated from each other:
    $$
    [c_1^n],[c_1^{n-1}],[c_2^n],[c_2^{n-1}],\cdots,[c_k^n],[c_k^{n-1}],[c^n]=[c_{k+1}^n],
    $$
    such that $[c_i^{n-1}]$ is the face of $[c_i^n]$ and $[c_{i+1}^n]$, and it satisfies $[c^p]\prec[c_i^{n-1}]$ for $1\leq i\leq k$.
\end{itemize}
Then it can be proved that
\begin{itemize}
\item $\{d'\in D|C_0(d')=[c^0]\}=\langle \{ \beta_i \circ \beta_j | \forall i,j: 1\leq i < j \leq n \} \rangle (d)\setminus\epsilon;$
\item $\{d'\in D|C_i(d')=[c^i]\}=\langle \beta_1,...,\beta_{i-1}, \beta_{i+1},...,\beta_n \rangle(d)\setminus\epsilon,\quad\forall\,1 \leq i \leq n.$
\end{itemize}
\end{theorem}

Before proving Theorem \ref{thm:orbit}, we first prove a lemma {about} how to permutate dart in a fixed simplex:
\begin{lemma}\label{lem:vertex_order}
Assume $n\geq 1$. Consider a given $n$-dimensional orientable quasi-manifold $\mathcal{K}$ and its G-map $G=(B,\alpha_0,\alpha_1,\cdots,\alpha_n)$. For two darts $d,d'\in B$ satisfying $C_p(d)=C_p(d'),\,\forall i\leq p\leq n$, there {exists} a {series} of partial permutations $\{\alpha_{n_k}\}_{k=1}^{z}$ that satisfy the following equation:
\begin{equation}\label{eq:dart_orbit_Gmap}
d'=\alpha_{n_{z}}\circ\alpha_{n_{z-1}}\circ\cdots\circ\alpha_{n_1}(d).
\end{equation}
{What's more, the series of partial permutations also satisfy $\alpha_p\notin\{\alpha_{n_k}\}_{k=1}^{z},\,\forall i\leq p\leq n$.}
\end{lemma}
\begin{proof}
{For simplicity, we denote the two darts as:}
\begin{align*}d'&=([c^n],[c^{n-1}]\dots,[c^{i+1}],[c^i],[{c^{i-1}}']\cdots,[{c^{1}}'], [{c^{0}}']),\\
d&=([c^n],[c^{n-1}]\dots,[c^{i+1}],[c^i],[c^{i-1}]\cdots,[c^{1}], [c^{0}]).
\end{align*}
{We will show that if we consider $d'=\alpha_{n_{z}}\circ\alpha_{n_{z-1}}\circ\cdots\circ\alpha_{n_1}(d)$, it can be viewed as a rearrangement of vertices of $[c^i]$.} We {will show how to choose} a series of $\{\alpha_{n_k}\}_{k=1}^{z}$ in the following process. {We can define} an {arrangement} of vertices of $[c^i]$ that are uniquely determined by the sub-tuple $([c^i],[c^{i-1}]\cdots,[c^{1}], [c^{0}])$:
$$
v_i\cdots v_2v_1v_0,
$$
where $v_p$ is the {vertex lies at} $[c_p]$ but not {at} $[c_{p-1}]$, and $v_0$ is the only vertex {lies at} $[c^0]$. {  It is not hard to find a rearrangement of these vertices:
$$
v_i\cdots v_{q+1},v_{q},\cdots v_2v_1v_0\rightarrow v_i\cdots v_{q},v_{q+1},\cdots v_2v_1v_0,\quad\forall 0\leq q\leq i-1,
$$
is related to one-step partial permutation $\alpha_q$.}
So, {there exists} a series of partial permutations $\{\alpha_{n_k}\}_{k=1}^{z}$ to guarantee
$$
d'=\alpha_{n_z}\circ\alpha_{n_{z-1}}\circ\cdots\circ\alpha_{n_{1}}(d),
$$
which also satisfies $\alpha_p\notin\{\alpha_{n_k}\}_{k=1}^{z},\,\forall i\leq p\leq n$. The proof is completed.
\end{proof}

{With} the help of Lemma \ref{lem:vertex_order}, we are able to consider {\bf the proof of Theorem \ref{thm:orbit}}.
\begin{proof}
{\bf Step1:} We first prove $\{d'\in D|C_0(d')=[c^0]\}=\langle \{ \beta_i \circ \beta_j | \forall i,j: 1\leq i < j \leq n \} \rangle (d)\setminus\epsilon$. When $1\leq j\leq n$, because $\beta_j$ and  $\beta_j^{-1}$ will not change the $0$-cell of {a} given dart, so we have
$$\langle \{ \beta_i \circ \beta_j | \forall i,j: 1\leq i < j \leq n \} \rangle(d)\setminus\epsilon\subset\{d'\in D|C_0(d')=[c^0]\}.$$

For any dart $d'\in D$ satisfying $C_0(d')=[c^0]$, we denote $[c_1^n]=C_n(d')$. Because of the extra constraint of the quasi-manifold $\mathcal{K}$ in the above, there {exists} a series of $n$-cells and $(n-1)$-cells that are separated from each other:
$$
[c_1^n],[c_1^{n-1}],[c_2^n],[c_2^{n-1}],\cdots,[c_k^n],[c_k^{n-1}],[c^n]=[c_{k+1}^n],
$$
such that $[c_i^{n-1}]$ is the face of $[c_i^n]$ and $[c_{i+1}^n]$ for $1\leq i\leq k$. Note we also have $[c^0]\prec[c_i^n],\,\forall1\leq i\leq k$. So, there is a series of darts $d_1,d_2,\cdots,d_{2k}$ satisfying
\begin{itemize}
\item $C_n(d_1)=[c_1^n]$, $C_n(d_{2k})=[c_{k+1}^n]$, $C_n(d_{2i})=C_n(d_{2i+1})=[c_{i+1}^n],\,\forall\,1\leq i\leq k-1$;
\item $C_{p}(d_{2i-1})=C_{p}(d_{2i}),\,\forall\,1\leq i\leq k,\,0\leq p\leq n-1$; specially, they satisfy $C_{n-1}(d_{2i-1})=C_{n-1}(d_{2i})=[c_{i}^{n-1}],\,\forall\,1\leq i\leq k$ and $\alpha_n(d_{2i-1})=d_{2i},\,\forall\,1\leq i\leq k$;
\item $C_0(d_i)=[c^0],\,\forall\,1\leq i\leq 2k$;
\item $C_n(d')=C_n(d_1)$, $C_n(d)=C_n(d_{2k})$. for simplicity, we denote $d_{0}:=d'$ and $d_{2k+1}:=d$.
\end{itemize}
{And there exists a series of partial permutations to guarantee}:
\begin{equation}\label{eq:d2d'}
d'=f_{1}\circ\alpha_{n}\circ f_2\circ\alpha_{n}\circ f_{3}\circ\cdots\circ\alpha_{n}\circ f_{k+1}(d),
\end{equation}
where $f_i\circ\cdots\circ\alpha_{n}\circ f_{k+1}(d)=d_{2i-2}$. {Here} $f_i\subset\langle\alpha_1,\alpha_2,\cdots,\alpha_{n}\rangle$ is a {partial} permutation, {and this} is because $C_0(d_{2i-2})=C_0(d_{2i-1})$ and $C_n(d_{2i-2})=C_n(d_{2i-1})$, so $f_i$ {related to an rearrangement of vertices} in the $n$-cell $[c^n_{i}]$ without {rearranging $v_0$}. And by following similar procedure in the proof of Lemma \ref{lem:vertex_order}, we can prove $f_i\subset\langle\alpha_1,\alpha_2,\cdots,\alpha_{n-1}\rangle$. At last, because of $\sigma(d')=\sigma(d)$, {so the total} number of partial permutations in \eqref{eq:d2d'} is even. {These partial permutations are denoted as} $\{\alpha_{n_k}\}_{k=1}^{2z}$, which guarantee the following equation:
\begin{equation}
\begin{aligned}\label{eq:dart_orbit0_Cmap}
d'=&\alpha_{n_{2z}}\circ\alpha_{n_{2z-1}}\circ\cdots\circ\alpha_{n_1}(d)\\
=&(\alpha_{n_{2z}}\circ\alpha_0\circ\alpha_0\circ\alpha_{n_{2z-1}})\circ\cdots(\alpha_{n_2}\circ\alpha_0\circ\alpha_0\circ\alpha_{n_1})(d)\\
=&(\beta_{n_{2z}}\circ\beta_{n_{2z-1}}^{-1})\circ\cdots(\beta_{n_2}\circ\beta_{n_1}^{-1})(d).
\end{aligned}
\end{equation}
{Here it also satisfies $\alpha_0\notin\{\alpha_{n_k}\}_{k=1}^{2z}$.}

Now, because we already know $d,d'\neq\epsilon$, we only need to prove $\beta_{n_2}\circ\beta_{n_1}^{-1}(d)\in\langle \{ \beta_i \circ \beta_j | \forall i,j: 1\leq i < j \leq n \} \rangle(d)$, where $\forall\,1\leq n_1,n_2\leq n$ and $\forall\,d\in \{d''\in D|\alpha_n(d'')\neq\epsilon\}$. First, for the partial involution $\beta_n$ and $\forall\,d\in \{d''\in D|\alpha_n(d'')\neq\epsilon\}$, we have
$$
\beta_n^{-1}(d)=(\alpha_n\circ\alpha_0)^{-1}(d)=\alpha_0^{-1}\circ\alpha_n^{-1}(d)=\alpha_0\circ\alpha_n(d)=\beta_n(d).
$$
{And because} we have proved in Lemma \ref{lem:beta_property} that $\beta_i,\,\forall2\leq i\leq n-1$ are involutions on $D$, so:
\begin{equation}\label{eq:partial_involution}
\beta_i^{-1}(d)=\beta_i(d),\,\forall\,d\in \{d''\in D|\alpha_n(d'')\neq\epsilon\},\,2\leq i\leq n.
\end{equation}

When $n_1=n_2\neq1$, because $\beta_{n_2}$ and $\beta_{n_1}$ are two partial involutions, we have $\beta_{n_2}\circ\beta_{n_1}^{-1}(d)=d$ according to \eqref{eq:partial_involution}; when $n_1=n_2=1$, because $\beta_1$ is a permutation, we have $\beta_1\circ\beta_1^{-1}=Id$.

When $1\leq n_1<n_2$. {According to} \eqref{eq:partial_involution}, we have $$\beta_{n_2}\circ\beta_{n_1}^{-1}(d)=\alpha_{n_2}\circ\alpha_{n_1}(d)=(\alpha_{n_1}\circ\alpha_{n_2})^{-1}(d)=(\beta_{n_1}\circ\beta_{n_2}^{-1})^{-1}(d)=(\beta_{n_1}\circ\beta_{n_2})^{-1}(d),$$ because $\beta_{n_2}$ is a partial involutions.

When $1\leq n_2<n_1$. {According to} \eqref{eq:partial_involution}, we have $\beta_{n_2}\circ\beta_{n_1}^{-1}(d)=\beta_{n_2}\circ\beta_{n_1}(d)$, because $\beta_{n_1}$ is a partial involutions.

The proof of $$\langle \{ \beta_i \circ \beta_j | \forall i,j: 1\leq i < j \leq n \} \rangle(d)\setminus\epsilon=\{d'\in D|C_0(d')=[c_0]\}$$
is completed. {The {\bf Step1} is finished.}

{\bf Step2:} When $1 \leq i\neq j \leq n$, because $\beta_j(\cdot)$ won't change the $i$-cell of the given dart, we have
$$\langle \beta_1,...,\beta_{i-1}, \beta_{i+1},...,\beta_n \rangle(d)\setminus\epsilon\subset\{d'\in D|C_i(d')=[c_i]\}.$$

For any dart $d'\in D$ satisfies $C_i(d')=[c_i]$, we state there must be an even number of partial permutations $\{\alpha_{n_k}\}_{k=1}^{2z}$ that satisfy $\alpha_i\notin\{\alpha_{n_k}\}_{k=1}^{2z}$, which guarantee the following equation
\begin{equation}\label{eq:dart_orbit_Gmap}
d'=\alpha_{n_{2z}}\circ\alpha_{n_{2z-1}}\circ\cdots\circ\alpha_{n_1}(d).
\end{equation}
In fact, if we denote $$d'=([{c^n}'],[{c^{n-1}}']\dots,[{c^{i+1}}'],[c^i],[{c^{i-1}}']\cdots,[{c^{1}}'], [{c^{0}}']),$$ $$d''=([{c^n}'],[{c^{n-1}}']\dots,[{c^{i+1}}'],[c^i],[c^{i-1}]\cdots,[c^{1}], [c^{0}]),$$
$$d=([c^n],[c^{n-1}]\dots,[c^{i+1}],[c^i],[c^{i-1}]\cdots,[c^{1}], [c^{0}]),$$
we can choose $\{\alpha_{n_k}\}_{k=1}^{2z}$ in the following process. Clearly, if we consider $d'\rightarrow d''$, according to Lemma \ref{lem:vertex_order}, {there exist} a series of partial permutations $\{\alpha_{n_k}\}_{k=1}^{p}$ guarantees
$$
d''=\alpha_{n_p}\circ\alpha_{n_{p-1}}\circ\cdots\circ\alpha_{n_{1}}(d').
$$
{This series also} satisfies $\alpha_i\notin\{\alpha_{n_k}\}_{k=1}^{p}$. By similar methods in {\bf Step1}, one can also {choose} a series of partial permutations $\{\alpha_{n_k}\}_{k=p+1}^{q}$ {to} guarantee
$$
d=\alpha_{n_{q}}\circ\alpha_{n_{q-1}}\circ\cdots\circ\alpha_{n_{p+1}}(d'').
$$
{This series} also satisfy $\alpha_i\notin\{\alpha_{n_k}\}_{k=p+1}^{q}$. Noting $\sigma(d')=\sigma(d)$, so $q$ is an even number, thus \eqref{eq:dart_orbit_Gmap} is proved.

So, by adding $\alpha_0$ to the \eqref{eq:dart_orbit_Gmap}, we have:
\begin{equation}
\begin{aligned}\label{eq:dart_orbit_Cmap}
d'=&\alpha_{n_{2z}}\circ\alpha_{n_{2z-1}}\circ\cdots\circ\alpha_{n_1}(d)\\
=&(\alpha_{n_{2z}}\circ\alpha_0\circ\alpha_0\circ\alpha_{n_{2z-1}})\circ\cdots(\alpha_{n_2}\circ\alpha_0\circ\alpha_0\circ\alpha_{n_1})(d)\\
=&(\beta_{n_{2z}}\circ\beta_{n_{2z-1}}^{-1})\circ\cdots(\beta_{n_2}\circ\beta_{n_1}^{-1})(d).
\end{aligned}
\end{equation}
By eliminating those $\beta_0=Id$ in the right of \eqref{eq:dart_orbit_Cmap}, we prove
$$\{d'\in D|C_i(d')=[c_i]\}=\langle \beta_1,...,\beta_{i-1}, \beta_{i+1},...,\beta_n \rangle(d)\setminus\epsilon.$$
The proof is completed.
\end{proof}

\section{Experiment specification}
This program is developed in Windows environment. The runtime environment and hardware specifications are given in Table \ref{table_time}. C++ and CUDA programming languages are used for code development.
\begin{table}[htb]  \caption{platform specification}
 \label{table_time}
 \centering
 \begin{tabular}{ll}
\toprule
\midrule
    Operating System  &Windows 11 \\
    \midrule
    CPU &Intel i7-8700K 6C12T 3.7Ghz \\
    \midrule
    Main Memory &128GB \\
\bottomrule

\end{tabular}
\end{table}

\bibliographystyle{siamplain}
\bibliography{references}

%% file: ex_shared.tex
\usepackage{lipsum}
\usepackage{amsfonts}
\usepackage{graphicx}
\usepackage{epstopdf}
\usepackage{algorithmic}
\usepackage{booktabs}
\usepackage{subfigure}
\ifpdf
  \DeclareGraphicsExtensions{.eps,.pdf,.png,.jpg}
\else
  \DeclareGraphicsExtensions{.eps}
\fi

\newsiamremark{remark}{Remark}
\newsiamremark{hypothesis}{Hypothesis}
\crefname{hypothesis}{Hypothesis}{Hypotheses}
\newsiamthm{claim}{Claim}

\headers{3D FEM for nonlocal problem}{G. Chen, Y. Ma, and J. Zhang}

\title{High performance implementation of 3D FEM for nonlocal Poisson problem with different ball approximation strategies}

\author{Gengjian Chen  \thanks {School of Mathematics and Statistics, Wuhan University, Wuhan 430072, China.}
\and Yuheng Ma
 \thanks {School of Mathematics and Statistics, Wuhan University, Wuhan 430072, China.}
\and Jiwei Zhang
	\thanks{School of Mathematics and Statistics, and Hubei Key Laboratory of Computational Science, Wuhan University, Wuhan 430072, China (jiweizhang@whu.edu.cn). }}

\usepackage{amsopn}


%% file: main.bbl
\begin{thebibliography}{10}

\bibitem{AS64}
M.~Abramowitz, Irene~A. Stegun, and David~M. Miller.
\newblock Handbook of mathematical functions with formulas, graphs and
  mathematical tables (national bureau of standards applied mathematics series
  no. 55).
\newblock {\em J. Appl. Mech.}, 32:239--239, 1964.

\bibitem{AC21}
Eugenio Aulisa, Giacomo Capodaglio, Andrea Chierici, and Marta D'Elia.
\newblock Efficient quadrature rules for finite element discretizations of
  nonlocal equations.
\newblock {\em Numer. Methods Partial Differential Equations}, 2021.

\bibitem{phase:bates}
Peter~W. Bates and Adam Chmaj.
\newblock An integrodifferential model for phase transitions: stationary
  solutions in higher space dimensions.
\newblock {\em J. Stat. Phys.}, 95(5):1119--1139, 1999.

\bibitem{BH12}
Florin Bobaru and Wenke Hu.
\newblock The meaning, selection, and use of the peridynamic horizon and its
  relation to crack branching in brittle materials.
\newblock {\em Int. J. Fract.}, 176:215--222, 2012.

\bibitem{BS94}
Susanne~C. Brenner and Leighton~R. Scott.
\newblock {\em The Mathematical Theory of Finite Element Methods}.
\newblock 1994.

\bibitem{stochastic:burch}
Nathanial Burch, Marta D'Elia, and Richard~B. Lehoucq.
\newblock The exit-time problem for a markov jump process.
\newblock {\em Eur. Phys. J.: Spec. Top.}, 223(14):3257--3271, 2014.

\bibitem{cmap:2010}
Guillaume Damiand.
\newblock {\em Contributions aux cartes combinatoires et cartes
  g{\'e}n{\'e}ralis{\'e}es: Simplification, mod{\`e}les, invariants
  topologiques et applications}.
\newblock PhD thesis, INSA de Lyon, 2010.

\bibitem{DL14}
Guillaume Damiand and Pascal Lienhardt.
\newblock {\em Combinatorial Maps: Efficient Data Structures for Computer
  Graphics and Image Processing (1st ed.).}
\newblock A K Peters/CRC Press, 2014.

\bibitem{phase:delgoshaie}
Amir~H. Delgoshaie, Daniel~W. Meyer, Patrick Jenny, and Hamdi~A. Tchelepi.
\newblock Non-local formulation for multiscale flow in porous media.
\newblock {\em J. Hydrol.}, 531:649--654, 2015.

\bibitem{approx}
Marta D'Elia, Qiang Du, Christian Glusa, Max Gunzburger, Xiaochuan Tian, and
  Zhi Zhou.
\newblock Numerical methods for nonlocal and fractional models.
\newblock {\em Acta Numer.}, 29:1--124, 2020.

\bibitem{stochastic:d2017}
Marta D'Elia, Qiang Du, Max Gunzburger, and Richard Lehoucq.
\newblock Nonlocal convection-diffusion problems on bounded domains and
  finite-range jump processes.
\newblock {\em Comput. Methods Appl. Math.}, 17(4):707--722, 2017.

\bibitem{DG21}
Marta D'Elia, Mamikon~A. Gulian, George~Em Karniadakis, and Hayley Olson.
\newblock A unified theory of fractional nonlocal and weighted nonlocal vector
  calculus.
\newblock {\em Proposed for presentation at the One Nonlocal World}, 2021.

\bibitem{d2021cookbook}
Marta D'Elia, Max Gunzburger, and Christian Vollmann.
\newblock A cookbook for approximating euclidean balls and for quadrature rules
  in finite element methods for nonlocal problems.
\newblock {\em Math. Models. Methods. Appl. Sci.}, 31(08):1505--1567, 2021.

\bibitem{peridynamics:D'ELIA}
Marta D'Elia, Mauro Perego, Pavel Bochev, and David Littlewood.
\newblock A coupling strategy for nonlocal and local diffusion models with
  mixed volume constraints and boundary conditions.
\newblock {\em Computers \& Mathematics with Applications}, 71(11):2218--2230,
  2016.
\newblock Proceedings of the conference on Advances in Scientific Computing and
  Applied Mathematics. A special issue in honor of Max Gunzburger’s 70th
  birthday.

\bibitem{DJ20}
Patrick Diehl, Prashant~K. Jha, Hartmut Kaiser, Robert Lipton, and Martin
  L{\'e}vesque.
\newblock Implementation of peridynamics utilizing hpx - the c++ standard
  library for parallelism and concurrency.
\newblock {\em J. Open. Source. Softw.}, 5:2352, 2020.

\bibitem{DW15}
Ning Du, Hong Wang, and Che Wang.
\newblock A fast method for a generalized nonlocal elastic model.
\newblock {\em J. Comput. Phys}, 297:72--83, 2015.

\bibitem{D2019}
Qiang Du.
\newblock {\em Nonlocal Modeling, Analysis, and Computation}.
\newblock SIAM, Philadelphia, PA, USA, 1st edition, 2019.

\bibitem{DG13}
Qiang Du, Max Gunzburger, R.~B. Lehoucq, and Kun Zhou.
\newblock A nonlocal vector calculus, nonlocal volume-constrained problems, ans
  nonlocal balance laws.
\newblock {\em Math. Models. Methods. Appl. Sci.}, 23(03):493--540, 2013.

\bibitem{difference:du2019}
Qiang Du, Yunzhe Tao, Xiaochuan Tian, and Jiang Yang.
\newblock Asymptotically compatible discretization of multidimensional nonlocal
  diffusion models and approximation of nonlocal green’s functions.
\newblock {\em IMA J. Numer. Anal.}, 39(2):607--625, 2019.

\bibitem{Yin}
Qiang Du, Hehu Xie, and Xiaobo Yin.
\newblock On the convergence to local limit of nonlocal models with
  approximated interaction neighborhoods.
\newblock {\em SIAM J. Numer. Anal.}, 60(4):2046--2068, 2022.

\bibitem{element:du2019conforming}
Qiang Du and Xiaobo Yin.
\newblock A conforming dg method for linear nonlocal models with integrable
  kernels.
\newblock {\em J. Sci. Comput.}, 80(3):1913--1935, 2019.

\bibitem{peridynamics:du}
Qiang Du and Kun Zhou.
\newblock Mathematical analysis for the peridynamic nonlocal continuum theory.
\newblock {\em ESAIM: Math. Model. Numer. Anal.}, 45(2):217--234, 2011.

\bibitem{phase:fife}
Paul Fife.
\newblock Some nonclassical trends in parabolic and parabolic-like evolutions.
\newblock {\em Trends in nonlinear analysis}, pages 153--191, 2003.

\bibitem{image:gilboa}
Guy Gilboa and Stanley Osher.
\newblock Nonlocal operators with applications to image processing.
\newblock {\em Multiscale Model. Simul.}, 7(3):1005--1028, 2009.

\bibitem{G67}
Marvin~J. Greenberg.
\newblock {\em Lectures on Algebraic topology}.
\newblock W. A. Benjamin, UNew York, 1967.

\bibitem{GL10}
Max~D. Gunzburger and Richard~B. Lehoucq.
\newblock A nonlocal vector calculus with application to nonlocal boundary
  value problems.
\newblock {\em Multiscale Model. Simul.}, 8:1581--1598, 2010.

\bibitem{peridynamics:HA}
Youn~Doh Ha and Florin Bobaru.
\newblock Characteristics of dynamic brittle fracture captured with
  peridynamics.
\newblock {\em Eng. Fract. Mech.}, 78(6):1156--1168, 2011.

\bibitem{JW21}
Siavash Jafarzadeh, Longzhen Wang, Adam Larios, and Florin Bobaru.
\newblock A fast convolution-based method for peridynamic transient diffusion
  in arbitrary domains.
\newblock {\em Comput. Methods Appl. Mech. Engrg.}, 375:113633, 2021.

\bibitem{HH13}
David C.~Handscomb John M.~Hammersley.
\newblock {\em Monte Carlo Methods}.
\newblock Springer Science \& Business Media, 2013.

\bibitem{combinatorialmaps}
Pierre Kraemer, Lionel Untereiner, Thomas Jund, Sylvain Thery, and David
  Cazier.
\newblock Cgogn: N-dimensional meshes with combinatorial maps.
\newblock In {\em Proceedings of the 22nd International Meshing Roundtable},
  pages 485--503. Springer, 2014.

\bibitem{meshfree:lehoucq2016radial}
Richard~B. Lehoucq and Stephen~T. Rowe.
\newblock A radial basis function galerkin method for inhomogeneous nonlocal
  diffusion.
\newblock {\em Comput. Methods Appl. Mech. Engrg.}, 299:366--380, 2016.

\bibitem{meshfree:leng2021asymptotically}
Yu~Leng, Xiaochuan Tian, Nathaniel Trask, and John~T Foster.
\newblock Asymptotically compatible reproducing kernel collocation and meshfree
  integration for nonlocal diffusion.
\newblock {\em SIAM J. Numer. Anal.}, 59(1):88--118, 2021.

\bibitem{cmap:lienhardt1991}
Pascal Lienhardt.
\newblock Topological models for boundary representation: a comparison with
  n-dimensional generalized maps.
\newblock {\em Comput. Aided Des.}, 23(1):59--82, 1991.

\bibitem{cmap:lienhardt1994}
Pascal Lienhardt.
\newblock N-dimensional generalized combinatorial maps and cellular
  quasi-manifolds.
\newblock {\em Int. J. Comput. Geom. Appl.}, 4(03):275--324, 1994.

\bibitem{LC18}
Huan Liu, Aijie Cheng, and Hong Wang.
\newblock A fast discontinuous galerkin method for a bond-based linear
  peridynamic model discretized on a locally refined composite mesh.
\newblock {\em J. Sci. Comput.}, 76:913--942, 2018.

\bibitem{image:lou}
Yifei Lou, Xiaoqun Zhang, Stanley Osher, and Andrea Bertozzi.
\newblock Image recovery via nonlocal operators.
\newblock {\em J. Sci. Comput.}, 42(2):185--197, 2010.

\bibitem{MC16}
Zhiping Mao, Sheng Chen, and Jie Shen.
\newblock Efficient and accurate spectral method using generalized jacobi
  functions for solving riesz fractional differential equations.
\newblock {\em Appl. Numer. Math.}, 106:165--181, 2016.

\bibitem{MD13}
Tadele Mengesha and Qiang Du.
\newblock Analysis of a scalar peridynamic model with a sign changing kernel.
\newblock volume~18, pages 1415--1437, 2013.

\bibitem{O06}
Barrett O'Neill.
\newblock {\em Elementary Differential Geometry (Second Edition)}.
\newblock Academic Press, Boston, second edition edition, 2006.

\bibitem{PL08}
Michael~L. Parks, Richard~B. Lehoucq, Steven~J. Plimpton, and Stewart~A.
  Silling.
\newblock Implementing peridynamics within a molecular dynamics code.
\newblock {\em Comput. Phys. Commun.}, 179:777--783, 2008.

\bibitem{PS22}
Marco Pasetto, Zhaoxiang Shen, Marta D'Elia, Xiaochuan Tian, Nathaniel Trask,
  and David Kamensky.
\newblock Efficient optimization-based quadrature for variational
  discretization of nonlocal problems.
\newblock {\em Comput. Methods Appl. Mech. Engrg.}, 396:115--104, 2022.

\bibitem{image:peyre}
Gabriel Peyr{\'e}, S{\'e}bastien Bougleux, and Laurent Cohen.
\newblock Non-local regularization of inverse problems.
\newblock In {\em European Conference on Computer Vision}, pages 57--68.
  Springer, 2008.

\bibitem{cmap:2007}
Mathieu Poudret, Agn{\`e}s Arnould, Yves Bertrand, and Pascal Lienhardt.
\newblock Cartes combinatoires ouvertes.
\newblock {\em BMC Res. Notes}, 1, 2007.

\bibitem{PS20}
Naveen Prakash and Ross~J Stewart.
\newblock A multi-threaded method to assemble a sparse stiffness matrix for
  quasi-static solutions of linearized bond-based peridynamics.
\newblock {\em J. Peridyn. Nonlocal Model.}, 2020.

\bibitem{machine:schmidt}
Stephan Schmidt, Caslav Ilic, Volker Schulz, and Nicolas~R Gauger.
\newblock Three-dimensional large-scale aerodynamic shape optimization based on
  shape calculus.
\newblock {\em AIAA J.}, 51(11):2615--2627, 2013.

\bibitem{peridynamics:SILLING}
S.A. Silling.
\newblock Reformulation of elasticity theory for discontinuities and long-range
  forces.
\newblock {\em J. Mech. Phys. Solids.}, 48(1):175--209, 2000.

\bibitem{meshfree:silling2005meshfree}
Stewart~A. Silling and Ebrahim Askari.
\newblock A meshfree method based on the peridynamic model of solid mechanics.
\newblock {\em Comput. Struct.}, 83(17-18):1526--1535, 2005.

\bibitem{neumann:diffusion}
Yunzhe Tao, Xiaochuan Tian, and Qiang Du.
\newblock {Nonlocal diffusion and peridynamic models with Neumann type
  constraints and their numerical approximations}.
\newblock {\em Appl. Math. Comput.}, 305(C):282--298, 2017.

\bibitem{difference:tian2017}
Hao Tian, Lili Ju, and Qiang Du.
\newblock A conservative nonlocal convection-diffusion model and asymptotically
  compatible finite difference discretization.
\newblock {\em Comput. Methods Appl. Mech. Engrg.}, 320:46--67, 2017.

\bibitem{collection:tian2013efficient}
Hao Tian, Hong Wang, and Wenqia Wang.
\newblock An efficient collocation method for a non-local diffusion model.
\newblock {\em Int. J. Numer. Anal. Model.}, 10(4), 2013.

\bibitem{difference:tian2013}
Xiaochuan Tian and Qiang Du.
\newblock Analysis and comparison of different approximations to nonlocal
  diffusion and linear peridynamic equations.
\newblock {\em SIAM J. Numer. Anal.}, 51(6):3458--3482, 2013.

\bibitem{TD14}
Xiaochuan Tian and Qiang Du.
\newblock Asymptotically compatible schemes and applications to robust
  discretization of nonlocal models.
\newblock {\em SIAM J. Numer. Anal.}, 52:1641--1665, 2014.

\bibitem{TD20}
Xiaochuan Tian and Qiang Du.
\newblock Asymptotically compatible schemes for robust discretization of
  parametrized problems with applications to nonlocal models.
\newblock {\em SIAM Rev.}, 62(1):199–227, 2020.

\bibitem{TE19}
Xiaochuan Tian and Björn Engquist.
\newblock Fast algorithm for computing nonlocal operators with finite
  interaction distance.
\newblock {\em Commun. Math. Sci.}, 17(6):1653--1670, 2019.

\bibitem{Vollmann2019}
Christian Vollmann.
\newblock {\em Nonlocal models with truncated interaction kernels - analysis,
  finite element methods and shape optimization}.
\newblock doctoralthesis, Universit{\"a}t Trier, 2019.

\bibitem{VS19}
Christian Vollmann and Volker Schulz.
\newblock Exploiting multilevel toeplitz structures in high dimensional
  nonlocal diffusion.
\newblock {\em Comput. Vis. Sci.}, pages 29--46, 2019.

\bibitem{collection:wang2017fast}
Che Wang and Hong Wang.
\newblock A fast collocation method for a variable-coefficient nonlocal
  diffusion model.
\newblock {\em J. Comput. Phys.}, 330:114--126, 2017.

\bibitem{WW17}
Che Wang and Hong Wang.
\newblock A fast collocation method for a variable-coefficient nonlocal
  diffusion model.
\newblock {\em J. Comput. Phys}, 330:114--126, 2017.

\bibitem{WT12}
Hong Wang and Hao Tian.
\newblock A fast galerkin method with efficient matrix assembly and storage for
  a peridynamic model.
\newblock {\em J. Comput. Phys.}, 231(23):7730--7738, 2012.

\bibitem{WT14}
Hong Wang and Hao Tian.
\newblock A fast and faithful collocation method with efficient matrix assembly
  for a two-dimensional nonlocal diffusion model.
\newblock {\em Comput. Methods Appl. Mech. Engrg.}, 273:19--36, 2014.

\bibitem{neumann:osti_1769929}
Huaiqian You, Xin~Yang Lu, Nathaniel~Albert Trask, and Yue Yu.
\newblock An asymptotically compatible approach for neumann-type boundary
  condition on nonlocal problems.
\newblock {\em Mathematical Modelling and Numerical Analysis}, 55, 2 2021.

\bibitem{element:zhang2016quadrature}
Xiaoping Zhang, Max Gunzburger, and Lili Ju.
\newblock Quadrature rules for finite element approximations of 1d nonlocal
  problems.
\newblock {\em J. Comput. Phys.}, 310:213--236, 2016.

\bibitem{collection:zhang2018accurate}
Xiaoping Zhang, Jiming Wu, and Lili Ju.
\newblock An accurate and asymptotically compatible collocation scheme for
  nonlocal diffusion problems.
\newblock {\em Appl. Numer. Math.}, 133:52--68, 2018.

\end{thebibliography}
